\documentclass[reqno,12pt]{amsart}
\usepackage{geometry}
\usepackage{amsmath,amssymb,mathrsfs,mathabx,color,paralist,graphicx}
\usepackage{float} 
\usepackage{times}
\usepackage[colorlinks,
linkcolor=red,
anchorcolor=green,
citecolor=blue, 
]{hyperref}

\makeatletter
\def\setliststart#1{\setcounter{\@listctr}{#1}%
  \addtocounter{\@listctr}{-1}}
\makeatother

\makeatletter
\@addtoreset{figure}{section}
\makeatother

\usepackage{calc}

\font\uj=cmssbx10
\newtheorem{The}{Theorem}[section]
\newtheorem{Cor}[The]{Corollary}
\newtheorem{Lem}[The]{Lemma}
\newtheorem{Pro}[The]{Proposition}
\theoremstyle{definition}
\newtheorem{Def}[The]{Definition}

\theoremstyle{remark}
\newtheorem*{Rem}{Remark}
\numberwithin{equation}{section}
\def\R{\mathbb R}
\def\Z{\mathbb Z}
\def\N{\mathbb N}
\def\T{\mathbb T}
\def\Q{\mathbb Q}
\def\C{\mathbb C}

\def\bG{\mathbf G}
\def\bL{\mathbf L}
\def\bq{\mathbf q}
\def\bp{\mathbf p}

\def\cM{\mathcal M}
\def\cA{\mathcal A}
\def\cN{\mathcal N}
\def\cG{\mathcal G}
\def\cC{\mathcal C}

\def\fS{\mathfrak S}
\def\fB{\mathfrak B}
\def\fM{\mathfrak M}

\def\sL{\mathscr L}
\def\sK{\mathscr K}
\def\sG{\mathscr G}
\def\sC{\mathscr C}
\def\sD{\mathscr D}
\def\sR{\mathscr R}

\def\rN{\mathrm N}
\def\rU{\mathrm U}

\def\bfo{0}
\def\vep{\varepsilon}
\def\rro{\mathbf{r_0}}

\begin{document}
\title{Gevrey genericity of Arnold diffusion in \emph{a priori} unstable Hamiltonian systems}
\author{Qinbo Chen $^{\dag, \ddag}$}
\address{$^\dag$ Department of Mathematics, Nanjing University, Nanjing 210093, China}
\address{$^\ddag$ Morningside Center of Mathematics, Academy of Mathematics and Systems Science, Chinese Academy of Sciences, Beijing 100190, China}
\email{chenqb@amss.ac.cn}

\author{Chong-Qing Cheng *}
\address{* Department of Mathematics, Nanjing University, Nanjing 210093, China}
\email{chengcq@nju.edu.cn}

\begin{abstract}
It is well known  that under generic  $C^r$ smooth  perturbations,  the phenomenon of global instability, known as Arnold diffusion, exists in  \emph{a priori} unstable Hamiltonian systems. In this paper, by using variational methods, we will prove that under generic Gevrey smooth perturbations,  Arnold diffusion still exists  in the \emph{a priori} unstable Hamiltonian systems of two and a half degrees of freedom.
\end{abstract}
\maketitle
\setcounter{tocdepth}{1}
\section{Introduction}\label{introduction}
Throughout this paper, we denote by $\T^n\times\R^n$ the cotangent bundle $T^*\T^n$ of the torus $\T^n$ with $\T=\R/\Z$, and endow  $\T^n\times\R^n$ with its usual coordinates $(q,p)$ where $q=(q_1,\cdots,q_n)$ and $p=(p_1,\cdots,p_n)$.  We also endow the phase space with its canonical symplectic form $\Omega=\sum_{i=1}^ndq_i\wedge dp_i$. A Hamiltonian system is usually a dynamical system governed by the following Hamilton's equations  
$$\dot{q}=\frac{\partial H}{\partial p}, \quad \dot{p}=-\frac{\partial H}{\partial q},\quad  (q,p)\in\T^n\times\R^n$$
where $H(q, p, t)$ is a Hamiltonian function and the dependence on the time $t$ is $1$-periodic, so $t\in\T$. 

The goal of this paper is to present  global instability in a class of Hamiltonians. The problem of the influence of small perturbations on an integrable Hamiltonian system was considered by Poincar\'e to
be \emph{the fundamental problem of Hamiltonian dynamics}. It is customary to consider a nearly-integrable Hamiltonian system of the form  $H=H_0(p)+\vep H_1(q, p, t)$. Notice that for $\vep=0$ such systems do not admit any instability phenomenon. For $0<\vep\ll 1$, the celebrated KAM theory asserts that a set of nearly full measure in the phase space consists of invariant tori carrying quasi-periodic motions, and the oscillation of the action variables $p$ on each KAM torus is at most of order $\sqrt{\vep}$. For $n\geq2$, the complement of the set of the union of KAM tori is connected, so a natural question that arises is whether it is possible to find large evolution of order $1$.  In the celebrated paper  \cite{Ar1964}, Arnold first proposed an example of a nearly-integrable Hamiltonian system with two and a half degrees of freedom, which admits trajectories whose action variables have large oscillation. Moreover, he also conjectured that such an instability phenomenon  occurs in generic nearly-integrable systems. This is known as the Arnold diffusion conjecture, and it has been investigated extensively since then. 

The mechanism of Arnold's original example is based on the existence of a normally hyperbolic invariant cylinder (NHIC) foliated by a family of hyperbolic invariant tori or whiskered tori. The unstable manifold of one torus intersects transversally the stable manifold of another nearby torus. These tori constitute a transition chain along which diffusion takes place. By Nekhoroshev theory, it must be extremely slow. This mechanism has inspired a large number of studies to the Hamiltonians  possessing certain hyperbolic geometric structures. In the literature, such a system is referred to as ``\emph{a priori} unstable", to be distinguished from a nearly-integrable system (i.e. ``\emph{a priori} stable"). There have been many works devoted to the \emph{a priori} unstable systems based on Arnold's  geometric mechanism, and most of them have tried to find transition chains in more general cases \cite{BBB2003,BB2002,Cr2003, FP2001,KL2008,LM2005,Zh2011}, \emph{etc}. However, for a general \emph{a priori} unstable system the transition chain cannot be formed by a continuous family of tori but a Cantorian family, and the size of gaps between the persisting tori could be larger than the size of the intersections of the stable and unstable manifolds. This is known as the \emph{large gap problem}. 

In the last two decades there have been several methods to overcome the large gap problem. Among them there are mainly two methods concerning the genericity of instability: variational  methods and geometric methods.  
 The first attempt to study Arnold's original example in variational viewpoint is by U. Bessi \cite{Bessi1996}. Essential progress has also been made by J. N. Mather. In the celebrated paper \cite{Ma1993}, he developed a powerful variational tool to study the global instability in the framework of convex Lagrangian systems.  In an unpublished manuscript \cite{Ma1995}, he further showed the existence of orbits with unbounded energy in perturbations of a  geodesic flow on  $\T^2$ by a generic time-periodic potential. Based on Mather's variational mechanism, the authors of \cite{CY2004} constructed  diffusing orbits  and proved the $C^r$-genericity ($r$ is finite and suitably large) of Arnold diffusion for the \emph{a priori} unstable systems with two and a half degrees of freedom. On the other hand, several authors have used  geometric methods, which also  apply to  Hamiltonians that are not necessarily convex, to obtain Arnold diffusion.  More precisely, the authors in \cite{DLS2000,DLS2003,DLS2006}  defined the so-called  scattering map  which accounts for the outer dynamics along homoclinic orbits, and  overcame the large gap problem by incorporating in the transition chain new invariant objects, like secondary tori and the stable and unstable manifolds of lower dimensional tori; In \cite{Tr2002}, the author geometrically defined the so-called separatrix map near the normally hyperbolic invariant cylinder,  then he showed in \cite{Tr2004} the existence of diffusion by making full use of the dynamics of this map, and even estimated the optimal diffusion speed of order $\vep/|\log \vep|$ (see also \cite{BBB2003}). Moreover, for the case of \emph{a priori} unstable Hamiltonians with higher degrees of freedom, similar results have also been obtained by variational or geometric methods in \cite{Be2008,CY2009,DLS2016,DT2018,GLS2019,LC2010,Tr2012}.

 The \emph{a priori} stable case poses a new difficulty:  the presence of    multiple resonances. In the paper \cite{Ma2004} (see also \cite{Ma2012}), Mather first made an announcement for systems with two degrees of freedom in the time-periodic case or with three degrees of freedom in the autonomous case, under a series of  cusp-residual conditions. Hence the diffusion problem in this situation was thought to possess only cusp-residual genericity. The complete proof for the autonomous systems with three degrees of freedom appeared in the preprint \cite{Ch2012}, and the main ingredients have been published in the recent works \cite{CZh2016,Ch2017Uniform,Ch2017,Ch2018}. Indeed, the main difficulty in this case arises from the dynamics around  strong double resonances. It is because away from double resonances, one could apply normal form theory to construct NHICs with a length independent of $\vep$, along which the local instability  can be obtained  as in the \emph{a priori} unstable case  \cite{Be2010Large,BKZ2016}.  To solve the problem of double resonance, the paper \cite{Ch2017}  presented a new variational mechanism to switch from one resonance to another, which eventually proved  the cusp-residual genericity of diffusion in the $C^r$ smooth category \cite{Ch2018}. Moreover, we  mention that similar results on diffusion have also been  obtained, by using variational methods, in the paper \cite{KZ2015} and the preprint \cite{KZ2013} for systems with 2.5 degrees of freedom.  Also, we refer the reader to the preprints \cite{Marco2016Arnold,Marco2016chains,GM2017} for systems with 3 degrees of freedom by using the geometric tools. As for the case
  of arbitrarily higher degrees of freedom, we refer the reader to the  preprint \cite{CX2015} and the announcement \cite{KZ2014Announ}. Anyway, there have been many other works related to the problem of Arnold diffusion but we cannot list all of them, see \cite{BCV2001,BT1999,DH2011,GKZ2016,GT2008,GR2007,KZ2014,ZC2014}, \emph{etc}.

To the author's knowledge, the genericity of Arnold diffusion is by now quite well understood in the $C^r$ smooth category, not yet in the analytic category, or the Gevrey smooth category \cite{Gev1918}. The present paper is interested in  whether the phenomenon of large evolution  exists generically in the Gevrey smooth Hamiltonians. Given  $\alpha\geq 1$, a Gevrey-$\alpha$  function is an ultra-differentiable function whose $k$-th order partial derivatives are bounded by  $O(M^{-|k|}k!^{\alpha})$. For the case $\alpha=1$,  it is exactly a real analytic function. Hence, the Gevrey class is intermediate between  the $C^\infty$ class and the real analytic class. Besides, a key point for the Gevrey class is that it allows the existence of a function with compact support (i.e. bump function). But no analytic function has compact support.

To consider the Arnold diffusion problem in the Gevrey topology, we would adopt the Gevrey norm  introduced by  Marco and   Sauzin  in \cite{MS2003} during a collaboration with  Herman (see Definition \ref{def of gev}). Apart from the theory of PDE where it has been widely used, the Gevrey class is also studied in the field of Dynamical Systems. For example, we refer to \cite{Bo2011,Bo2013,BF2017,BM2011,LDG2017,Po2004}, \emph{etc} for the stability theory, such as KAM theory and  Nekhoroshev theory.  We also refer the reader to \cite{BK2005,LMS2015,MS2004,Wa2015,FaSa2018}, \emph{etc} for some relevant results on instability. All these studies make us believe that one can also consider the genericity problem of diffusion in the Gevrey case.

Therefore, in this paper we start by considering the \emph{a priori} unstable, Gevrey-$\alpha$ ($\alpha>1$)  Hamiltonian systems of two and a half degrees of freedom.  The case  $\alpha=1$ (i.e. the analytic genericity) is more complicated and  has not been fully studied.  Here we only mention  a recent work \cite{GLS2019} which proposes a general geometric mechanism that might be useful for analytic genericity. 
In the same spirit as in \cite{GLS2019}, the paper \cite{GT2017}  gives  models where the analytic genericity can be achieved for \emph{a priori} chaotic symplectic maps, provided that  the scattering map has no monodromy and is globally defined on the NHIC.

Before stating our main results , we review the concept of Gevrey function and some standard facts.

\begin{Def}[\emph{Gevrey function} \cite{MS2003}]\label{def of gev}
Let $\alpha\geq 1, \bL>0$ and  $K$ be a $n$-dimensional compact domain. A real-valued $C^\infty$ function $f(x)$ defined on $K$ is said to be Gevrey-($\alpha,\bL$) if
$$ \| f\|_{\alpha,\bL}:=\sum_{k\in\N^n}\frac{\bL^{|k|\alpha}}{(k!)^\alpha}\|\partial^kf\|_{C^0(K)}<+\infty,$$
with the standard multi-index notation $k=(k_1,\cdots,k_n)\in\N^n$, $|k|=k_1+\cdots+k_n$, $k!=k_1!\cdots k_n!$ and
$\partial^k=\partial^{k_1}_{x_1}\cdots\partial^{k_n}_{x_n}$.
\end{Def}

Let $\bG^{\alpha,\bL}(K):=\{ f\in C^\infty(K)~:~\|f\|_{\alpha,\bL}<+\infty\}$. The space $\bG^{\alpha,\bL}(K)$ endowed with the norm $\|\cdot\|_{\alpha,\bL}$ is a Banach space. Sometimes we also write $\bG^{\alpha}(K):=\bigcup_{\bL>0}\bG^{\alpha,\bL}(K)$. In particular, for $K\subset \R^n$ and $\alpha=1$,  $\bG^{1}(K)$ is exactly the space of real analytic functions on $K$: any function $f\in \bG^{1,\bL}(K)$
is real analytic in $K$ and admits an analytic extension in the complex domain $\{z\in\C^n:\textup{dist}(z,K)< \bL\}$. Conversely, for any real analytic function $f$ in $K$, there exists $\bL>0$ such that $f\in \bG^{1,\bL}(K)$. However, for $\alpha>1$, $\bG^{\alpha,\bL}(K)$ admits non-analytic functions. Therefore, the Gevrey-smooth category  is intermediate between the $C^\infty$ category and the analytic category.

 Gevrey class has the following useful properties which have been already proved in \cite{MS2003}:
\begin{enumerate}[\rm(G1)]
\item\label{algebra norm}  The norm $\|\cdot\|_{\alpha,\bL}$  is an algebra norm, namely $\|fg\|_{\alpha,\bL}\leq\|f\|_{\alpha,\bL}\|g\|_{\alpha,\bL}$.
\item\label{derivative Gevrey}Suppose $0<\lambda<\bL$ and $f\in\bG^{\alpha,\bL}(K)$, then all partial derivatives of $f$ belong to $\bG^{\alpha,\bL-\lambda}(K)$ and
$\sum\limits_{k\in\N^n,|k|=l} \|\partial^kf\|_{\alpha,\bL-\lambda}\leq l!^\alpha\lambda^{-l\alpha}\|f\|_{\alpha,\bL}.$
\item\label{composition}Let $f\in\bG^{\alpha,\bL}(K_m)$  where $K_m$ is a $m$-dimensional domain and let $g=(g_1,\cdots,g_m)$ be a mapping whose component $g_i\in\bG^{\alpha,\bL_1}(K_n)$ . If $g(K_n)\subset K_m$ and
$\|g_i\|_{\alpha,\bL_1}-\|g_i\|_{C^0(K_n)}\leq\bL^\alpha/n^{\alpha-1}$ for all $1\leq i\leq m,$
then $f\circ g\in\bG^{\alpha,\bL_1}(K_n)$ and $\|f\circ g\|_{\alpha,\bL_1}\leq\|f\|_{\alpha,\bL}$.
\end{enumerate}

\subsection{Setup and main result}
The current paper will mainly focus on the convex Hamiltonians of two and a half degrees of freedom. As we will see later, all discussions will be restricted 
 on a compact domain in $\T^2\times\R^2\times\T$, so we fix, once and for all, a constant $R>1$ and a compact set
$$\sD_R=\T^2\times\bar{B}_R(0)\times\T,$$ 
where $B_R(0)\subset\R^2$ is an open ball of radius $R$ centered at 0 and $\bar{B}_R(0)$ is the closure.
By Definition \ref{def of gev}, the space $\bG^{\alpha,\bL}(\sD_R)$ consists of all real-valued smooth functions $f(q,p,t)$ satisfying
\begin{equation}\label{gevrey norm}
 \| f\|_{\alpha,\bL}=\sum_{k\in\N^{5}}\frac{\bL^{|k|\alpha}}{k!^\alpha}\|\partial^kf\|_{C^0(\sD_R)}<+\infty,
\end{equation}
Let $C^\omega_d(\sD_R)$ be the space of all real-valued analytic functions on $\sD_R$, admitting an analytic extension in the complex domain   $\{(q,p,t)\in(\C/\Z)^2\times\C^2\times(\C/\Z):\|\textup{Im}q\|_{\infty}< d,~\textup{dist}(p,\bar{B}_R(0)) < d, |\textup{Im}t|<d\}$.
Set $C^\omega(\sD_R)=\bigcup_{d>0} C^\omega_d(\sD_R)$, it is well known that
\begin{enumerate}[(i)]
  \item For $\alpha\geq 1$, $\bL>0$  and any $d>\bL^\alpha$, 
$$C^\omega_d(\sD_R)\subset \bG^{\alpha,\bL}(\sD_R)\subset C^\infty(\sD_R),\quad C^\omega(\sD_R)\subset \bG^{\alpha}(\sD_R)\subset C^\infty(\sD_R).$$
\item $C^\omega(\sD_R)=\bG^{1}(\sD_R).$
\end{enumerate}

Now, we introduce the \emph{a priori} unstable Hamiltonian model considered in this paper
 and state the main assumptions. Let $q=(q_1, q_2)\in\T^2$ and $p=(p_1,p_2)\in\R^2$. 
 We consider a time-periodic and $C^r  (r>2)$ smooth Hamiltonian of the form:
 \begin{equation}\label{hamiltonian}
\begin{aligned}
 H(q, p, t)=H_0(q, p)+H_1(q, p, t),\quad \text{where}\quad  H_0(q, p)=h_1(p_1)+h_2(q_2,p_2).
 \end{aligned}
\end{equation}
Here, the term $H_1$ is a small perturbation which is periodic of period 1 in $t$. Our main assumptions on $H_0$ are the following:
\begin{enumerate}[\bf(H1)]
         \item \textbf{Convexity and superlinearity}: for each $q\in\T^2$, the Hessian $\partial_{pp}H_0(q,p)$ is positive definite, and $\lim_{\|p\|\rightarrow +\infty} H_0(q,p)/\|p\|=+\infty.$
         \item \textbf{A priori hyperbolicity}: the Hamiltonian flow $\Phi^t_{h_2}$, determined by $h_2$, has a  hyperbolic fixed point  $(q_2,p_2)=(x^*,y^*)$.  Moreover,  the function $h_2(q_2,y^*): \T \rightarrow \R$ attains its unique maximum at $q_2=x^*$. Without loss of generality, we can assume $(x^*,y^*)=(0,0)$.
    \end{enumerate}

 A prototype example of such a system is the coupling of a rotator and a pendulum 
$$H=\frac{p_1^2}{2}+\frac{p_2^2}{2}+(\cos2\pi q_2-1)+H_1(q, p, t),$$ 
which has been considered many times in the literature. Keep this example in mind will help the reader better understand our result and method. As we will see later, the above assumptions {\bf(H1)--(H2)} are in the same spirit as in \cite{CY2004} while our main result and  approach have some differences.

Let $\fB^{\bL}_{\vep,R}=\{ H_1\in C^\infty(\sD_R) : \|H_1\|_{\alpha,\bL}<\vep\}$ $\subset\bG^{\alpha,\bL}(\sD_R)$ denote the open ball of radius $\vep$ centered at the origin with respect to the norm $\|\cdot\|_{\alpha,\bL}$.
\begin{The}\label{main theorem}
Let $\alpha>1, R>1$ and  assume that $H_0$ in \eqref{hamiltonian} is of class $C^r$ $( r>2)$, then there exists a positive constant $\bL_0=\bL_0(H_0,\alpha,R)$ such that, for each $\bL\in(0,\bL_0]$ and  a sequence of open balls $B_s(y_1),\cdots, B_s(y_k)\subset\R^2$,  of radius $s$ centered at $y_\ell\in[-R+1,R-1]\times\{\bfo\}\subset\R^2$, $\ell=1,\cdots, k$, we have:

 there exist a positive number $\vep_0=\vep_0(H_0,\alpha,R,s,\bL)$  and an open and dense subset $\fS^{\bL}_{\vep_0,R}\subset\fB^{\bL}_{\vep_0,R}$  such that for each perturbation $H_1\in\fS^{\bL}_{\vep_0,R}$, the system $H=H_0+H_1$ has a trajectory $(q(t),p(t))$  whose action variables $p(t)$ pass through the ball $B_s(y_\ell)$ at the time $t=t_\ell$, where $t_1<t_2<\cdots<t_k$.
\end{The}

\begin{Rem}
Just as J. N. Mather did in \cite{Ma2004,Ma2012},  the smoothness of the unperturbed Hamiltonian $H_0$ could differ from that of the perturbation term $H_1$.  Notice that $\bL_0$ can not be an arbitrary constant, the reason is that our approach needs to adopt the Gevrey approximation (see Theorem \ref{Gevrey approx}).
\end{Rem}

\begin{Rem}[Autonomous case]
Recall that Mather's cohomology equivalence is trivial for an autonomous system (cf. \cite{Be2002}). The problem is that, unlike  the time-periodic case, there is no canonical global transverse section of the flow in an autonomous system. In \cite{LC2010}, this difficulty was overcome  by taking local transverse sections, which could generalize Mather's cohomology equivalence. Thus we believe that the Gevrey genericity is still valid for the \emph{a priori} unstable autonomous Hamiltonians. However, in this paper,
 we only consider the non-autonomous case.
\end{Rem}

The perturbation technique used in the current paper can also prove the genericity in the sense of Ma\~n\'e, which means that the diffusion is still a typical phenomenon when $H_0$ is perturbed by potential functions. More precisely, let $\mathbf{B}^{\bL}_{\vep}$ $\subset\bG^{\alpha,\bL}(\T^2\times\T)$ denote the open ball of radius $\vep$ centered at the origin with respect to the norm $\|\cdot\|_{\alpha,\bL}$, we have
\begin{The}\label{main thm2}
Under the same assumptions as in Theorem \ref{main theorem},  there exists  $\bL_0=\bL_0(H_0,\alpha,R)>0$ such that, for each $\bL\in(0,\bL_0]$ and  a sequence of open balls $B_s(y_1),\cdots, B_s(y_k)\subset\R^2$,  of radius $s$ centered at $y_\ell\in[-R+1,R-1]\times\{\bfo\}\subset\R^2$, $\ell=1,\cdots, k$, we have: 

there exist a positive number $\vep_0=\vep_0(H_0,\alpha,R,s,\bL)$  and an open and dense subset $\mathbf{S}^{\bL}_{\vep_0}\subset\mathbf{B}^{\bL}_{\vep_0}$  such that for each potential perturbation $H_1\in\mathbf{S}^{\bL}_{\vep_0}$, the system $H=H_0+H_1$  has a trajectory $(q(t),p(t))$  whose action variables $p(t)$ pass through the ball $B_s(y_\ell)$ at the time $t=t_\ell$, where $t_1<t_2<\cdots<t_k$. 
\end{The}

\subsection{Outline of this paper}
This paper mainly adopts variational methods to construct diffusing orbits, so it requires us to transform into Lagrangian formalism. We still denote by $\T^2\times\R^2$ the tangent bundle $T\T^2$, and  endow  $\T^2\times\R^2$ with its usual coordinates $(q, v)$. The Lagrangian $L:\T^2\times\R^2\times\T \rightarrow \R$ associated to $H$ is defined as follows:
\begin{equation}\label{lagrangian}
  L(q,v,t):=\max_{p}\{\langle p, v\rangle-H(q,p,t)\}=L_0(q,v)+L_1(q,v,t),\quad L_0=l_1(v_1)+l_2(q_2,v_2).
\end{equation}

In our proofs, we will apply results in Mather theory, where the Lagrangian is required to  satisfy the Tonelli conditions (see Section \ref{sec_Preliminaries}): the fiberwise Hessian is positive definite, fiberwise superlinear, and the completeness of the Euler-Lagrange flow. 

In fact, without affecting our analysis, we can always reduce to the Tonelli case. For our Lagrangian $L=L_0+L_1$ in  \eqref{lagrangian}, it is clear that the unperturbed part $L_0$ is a Tonelli Lagrangian as a result of hypothesis \textbf{(H1)}. Now, we turn to the small perturbation term $L_1$. As we will see later, only the information on a compact region is needed in our proofs,  then it will not affect the study of Arnold diffusion if one modifies the perturbation function $L_1$ outside the compact set. For example, 
one can introduce a  new function $\widetilde{L}_1$ which has compact support, and is identically equal to $L_1$ on a compact set $\{\|v\|_q\leq K\}$. In terms of this modification, we then introduce a new Lagrangian  $\widetilde{L}:=L_0+\widetilde{L}_1$.
Observe that the modified Lagrangian $\widetilde{L}$  satisfies the Tonelli conditions since the perturbation term $\widetilde{L}_1$ is small enough and has compact support. Also, it is quite clear that
$\widetilde{L}$ and $L$ generate the same Euler-Lagrange flow  when restricted on the compact region $\{\|v\|_q\leq K\}$. 
 Such a modification is elementary, see for instance \cite{Ma2004}.

Therefore, in what follows, we can always assume, without loss of generality, that our  Lagrangian \eqref{lagrangian} satisfies the Tonelli conditions. 
Then, through the Legendre transformation 
\begin{equation}\label{Legendre_tran}
	\begin{aligned}
		\sL: T^*\T^2\times\T &\rightarrow T\T^2\times\T,\\
(q,p,t) &\mapsto(q,\frac{\partial H}{\partial p}(q,p,t),t),
	\end{aligned}
\end{equation}
 we can write
\begin{equation*}
  L(q,v,t)=\langle \pi_p\circ\sL^{-1}(q,v,t),v\rangle-H\circ\sL^{-1}(q,v,t),
\end{equation*}
where $\pi_p$ denotes the projection $(q,p,t)\mapsto p$. Then, the Hamilton's equations $\dot{q}=\frac{\partial H}{\partial p}$, $\dot{p}=-\frac{\partial H}{\partial q}$ is equivalent to the Euler-Lagrange equation
\begin{equation*}
 \frac{d}{dt}\bigg(\frac{\partial L}{\partial v}\bigg)-\frac{\partial L}{\partial q}=0.
\end{equation*}
Throughout this paper, we use $\phi^t_L$ to denote the Euler-Lagrange flow determined by $L$ and  $\Phi^t_H$ to denote the Hamiltonian flow determined by $H$. 

The Fenchel inequality and  hypothesis (\textbf{H2}) together give rise to
\begin{equation*}
  h_2(q_2,\bfo)+l_2(q_2,v_2)\geq 0,\quad  h_2(\bfo,\bfo)+l_2(\bfo,\bfo)=0,\quad  (q_2,v_2)\in T\T.
\end{equation*}
Since $q_2=\bfo$ (mod 1) is the unique maximum point of the function $h_2(\cdot,\bfo):\T\to\R$, one gets
\begin{equation}\label{weiyimin}
  l_2(\bfo,\bfo)=-h_2(\bfo,\bfo)\leq-h_2(q_2,\bfo)\leq l_2(q_2,v_2),\quad (q_2,v_2)\in T\T.
\end{equation}
Then the point  $(q_2,v_2)=(\bfo,\bfo)$ is the unique minimum point of the function $l_2$ as a consequence of the strict convexity.  Also, $(\bfo,\bfo)$ is a hyperbolic fixed point for the Euler-Lagrange flow $\phi^t_{l_2}$.

Compared with the variational proofs of $C^r$-genericity  in \cite{CY2004,CY2009}, the method in this paper contains some new techniques. Indeed, the strategy used in \cite{CY2004,CY2009}, which perturbs the generating functions to  create genericity, seems not applicable to the Gevrey genericity. The main difficulty arises from the fact that, when we estimate the Gevrey smoothness of a Hamiltonian flow, we cannot avoid the decrease of Gevrey coefficient $\bL$ during the switch from a generating function to its corresponding Hamiltonian, or the switch from a Lagrangian to its associated Hamiltonian (see property (G\ref{derivative Gevrey}) above). Thus in this paper, inspired by the ideas in \cite{Ch2017}, we decide to directly perturb a Hamiltonian by  potential functions, one advantage of this approach is that the Lagrangian associated to the perturbed Hamiltonian $H+V(q,t)$ is exactly $L-V(q,t)$. To this end,  some  quantitative  estimations are required, such as the Gevrey approximation and the corresponding inverse function theorem. It is also worth mentioning  that one can establish the genericity not only in the usual sense but also in the sense of Ma\~n\'e. Besides, we also believe that our results could  be obtained by geometric tools, such as the scattering maps developed in \cite{DLS2006,DH2009,DLS2008,DLS2016}, or  the separatrix maps in  \cite{Tr2004,Tr2012,DT2018}.

In our variational proof of genericity, the modulus of continuity of barrier functions is crucial. To implement this argument, the work \cite{CY2004} introduced the following parameterization technique: fixing an invariant curve $\Gamma_0$ on the NHIC,  for any other invariant curve $\Gamma_\sigma$ on the NHIC, we parameterize it by  $\sigma$, the area between the two curves. Then it can be shown that  $\Gamma_\sigma$ is H\"older continuous with respect to $\sigma$ in the $C^0$ topology. However, by taking advantage of the tools in weak KAM theory, now we can show that this ``area" parameter $\sigma$ is exactly the cohomology class (See section \ref{sub holder}). This will help us simplify the  proof.

The structure of this paper is as follows.  In Section \ref{Mathertheory}, we review  some standard results, related to Arnold diffusion problem, in Mather theory.  Section \ref{sec EWS} discusses the elementary weak KAM solutions, and a special ``barrier function" whose  minimal points correspond to heteroclinic orbits.  In Section \ref{sec local and global}, we introduce the concept of generalized transition chain and then give the variational mechanism of constructing diffusing orbits along this chain. In Section \ref{some properties Gev}, we present some properties of Gevrey functions which are necessary for our proofs.  Section \ref{sec proof main} is the main part of this paper, and applies the tools exposed in previous sections to study the Gevrey smooth systems. First, we generalize the genericity of uniquely minimal measure in the Gevrey or analytic topology. Second, we obtain certain regularity of the elementary weak KAM solutions and show how to choose suitable Gevrey space.  Finally, by proving  total disconnectedness for the minimal sets of barrier functions, we establish the genericity of generalized transition chain along which the global instability occurs, this therefore completes the proofs of Theorem \ref{main theorem} and Theorem \ref{main thm2}.

We would like to thank the anonymous referees for their insightful comments and valuable suggestions on improving our results.

\section{Preliminaries: Mather theory}\label{Mathertheory}\label{sec_Preliminaries}
In this section, we recall some standard results in Mather theory which are necessary for the purpose of our study, the main references are Mather's original papers \cite{Ma1991, Ma1993}. Let $M$ be a connected and compact smooth manifold without boundary, equipped with a smooth Riemannian metric $g$. Let $TM$ denote the tangent bundle, a point of $TM$ will be denoted by $(q,v)$ with $q\in M$ and $v\in T_qM$. We shall  denote by $\|\cdot\|_q$ the norm induced by $g$ on the fiber $T_qM$.  A time-periodic $C^2$ function $L=L(q, v, t):TM\times\T\rightarrow \R$  is called a \emph{Tonelli Lagrangian} if it satisfies:
\begin{enumerate}[\rm(1)]
  \item \emph{Convexity}: $L$ is strictly convex in each fiber, i.e., the second partial derivative $\partial^2 L/\partial v^2(q, v, t)$ is positive definite, as a quadratic form, for each $(q, t) \in M\times\T$;
  \item \emph{Superlinear growth}: $L$ is superlinear in each fiber, i.e. for each $(q,t)\in M\times\T$,
  $$\lim_{\|v\|_q\rightarrow +\infty}\frac{L(q,v,t)}{\|v\|_q}=+\infty.$$
  \item \emph{Completeness}: All solutions of the Euler-Lagrange equation are well defined for all $t\in\R$.
\end{enumerate}

Let $I=[a,b]$ be an interval and $\gamma:I\rightarrow M$ be any absolutely continuous curve. Given a cohomology class $c\in H^1(M,\R)$,  we choose and fix a closed 1-form $\eta_c$ with $[\eta_c]=c$. Denote by
$$A_c(\gamma):=\int_a^b L(d\gamma(t),t)-\eta_c(d\gamma(t))\, dt$$ 
the action of $L-\eta_c$ along $\gamma$,  where $d\gamma(t)=(\gamma(t),\dot\gamma(t))$. A curve $\gamma: I\rightarrow M$ is called  \emph{$c$-minimal} if
$$A_c(\gamma)=\min_{\substack{\xi(a)=\gamma(a),\xi(b)=\gamma(b)\\ \xi\in C^{ac}(I,M)}}\int_a^b L(d\xi(t),t)-\eta_c(d\xi(t))\, dt,$$
where $C^{ac}(I,M)$ denotes the set of  absolutely continuous curves.
As is known to all, each minimal curve satisfies the Euler-Lagrange equation.
A curve $\gamma:$ $\R$ $\rightarrow$ $M$ is called \emph{globally $c$-minimal} if for any $a<b$, the curve $\gamma:[a,b]\to M$ is  $c$-minimal. Therefore, we introduce the \emph{globally minimal set}
$$\widetilde{\cG}(c):=\bigcup_{\gamma}\{  (d\gamma(t),t)~:~ \gamma:\R\rightarrow M~ \text{is ~} c\text{-minimal} \}.$$

Let $\phi^t_L$ be the Euler-Lagrange flow on $TM\times\T$, and $\fM$ be the space of all $\phi^t_L$-invariant probability measures on $TM\times\T$. To each $\mu\in\fM$, Mather has proved that $\int_{TM\times\T} \lambda \,d\mu$=0 holds for any exact 1-form $\lambda$, which yields that $\int_{TM\times\T} L-\eta_c \,d\mu=\int_{TM\times\T} L-\eta^\prime_c \, d\mu$  if $\eta_c-\eta_c^\prime$ is exact. This leads us to define  \emph{Mather's $\alpha$ function},
$$\alpha(c):=-\inf_{\mu\in\fM}\int_{TM\times\T} L-\eta_c \,d\mu.$$
To some extent, the value $\alpha(c)$ is a minimal average action for $L-\eta_c$. Mather has proved that $\alpha:H^1(M,\R)\rightarrow\R$ is finite everywhere, convex and superlinear. 

For each $\mu\in\fM$, the \emph{rotation vector} $\rho(\mu)$ associated with $\mu$   is the unique element in $H_1(M,\R)$ that satisfies
$$\langle\rho(\mu),[\eta_c]\rangle=\int_{TM\times\T}\eta_c\, d\mu,\quad  \text{~ for all closed 1-form~} \eta_c,$$
here $\langle\cdot,\cdot\rangle$ denotes the dual pairing between homology and cohomology classes.
Then, we can define \emph{Mather's $\beta$ function} as follows:
$$\beta(h):=\inf_{\mu\in\fM, \rho(\mu)=h}\int_{TM\times\T} L\,d\mu.$$
This function $\beta:H_1(M,\R)\rightarrow\R$ is also finite everywhere, convex and superlinear. In fact, $\beta$ is  the Legendre-Fenchel dual of the function $\alpha$, i.e. $\beta(h)=\max_{c}\{\langle h,c\rangle-\alpha(c)\}$.

We define $$\fM^c(L):=\left\{ \mu: \int_{TM\times\T}L-\eta_c \,d\mu=-\alpha(c) \right\},~ \fM_h(L):=\left\{ \mu : \rho(\mu)=h,  \int_{TM\times\T}L \,d\mu=\beta(h) \right\}.$$
By duality,  it can be easily checked that $$\fM^c(L)=\bigcup_{h\in\partial \alpha(c)} \fM_h(L),$$
where $\partial \alpha(c)$ is the sub-differential. For one and a half degrees of freedom  systems including twist maps, Mather's $\alpha$ function is of class $C^1$, then 
\begin{equation}\label{ssff}
	\fM^c(L)= \fM_h(L),\quad d\alpha(c)=h.
\end{equation}

We call each element $\mu\in\fM^c(L)$  a \emph{$c$-minimal measure}. The  \emph{Mather set} of
cohomology class $c$ is then defined by
$$\widetilde{\cM}(c):=\overline{\bigcup_{\mu\in\fM^c(L)} \text{supp}\mu}.$$

To study more dynamical properties, we need to find some ``larger" minimal invariant sets and discuss their topological structures. Let $t^\prime> t$,  the action function $h^{t,t^\prime}_c:M\times M\rightarrow \R$  is defined by
$$h^{t,t^\prime}_c(x,x^\prime):=\min_{\substack{\gamma(t)=x, \gamma(t^\prime)=x^\prime\\ \gamma\in C^{ac}([t,t'],M)}}\int_t^{t^\prime} (L-\eta_c)(d\gamma(s),s)\,ds+\alpha(c)\cdot(t^\prime-t).$$
Then we define a real-valued function  $\Phi_c:(M\times\T)\times(M\times\T)\rightarrow \R$  by
$$\Phi_c((x,\tau),(x^\prime, \tau^\prime)):=\inf_{\substack{t^\prime>t,~t\equiv\tau  \text{mod}~1 \\  t^\prime\equiv \tau^\prime \text{mod}~1}}h^{t,t^\prime}_c(x,x^\prime).$$
and a real-valued function $h_c^\infty:(M\times\T)\times(M\times\T)\rightarrow \R $ by
\begin{equation}
 h_c^\infty((x,\tau),(x^\prime, \tau^\prime))=\liminf\limits_{\substack{t\equiv\tau  \text{mod}~1\\ t^\prime\equiv \tau^\prime \text{mod}~1,t^\prime-t\rightarrow+\infty}}h^{t,t^\prime}_c(x,x^\prime)
\end{equation}
In the literature,   $h_c^\infty$ and $\Phi_c$ are called the \emph{Peierls barrier function} and \emph{Ma\~n\'e's potential} respectively.

A minimal curve $\gamma:\R\rightarrow M$ is called $c$-semi static if for any $t<t'$,
\begin{equation}\label{semi static}
A_c(\gamma|_{[t,t^\prime]})+\alpha(c)\cdot(t^\prime-t)=\Phi_c\big(~(\gamma(t),t\text{~mod~} 1),~(\gamma(t^\prime),t^\prime\text{~mod~} 1)~\big).
\end{equation}
A minimal curve $\gamma:\R\rightarrow M$ is called $c$-static if for any $t<t'$,
\begin{equation}\label{static}
A_c(\gamma|_{[t,t^\prime]})+\alpha(c)\cdot(t^\prime-t)=-\Phi_c\big(~(\gamma(t^\prime),t^\prime\text{~mod~} 1),~(\gamma(t),t\text{~mod~} 1)~\big).
\end{equation}
This gives the so-called \emph{Aubry set} $\widetilde{\cA}(c)$ and \emph{Ma\~{n}\'{e} set} $\widetilde{\cN}(c)$ in $TM\times\T$:
\begin{equation*}
    \begin{split}
        \widetilde{\cA}(c)=\bigcup \big\{(d\gamma(t),t\text{~mod~}1):~\gamma ~\text{is}~ c\text{-static}\big\}, ~~
        \widetilde{\cN}(c)=\bigcup \big\{(d\gamma(t),t\text{~mod~}1):~\gamma ~\text{is}~ c\text{-semi static}\big\}.
    \end{split}
\end{equation*}
The $\alpha$-limit and $\omega$-limit sets of a $c$-minimal curve $(d\gamma(t),t)$ belong to $\widetilde{\cA}(c)$, see for instance \cite{Be2002}.  In addition, with the canonical projection  $\pi:TM\times\T\rightarrow M\times\T$, one could define the \emph{projected Aubry set}  $\cA(c)=\pi\widetilde{\cA}(c)$, the \emph{projected Mather set} $\cM(c)=\pi\widetilde{\cM}(c)$, the \emph{ projected Ma\~n\'e set} $\cN(c)=\pi\widetilde{\cN}(c)$ and the \emph{projected globally minimal set} $\cG(c)=\pi\widetilde{\cG}(c)$.   Then the following inclusion relations hold (see  \cite{Be2002}):
\begin{equation*}
  \widetilde{\cM}(c)\subset\widetilde{\cA}(c)\subset\widetilde{\cN}(c)\subset\widetilde{\cG}(c),\quad \cM(c)\subset\cA(c)\subset \cN(c)\subset\cG(c).
\end{equation*}

Next, we present some key properties of the  minimal  sets above, which will be fully exploited in the construction of diffusing orbits. Property (1) below is a classical result which has been proved by J. N. Mather in \cite{Ma1991}, and the proof of property (2) could be found in \cite{Be2002,CY2004}.
\begin{Pro}\label{upper semi}
For the Tonelli Lagrangian $L$, we have:
\begin{enumerate}[\rm(1)]
  \item \textup{(Graph property)} Let  $\pi:TM\times\T\rightarrow M\times\T$ be the canonical projection. Then 
  the restriction  of $\pi$ to $\widetilde{\mathcal{A}}(c)$ is a bi-Lipschitz homeomorphism.
  \item \textup{(Upper semi-continuity)} The set-valued map $(c,L)\mapsto\widetilde{\cG}(c,L)$ and the set-valued map $(c,L)\mapsto\widetilde{\cN}(c,L)$ are both upper semi-continuous.
\end{enumerate}
\end{Pro}

For $(x,\tau)$, $(x^\prime, \tau^\prime)$ $\in M\times\T$,
we set  $$d_c((x,\tau),(x^\prime, \tau^\prime)):=h_c^\infty((x,\tau),(x^\prime, \tau^\prime))+h_c^\infty((x^\prime, \tau^\prime),(x,\tau)).$$ By definition \eqref{static}, it follows that
$$h^\infty_c((x,\tau),(x, \tau))=0\Longleftrightarrow(x,\tau)\in\cA(c),$$ and hence $d_c$ is a pseudo-metric on the projected Aubry set $\cA(c)$. Two points $(x,\tau), (x^\prime, \tau^\prime)\in\cA(c)$ are said to be in the same \emph{Aubry class} if $d_c((x,\tau),(x^\prime, \tau^\prime))=0$. Clearly, each Aubry class is a closed set. If only one $c$-minimal measure exists, then the Aubry class is unique and $$\widetilde{\cA}(c)=\widetilde{\cN}(c).$$

To characterize the Ma\~{n}\'{e} set from another point of view, we define the following function
\begin{equation*}
B_c^*(x,\tau):=\min\limits_{\substack{(x_\ell,\tau_\ell)\in\cA(c)\\\ell=1,2}}\{h_c^\infty((x_1,\tau_1),(x, \tau))+h_c^\infty((x,\tau),(x_2, \tau_2))-h_c^\infty((x_1,\tau_1),(x_2, \tau_2)) \}.
\end{equation*}
Mather has proved in \cite{Ma2004} that $\min  B_c^*=0$, and  the set of  all minimal points is exactly $\cN(c)$, i.e.
\begin{equation}\label{manebarrier}
B_c^*(x,\tau)=0 \Longleftrightarrow (x,\tau)\in\cN(c).
\end{equation}

To prove Theorem \ref{generic G1} in Section \ref{sec proof main}, 
 it is convenient to adopt the equivalent definition of minimal measures originating from  Ma{\~{n}}{\'{e}} \cite{Mane1996}. In his setting, the minimal measures are obtained through a variational principle not requiring the invariance a priori. Let $\rm C$ be the set of all continuous functions $f:TM\times\T\to\R$ having linear growth at most, i.e.
$$\|f\|_l:=\sup_{(v,t)}\frac{|f(q,v,t)|}{1+\|v\|_q}<+\infty,$$
and endow $\rm C$ with the norm $\|\cdot\|_l$. Let $\rm C^*$ be the vector space of all continuous linear functionals $\nu: \rm C\to\R$ provided with the weak-$*$ topology, namely,
$$\lim\limits_{k\to+\infty}\nu_k=\nu \Longleftrightarrow \lim\limits_{k\to+\infty}\int_{TM\times\T}f\, d\nu_k=\int_{TM\times\T}f \,d\nu,\quad \forall f\in \rm C.$$

For each $N\in\Z^+$ and each  $N$-periodic absolutely continuous curve $\gamma:\R\rightarrow M$, one can define a probability measure $\mu_\gamma$ associated to $\gamma$ as follows:
\begin{equation}\label{holonomic}
\int_{TM\times\T}f \,d \mu_\gamma :=\frac{1}{N}\int_0^N f(d\gamma(t),t)\,dt,\quad \forall f\in \rm C.
\end{equation}
  Let
$$\Gamma:=\bigcup_{N\in\Z^+}\big\{~\mu_\gamma~:~\gamma\in C^{ac}(\R,M) \textup{~is~} N \textup{-periodic} ~\big\}\subset \rm C^*.$$
and let  $\mathcal{H}$  be the closure of $\Gamma$ in $\rm C^*$. It is easily seen that the set $\mathcal{H}$ is  convex.

$\mu_\gamma$ in $\Gamma$ has a naturally associated homology class $\rho(\mu_\gamma)=\frac{1}{N}[\gamma]\in H_1(M,\R),$ where $[\gamma]$ denotes the homology class of $\gamma$. The map $\rho:\Gamma\rightarrow H_1(M,\R)$ can extend continuously to $\rho:\mathcal{H}\rightarrow H_1(M,\R)$ which
is surjective.  Then, Ma\~n\'e introduced the  following minimal measures:
\begin{equation}\label{definition of mane}
\begin{split}
 \mathfrak{H}^c(L)&:=\bigg\{\mu\in\mathcal{H} ~:~ \int L-\eta_c \,d\mu=\min\limits_{\nu\in\mathcal{H}}\int L-\eta_c\, d\nu \bigg\}, \\
 \mathfrak{H}_h(L)&:=\bigg\{\mu\in\mathcal{H} ~:~\rho(\mu)=h, \int L\,d\mu=\min\limits_{\nu\in\mathcal{H},\rho(\nu)=h}\int L\,d\nu \bigg\}.
\end{split}
\end{equation}

We end this section by the following equivalence property:
\begin{Pro}\rm{(\cite{Mane1996})}\label{Mather and Mane}
The sets $\fM^c(L)=\mathfrak{H}^c(L)$ and $\fM_h(L)=\mathfrak{H}_h(L)$.
\end{Pro}

\section{Elementary weak KAM solutions and heteroclinic orbits}\label{sec EWS}
\subsection{Weak KAM solutions}\label{sub weakkam}
Weak KAM solution is the basic element in weak KAM theory which builds a link between Mather theory and the theory of viscosity solutions of Hamilton-Jacobi equations. Here we only recall some basic concepts and properties which  help us  better understand Mather theory. For more details, we refer the reader to  Fathi's book \cite{Fa2008} for time-independent systems, and to \cite{Be2008,CIS2013,WY2012} for time-periodic systems.
\begin{Def}\label{weakKAM}A continuous function $u_c^-:M\times\T\rightarrow\R$ is called a \emph{backward weak KAM solution} if
\begin{enumerate}[\rm(1)]
  \item For any absolutely continuous curve $\gamma:[a,b]\to M$,
  $$u_c^-(\gamma(b),b)-u_c^-(\gamma(a),a)\leq \int_{a}^{b} (L-\eta_c)(d\gamma(s),s)+\alpha(c)\,ds.$$
  \item For each $(x,t)\in M\times\R$, there exists a \emph{backward calibrated curve} $\gamma^-:(-\infty,t]\rightarrow M$ with $\gamma^-(t)=x$ such that for all $a<b\leq t$,
  $$u_c^-(\gamma^-(b),b)-u_c^-(\gamma^-(a),a)=\int_{a}^{b}(L-\eta_c)(d\gamma^-(s),s)+\alpha(c)\,ds.$$
\end{enumerate}
Similarly, a continuous function $u_c^+:M\times\T\rightarrow\R$ is called a \emph{forward weak KAM solution} if
\begin{enumerate}[\rm(1)]
  \item For any absolutely continuous curve $\gamma:[a,b]\to M$,
  $$u_c^+(\gamma(b),b)-u_c^+(\gamma(a),a)\leq \int_{a}^{b} (L-\eta_c)(d\gamma(s),s)+\alpha(c)\,ds.$$
  \item For each $(x,t)\in M\times\R$, there exists a \emph{forward calibrated curve} $\gamma^+:[t,+\infty)\rightarrow M$ with $\gamma^+(t)=x$ such that for all $t\leq a<b$,
  $$u_c^+(\gamma^+(b),b)-u_c^+(\gamma^+(a),a)=\int_{a}^{b}(L-\eta_c)(d\gamma^+(s),s)+\alpha(c)\,ds.$$
\end{enumerate}
\end{Def}

For example, it is well known in weak KAM theory that, for each $(x_0,t_0)\in M\times\T$ the barrier function $ h_c^\infty((x_0,t_0),\cdot): M\times\T\to\R$ is a backward weak KAM solution and  $-h_c^\infty(\cdot,(x_0, t_0)):M\times\T\to\R$ is a forward weak KAM solution. If there is only one Aubry class, then $ h_c^\infty((x_0,t_0),\cdot)$  is the unique backward  weak KAM solution up to an additive constant,   and $-h_c^\infty(\cdot,(x_0,t_0))$ is also the unique forward  weak KAM solution up to an additive constant.

It is easily seen that backward (forward) calibrated curves are semi static. The following properties for weak KAM solutions are well known and the proof can be found in \cite{Fa2008} or \cite{CIS2013}:
\begin{Pro}\label{properties weak KAM}
\textup{(1)}  $u^-_c$ is Lipschitz continuous, and  is differentiable on $\cA(c)$. If $u^-_c$ is differentiable at $(x_0,t_0)\in M\times\T$, then
  \begin{equation*}
     \partial_tu^-_c(x_0,t_0)+H(x_0,c+\partial_xu^-_c(x_0,t_0),t_0)=\alpha(c).
  \end{equation*}
It also determines a unique $c$-semi static curve $\gamma^-_c:(-\infty,t_0]\rightarrow M$ with $\gamma^-_c(t_0)=x_0$, and such that $u_c^-$ is differentiable at each point $(\gamma^-_c(t),t)$  with $t\leq t_0$, namely $c+\partial_xu^-_c(\gamma^-_c(t),t)=\frac{\partial L}{\partial v}(d\gamma^-_c(t),t)$.\\
\textup{(2)}  $u^+_c$ is Lipschitz continuous, and is differentiable on $\cA(c)$. If $u^+_c$ is differentiable at $(x_0,t_0)\in M\times\T$, then
  \begin{equation*}
     \partial_tu^+_c(x_0,t_0)+H(x_0,c+\partial_xu^+_c(x_0,t_0),t_0)=\alpha(c).
  \end{equation*}
It also determines a unique $c$-semi static curve $\gamma^+_c:[t_0,+\infty)\rightarrow M$ with $\gamma^+_c(t_0)=x_0$, and such that $u_c^+$ is differentiable at each point $(\gamma^+_c(t),t)$ with $t\geq t_0$, namely  $c+\partial_xu^+_c(\gamma^+_c(t),t)=\frac{\partial L}{\partial v}(d\gamma^+_c(t),t)$.
\end{Pro}

\subsection{Elementary weak KAM solutions}
It is a generic property that a Lagrangian  has  finitely many Aubry classes \cite{BC2008} for each cohomology class. Recall that the weak KAM solution is unique (up to an additive constant) if the Aubry class is unique. If two or more Aubry classes exist, there are infinitely many weak KAM solutions, among which we are only interested in the elementary weak KAM solutions.  In what follows, we assume that for certain cohomology class $c$ the  Aubry classes are $\{\cA_{c,i}: i=1,2,\cdots,k\}$, and hence the projected Aubry set 
$\cA(c)=\bigcup_i\cA_{c,i}.$ The concept of elementary weak KAM solution  appeared in  the  work  \cite{Be2008}. However, for the purpose of our applications,  we decide to adopt an 
analogous concept defined in \cite{Ch2012}.
\begin{Def}\label{def of EWS}
We fix an $i\in\{1,\cdots,k\}$ and perturb the Lagrangian $L\rightarrow L+\vep V(x,t)$ where $\vep>0$ and $V$ is a  non-negative $C^\infty$ function satisfying $\textup{supp}V\cap\cA_{c,i}=\emptyset$ and $V\big|_{\cA_{c,j}}>0$ for each $j\neq i$. Then for the cohomology class $c$, the perturbed Lagrangian has only one Aubry class $\cA_{c,i}$, and its backward weak KAM solution, denoted by $u^-_{c,i,\vep}$, is unique up to an additive constant. If for a subsequence $\{u^-_{c,i,\vep_k}\}$, the limit \begin{equation}\label{def ofews}
u^-_{c,i}:=\lim\limits_{\vep_k\rightarrow 0^+}u^-_{c,i,\vep_k},
\end{equation}
exists,  then we call  $u^-_{c,i}$ a \emph{backward elementary weak KAM solution}. Analogously, one can define a \emph{forward elementary weak KAM solution} $u^+_{c,i}$.
\end{Def}

In the following theorem, we will prove the existence of elementary  weak KAM solutions and give explicit representation formulas as well. 

\begin{The}\label{representation of EWS}
For each $i$, the backward (resp. forward) elementary weak KAM solution $u^-_{c,i}$ (resp. $u^+_{c,i}$) always exists and is unique up to an additive constant. More precisely, let $(x_i, \tau_i)$ be any point in $\cA_{c,i}$, then there exists a constant $C$ (resp. $C'$) depending on $(x_i, \tau_i)$, such that
$$u^-_{c,i}(x,\tau)=h^\infty_c((x_i,\tau_i),(x,\tau))+C \qquad (resp. ~ u^+_{c,i}(x,\tau)=-h^\infty_c((x,\tau),(x_i,\tau_i))+C'.)$$
\end{The}
\begin{proof}
We only give the proof for $u^-_{c,i}$ since $u^+_{c,i}$ is similar. Denote by $\alpha(c)$ and $\alpha_\vep(c)$ the value of Mather's $\alpha$-function at the cohomology class $c$ for the Lagrangians $L-\eta_c$ and $L-\eta_c+\vep V$ respectively, and denote by  $h^\infty_c((x_i,\tau_i),(x,\tau))$ and $h^\infty_{c,\vep}((x_i,\tau_i),(x,\tau))$ the corresponding Peierls barrier functions. We first claim that
\begin{equation}\label{limit of barrierfun}
 h^\infty_c((x_i,\tau_i),(x,\tau))=\lim\limits_{\vep\to0}h^\infty_{c,\vep}((x_i,\tau_i),(x,\tau)).
\end{equation}
Indeed, as $V\geq 0$  and its support does not intersect with $\cA_{c,i}$, we have $\alpha_\vep(c)=\alpha(c)$ and
$$h^\infty_c((x_i,\tau_i),(x,\tau))\leq \liminf\limits_{\vep\to0} h^\infty_{c,\vep}((x_i,\tau_i),(x,\tau)).$$

Now we turn to  the opposite inequality $ \limsup_{\vep\to0} h^\infty_{c,\vep}((x_i,\tau_i),(x,\tau))\leq h^\infty_c((x_i,\tau_i),(x,\tau))$. Assume by contradiction that there exist a subsequence $\{h^\infty_{c,\vep_k}\}_{k\in\N}$ and a point $(x',\tau')$ such that 
\begin{equation}\label{contradiction_i_k}
	\lim_{k\to\infty} h^\infty_{c,\vep_k}((x_i,\tau_i),(x',\tau'))> h^\infty_c((x_i,\tau_i),(x',\tau'))
\end{equation} 
For abbreviation, we denote 
\begin{equation}\label{phi_i_k}
	\phi^-_{c,\vep_k}(x,\tau):= h^\infty_{c,\vep_k}((x_i,\tau_i),(x,\tau)).
\end{equation}
By Definition  \ref{def of EWS}, $\cA_{c,i}$ is the only Aubry class for the Lagrangian $L-\eta_c+\vep_kV$, which gives $\phi^-_{c,\vep_k}(x_i,\tau_i)$ $=0$. Moreover, it is not hard to verify that the sequence $\{\phi^-_{c,\vep_k}\}_k$ is uniformly Lipschitz. Hence this sequence is also uniformly bounded. Thus, it follows from the Arzel\`{a}-Ascoli theorem that, by taking a subsequence if necessary, $\phi^-_{c,\vep_k}$  converges uniformly to a Lipschitz function $\phi^-_{c}$.

It is well known  in weak KAM theory that  $\phi^-_{c}$ is a backward weak KAM solution for $L-\eta_c$, then
  \[\phi^-_{c}(x,\tau)=\phi^-_{c}(x,\tau)-\phi^-_{c}(x_i,\tau_i)\leq h^\infty_{c}((x_i,\tau_i),(x,\tau)),\]
further, by letting  $k\to\infty$ on both sides of \eqref{phi_i_k}, we get
\[\lim_{k} h^\infty_{c,\vep_k}((x_i,\tau_i),(x,\tau))=\lim_{k}\phi^-_{c,\vep_k}(x,\tau)=\phi^-_c(x,\tau)\leq h^\infty_{c}((x_i,\tau_i),(x,\tau)),\]
which contradicts \eqref{contradiction_i_k}. This therefore proves equality $\eqref{limit of barrierfun}$.

Finally, we recall that $\cA_{c,i}$ is the unique Aubry class for  $L-\eta_c+\vep V$ ($\vep>0$). Then for any backward weak KAM solution $u^-_{c,i,\vep}$,
one has $u^-_{c,i,\vep}(\cdot)=h^\infty_{c,\vep}((x_i,\tau_i),\cdot)+C$ with $C$ a constant.  Now the theorem is evident from what we have proved.
\end{proof}
\begin{Rem}
By fixing a point  $(x_i, \tau_i)\in\cA_{c,i}$ for each index $i\in\{1,\cdots,k\}$, we conclude from Theorem \ref{representation of EWS} that the set of all backward  elementary weak KAM solutions is exactly
$\big\{h^\infty_c((x_i,\tau_i),\cdot)+C~:~C\in\R, ~ i=1,\cdots,k\big\}$, and the set of all forward elementary weak KAM solutions is exactly
$\big\{-h^\infty_c(\cdot,(x_i,\tau_i))+C ~: ~C\in\R, ~ i=1,\cdots,k\big\}$.
\end{Rem}

\subsection{Heteroclinic orbits between Aubry classes}
To study the heteroclinic trajectories from a variational viewpoint, we will use a special type of barrier function. Indeed, let $u^-_{c,i}(x,\tau)$ and $u^+_{c,j}(x,\tau)$ be a backward and a forward elementary weak KAM solution respectively. Now we define a function
\begin{equation}\label{anotherkind barr}
B_{c,i,j}(x,\tau):=u^-_{c,i}(x,\tau)-u^+_{c,j}(x,\tau), ~\textup{for each~} (x,\tau)\in M\times\T.
\end{equation}
Roughly speaking, it measures the  action along curves joining the Aubry class $\cA_{c,i}$ to $\cA_{c,j}$, we refer the reader to \cite{Be2008,CY2009,Ch2012} for more discussions.

In the sequel,  the notation   $\arg\min f$ denotes the minimal set $\{a ~|~f(a)=\min f\}$, then we have
\begin{Pro}\label{manejifenlei}
Suppose that the projected Aubry set $\cA(c)=\bigcup_{i=1}^k\cA_{c,i}$ consists of  $k$ $(k\geq 2)$ Aubry classes, then the projected Ma\~n\'e set
\begin{equation*}
\cN(c)=\bigcup_{i,j=1}^k \arg\min B_{c,i,j},
\end{equation*}
\end{Pro}
\begin{proof}
We first prove $ \cN(c)\supseteq\arg\min B_{c,i,j}$ for each $i, j$. Taking two points $(x_i,\tau_i)\in\cA_{c,i}$ and $(x_j,\tau_j)\in\cA_{c,j}$,  Theorem \ref{representation of EWS} implies that there exist two constants $C_i$ and $C_j$ such that
$$u^-_{c,i}(x,\tau)=h^\infty_c((x_i,\tau_i),(x,\tau))+C_i,\quad u^+_{c,j}(x,\tau)=-h^\infty_c((x,\tau),(x_j,\tau_j))+C_j.$$
Thus it's easy to compute that
$$\min B_{c,i,j}(x,\tau)=h^\infty_c((x_i,\tau_i),(x_j,\tau_j))+C_i-C_j.$$
If $(\tilde{x},\tilde{\tau})\in\arg\min B_{c,i,j}$, then
$$h^\infty_c((x_i,\tau_i),(\tilde{x},\tilde{\tau}))+C_i+h^\infty_c((\tilde{x},\tilde{\tau}),(x_j,\tau_j))-C_j=h^\infty_c((x_i,\tau_i),(x_j,\tau_j))+C_i-C_j,$$
namely $h^\infty_c((x_i,\tau_i),(\tilde{x},\tilde{\tau}))+h^\infty_c((\tilde{x},\tilde{\tau}),(x_j,\tau_j))-h^\infty_c((x_i,\tau_i),(x_j,\tau_j))=0$. By \eqref{manebarrier} one obtains $(\tilde{x},\tilde{\tau})\in\cN(c)$.

Now it remains to show $\cN(c)\subset\bigcup_{i,j=1}^k \arg\min B_{c,i,j}$. For each $(\bar{x},\bar{\tau})\in\cN(c)$, one deduces from \eqref{manebarrier} that there always exist $m, n\in\{1,2,\cdots,k\}$, and two points $(x_m,\tau_m)\in\cA_{c,m}$, $(x_n,\tau_n)\in\cA_{c,n}$ such that
$$h_c^\infty((x_m,\tau_m),(\bar{x},\bar{\tau}))+h_c^\infty((\bar{x},\bar{\tau}),(x_n, \tau_n))=h_c^\infty((x_m,\tau_m),(x_n, \tau_n)).$$
Combining with Theorem \ref{representation of EWS}, one gets that for each $(x,\tau)\in M\times\T$,
\begin{equation*}
\begin{split}
&u_{c,m}^-(\bar{x},\bar{\tau})-u_{c,n}^+(\bar{x},\bar{\tau})-\Big(u^-_{c,m}(x,\tau)-u^+_{c,n}(x,\tau)\Big)\\
=&h_c^\infty((x_m,\tau_m),(\bar{x},\bar{\tau}))+h_c^\infty((\bar{x},\bar{\tau}),(x_n, \tau_n))-\Big(h_c^\infty((x_m,\tau_m),(x,\tau))+h_c^\infty((x,\tau),(x_n, \tau_n))\Big)\\
=&h_c^\infty((x_m,\tau_m),(x_n, \tau_n))-\Big(h_c^\infty((x_m,\tau_m),(x,\tau))+h_c^\infty((x,\tau),(x_n, \tau_n))\Big)\leq 0,
\end{split}
\end{equation*}
hence $(\bar{x},\bar{\tau})\in \arg\min B_{c,m,n}$. This completes the proof.
\end{proof}

From now on, we denote by $\cN_{i,j}(c)$ the set of $c$-semi static curves which are negatively asymptotic to $\cA_{c,i}$ and positively asymptotic to $\cA_{c,j}$, i.e.,
\begin{equation}\label{ij maneorbit}
\cN_{i,j}(c)=\{(x,\tau): \exists \textup{~a~} c\textup{-semi static curve~} \gamma, \gamma(\tau)=x,  \textup{~and~} \alpha(\gamma(t),t)\subset\cA_{c,i},\omega(\gamma(t),t)\subseteq\cA_{c,j} \}.
\end{equation}
Obviously, $\cN_{i,j}(c)\subset\cN(c)$, and  each point $(x,\tau)\in\cN_{i,j}(c)$ satisfies
$$h^\infty_c((x_i,\tau_i),(x,\tau))+h^\infty_c((x,\tau),(x_j,\tau_j))=h^\infty_c((x_i,\tau_i),(x_j,\tau_j)),$$
then $\cN_{i,j}(c)\subset\arg\min B_{c,i,j}$ thanks to Theorem \ref{representation of EWS}. Moreover, $\cA_{c,i}\cup\cA_{c,j}\cup\cN_{i,j}(c)\subseteq\arg\min B_{c,i,j}$.

Conversely, the equality $\arg\min B_{c,i,j}\setminus\cA(c)=\cN_{i,j}(c)$ may not hold in general. For instance,  the pendulum Lagrangian $L=\frac{v^2}{2}-(\cos8\pi x-1)$  has four Aubry classes for the cohomology class $c=0\in H^1(\T,\R)$: 
$$\widetilde{\cA}_1=(0,0), \widetilde{\cA}_2=(\frac{1}{4},0),  \widetilde{\cA}_3=(\frac{1}{2},0),  \widetilde{\cA}_4=(\frac{3}{4},0).$$
They are all hyperbolic fixed points. By symmetry, it's easy to compute that $\arg\min B_{c,1,3}=\T$ but $\cN_{1,3}(c)=\emptyset.$

However, in the case of only two Aubry classes, we can give a precise description.
\begin{Pro}\label{double description}
Suppose that the projected Aubry set $\cA(c)=\cA_{c,1}\cup\cA_{c,2}$ has only two Aubry classes, then
$$ \arg\min B_{c,1,2}= \cA_{c,1}\cup\cA_{c,2}\cup\cN_{1,2}(c)\quad\textup{and}\quad \arg\min B_{c,2,1}=\cA_{c,1}\cup\cA_{c,2}\cup\cN_{2,1}(c).$$
\end{Pro}
\begin{proof}
We only prove $\arg\min B_{c,1,2}= \cA_{c,1}\cup\cA_{c,2}\cup\cN_{1,2}(c)$ and the other case is similar. By the analysis above, it only remains for us to verify $\arg\min B_{c,1,2}\subset\cA_{c,1}\cup\cA_{c,2}\cup\cN_{1,2}(c)$. Indeed, for each point$(x,\tau)\in\arg\min B_{c,1,2}$, we take $\tau=0$ for simplicity, then
\begin{equation}\label{Cor1}
B_{c,1,2}(x,0)=u^-_{c,1}(x,0)-u^+_{c,2}(x,0)=\min B_{c,1,2},
\end{equation}
and Proposition  \ref{manejifenlei} implies that there exists a $c$-semi static curve  $\gamma:\R\to M$, $\gamma(0)=x$, to be calibrated by $u^-_{c,1}$ on $(-\infty, 0]$ and be calibrated by $u^+_{c,2}$ on $(0, +\infty]$.

Next, there exist two points $(\alpha,0)$, $(\omega, 0)\in\cA(c)$ and a sequence of positive integers $\{m_k\}_{k}, \{n_k\}_{k}$ $\subset \Z^+$ such that
$$\lim_{k\to\infty}\gamma(-m_k)=\alpha\textup{~and~}\lim_{k\to\infty}\gamma(n_k)=\omega.$$
By the calibration property, 
\begin{align*}
	&u^-_{c,1}(\gamma(0),0)-u^-_{c,1}(\gamma(-m_k),0)+u^+_{c,2}(\gamma(n_k),0)-u^+_{c,2}(\gamma(0),0)\\
	=&\int_{-m_k}^{n_k}L(d\gamma(t),t)- \eta_c(d\gamma(t))+\alpha(c)\, dt.
\end{align*}
Let $\lim\inf k\to\infty$, 
\begin{equation}\label{Cor2}
B_{c,1,2}(x,0)=u^-_{c,1}(\alpha,0)-u^+_{c,2}(\omega,0)+h_c^\infty\big((\alpha,0),(\omega,0)\big).
\end{equation}

On the other hand, without loss of  generality (see Theorem \ref{representation of EWS}), we could assume $u^-_{c,1}(x,0)=h_c^\infty\big((x_1,0),(x,0)\big)$ with $(x_1,0)\in\cA_{c,1}$ and $u^+_{c,2}(x,0)=-h_c^\infty\big((x,0),(x_2,0)\big)$ with  $(x_2,0)\in\cA_{c,2}$. Then equalities  \eqref{Cor1} and \eqref{Cor2} together give rise to
\begin{equation*}
h_c^\infty\big((x_1,0),(x_2,0)\big)=h_c^\infty\big((x_1,0),(\alpha,0)\big)+h_c^\infty\big((\omega,0),(x_2,0)\big)+h_c^\infty\big((\alpha,0),(\omega,0)\big),
\end{equation*}
this could happen only if either $(\alpha,0), (\omega,0)$ belong to the same Aubry class or $(\alpha,0)\in\cA_{c,1}$, $(\omega,0)\in\cA_{c,2}$. This therefore completes the proof.
\end{proof}

Proposition \ref{double description} will be fully exploited in Section \ref{sec proof main} where we extend the Lagrangian to a double covering space such that the lift of the Aubry set contains two Aubry classes.

\section{Variational mechanism of diffusing orbits}\label{sec local and global}
In this section, we aim to give a master theorem  which guarantees the existence of diffusion for Tonelli Lagrangian $L: TM\times\T\to\R$ with $M=\T^n$.  Our construction of diffusion is variational, which requires less information about the geometric structure.  The orbits are constructed by shadowing a sequence of local connecting orbits, along each of them the Lagrangian action attains ``local minimum". Basically, among them there are two types of local connecting orbits, one is  based on Mather's variational mechanism  constructing orbits with respect to the cohomology equivalence \cite{Ma1993,Ma1995}, the other one is based on Arnold's geometric mechanism \cite{Ar1964} whose variational version was first achieved by Bessi \cite{Bessi1996} for Arnold's original example, and was later generalized to more general systems \cite{CY2004,CY2009,Be2008}.

Given a cohomology class $c\in H^1(M,\R) $, following Mather, we define
\begin{equation*}
\mathbb{V}_{c}=\bigcap_U\{i_{U*}H_1(U,\R): U\, \text{is a neighborhood of}\ \cN_0(c) \},
\end{equation*}
Here, $i_{U*}:H_1(U,\R)\to H_1(M,\R)$ is the mapping induced by the inclusion map $i_U$: $U\to M$, and $\cN_0(c)$ denotes the time-0 section of the projected Ma\~n\'e set $\cN(c)$ . Let $\mathbb{V}_{c}^{\bot}\subset  H^1(M,\mathbb{R})$ denote the annihilator of $\mathbb{V}_{c}$, i.e. $c'\in \mathbb{V}_{c}^{\bot}$ if and only if $\langle c',h \rangle =0$ for all $h\in \mathbb{V}_c$. Clearly,
\begin{equation*}
\mathbb{V}_{c}^{\bot}=\bigcup_U\{\ker i_{U}^*: U\, \text{is a neighborhood of}\, \cN_0(c)\}.
\end{equation*}
In fact, Mather has proved that there exists a neighborhood $U$ of $\cN_0(c)$ in $M$ such that $\mathbb{V}_{c}=i_{U^*}H_1(U,\R)$ and $\mathbb{V}^\bot_{c}=\ker i^*_U$ (see \cite{Ma1993}). Then we can introduce the cohomology equivalence (also known as $c$-equivalence).

\begin{Def}[\emph{Mather's $c$-equivalence}]\label{def_c_equivalenve}
We say that $c,c'\in H^1(M,\R)$ are $c$-equivalent if there exists a continuous curve $\Gamma$: $[0,1]\to H^1(M,\R)$ such that $\Gamma(0)=c$, $\Gamma(1)=c'$ and for each $s_0\in [0,1]$, $\exists$ $\vep>0$ such that $\Gamma(s)-\Gamma(s_0)\in \mathbb{V}_{{\Gamma}(s_0)}^{\bot}$ whenever $|s-s_0|<\vep$  and $s\in [0,1]$.
\end{Def}

By making full use of the cohomology equivalence,  Mather obtained a remarkable result on connecting orbits : if $c$ is equivalent to $c'$, the system has an orbit which in the infinite past tends to the Aubry set $\widetilde{\cA}(c)$ and  in the infinite future tends to the Aubry set $\widetilde{\cA}(c')$ \cite{Ma1993}. 

Next, we recall Arnold's famous example in \cite{Ar1964}: when the stable and unstable manifolds of an invariant circle transversally intersect each other, then the unstable manifold of this circle would also intersect the stable manifold of another invariant circle nearby. To understand this mechanism from a variational viewpoint, we let $\check{\pi}:\check{M}\rightarrow \T^n$ be a finite covering of $\T^n$. Denote by $\widetilde{\cN}(c,\check{M}), \widetilde{\cA}(c,\check{M})$ the corresponding Ma\~{n}\'{e} set and Aubry set with respect to  $\check{M}$. $\widetilde{\cA}(c,\check{M})$  may have several Aubry classes even if $\widetilde{\cA}(c)$ is unique. Here, we would like to emphasize that $\check{\pi}\widetilde{\cA}(c,\check{M})=\widetilde{\cA}(c)$. Also, it is not necessary to work always in a  nontrivial finite covering space, one can choose $\check{M}=M$ if the Aubry set already contains more than one classes.
Hence, for Arnold's famous example, the intersection of the stable and unstable manifolds implies that the set $\check{\pi}\cN(c,\check{M})\big|_{t=0}\setminus\big(\cA(c)\big|_{t=0}+\delta\big)$
is  discrete. 
Here, $\cA(c)\big|_{t=0}+\delta$ stands for a $\delta$-neighborhood of the set $\cA(c)\big|_{t=0}$.

This leads us to introduce the concept of \emph{generalized transition chain}. This notion could be found in \cite[Definition 5.1]{CY2009} as a generalization of  Arnold's transition chain  \cite{Ar1964}. In this paper, we adopt the definition as in \cite[Definition 4.1]{Ch2017} (see also \cite[Definition 2.2]{Ch2018}).
\begin{Def}[\emph{Generalized transition chain}]\label{transition chain}
Two cohomology classes $c, c'\in H^1(M,\R)$ are joined by a generalized transition chain if a continuous path $\Gamma: [0,1]\to H^1(M,\R)$ exists  such that $\Gamma(0)=c, \Gamma(1)=c'$, and for each $s\in[0,1]$ at least one of the following cases takes place:
\begin{enumerate}[(1)]
\item There is $\delta_s>0$ such that for each $s'\in (s-\delta_s, s+\delta_s)\bigcap  [0,1]$, $\Gamma(s')$ is $c$-equivalent to $\Gamma(s)$.
\item There exist a finite covering  $\check{\pi}:\check{M}\to M$ and a small $\delta_s>0$ such that the set
$\check{\pi}\cN(\Gamma(s),\check{M})\big|_{t=0}$ $\setminus$ $\big(\cA(\Gamma(s))\big|_{t=0}+\delta_s\big)$ is non-empty and totally disconnected.  $\cA(\Gamma(s'))$ lies in a neighborhood of $\cA(\Gamma(s))$ provided $|s'-s|$ is small.
\end{enumerate}
\end{Def}

We would like to emphasize that, the statement ``$\cA(\Gamma(s'))$ lies in a neighborhood of $\cA(\Gamma(s))$ provided $|s'-s|$ is small" in condition (2) could be guaranteed by the upper semi-continuity of Aubry sets. In fact, this upper semi-continuity is always true in our model since  the number of Aubry classes is only finite  (in fact, two at most), see \cite{Be2010On}. Also, condition (2) appears weaker than the condition of transversal intersection of stable and unstable manifolds because it still  works when the intersection is only topologically transversal. Our condition (2) is usually applied to the case where the Aubry set $\cA(\Gamma(s))$ is contained in a neighborhood of a lower dimensional torus, while condition (1) is usually applied to the case where the Ma\~n\'e set $\cN(\Gamma(s))$ is homologically trivial. 

Along a generalized transition chain, one can construct an orbit along which there is a substantial variation:
\begin{The}\label{generalized transition thm}
If $c$, $c'\in H^1(M,\R)$ are connected by a generalized transition chain $\Gamma$, then
\begin{enumerate}[\rm(1)]
\item there exists an orbit $(d\gamma(t),t)$ of the Euler-Lagrange flow  connecting the Aubry set $\widetilde{\cA}(c)$ to $\widetilde{\cA}(c')$, which means the $\alpha$-limit set $\alpha(d\gamma(t),t)\subset\widetilde{\cA}(c)$ and the $\omega$-limit set $\omega(d\gamma(t),t)\subset\widetilde{\cA}(c')$.
\item for any $c_1,\cdots, c_k\in \Gamma$ and small $\vep>0$, there exist an orbit $(d\gamma(t), t)$ of the Euler-Lagrange flow  and times $t_1<\cdots<t_k$, such that the orbit $(d\gamma(t),t)$ passes through the $\vep$-neighborhood of $\widetilde{\cA}(c_\ell)$ at the time $t=t_\ell$.
\end{enumerate}
\end{The}

The proof of Theorem \ref{generalized transition thm} is similar to  that of \cite[Section 5]{CY2009} and can also be found in \cite[Section 7]{Ch2012}. This variational mechanism of connecting orbits has already been used in \cite{Ch2017,Ch2018}. However, for the reader's convenience, we provide a proof of the theorem in appendix \ref{sec_proof_of_connectingthm}. We end this section by a simple illustration of the diffusing orbits in geometry, such orbits  constructed in Theorem $\ref{main theorem}$ and Theorem \ref{main thm2} would drift near the normally hyperbolic cylinder (see figure \ref{picture1}).
\begin{figure}[H]
    \centering
    \includegraphics[width=7cm]{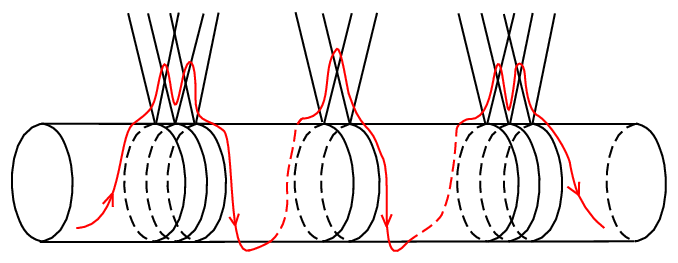}
    \caption{A global connecting orbit shadowing the generalized transition chain }\label{picture1}
\end{figure}

\section{Technical estimates on Gevrey functions}\label{some properties Gev}
In this part, we provide some necessary results for Gevrey functions defined on the torus $\T^n=\R^n/\Z^n$, which will be useful for our choice of Gevrey space in section \ref{determin of coeff}. We present this section in a self-contained way for the reader's convenience.

The variational  proof of  the genericity of Arnold diffusion
 usually depends on the existence of functions with compact support, i.e. bump functions. This technique cannot apply to the problem of analytic genericity since  no analytic function has compact support. However, the bump function does exist in the Gevrey-$\alpha$ category with $\alpha>1$. Here we give a modified Gevrey bump function which is based on the one constructed in \cite{MS2004}.

\begin{Lem}[Gevrey bump function]\label{Gevrey bumpfunction}
Let $\alpha>1, \bL>0$, $D=[a_1,b_1]\times\cdots\times[a_n,b_n]$$\varsubsetneq\T^n$ be a $n$-dimensional cube and $U$ be an open neighborhood of $D$. Then there exists $f\in\bG^{\alpha,\bL}(\T^n)$ such that $0\leq f\leq 1$, $\textup{supp}f\subset U$, and
$$f(x)=1 \Longleftrightarrow x\in D.$$
\end{Lem}
\begin{proof}
We first claim that for $0<d<d'<\frac{1}{2}$, there exists a  function $g\in\bG^{\alpha,\bL}(\T)$ such that $0\leq g\leq 1$ and
$$g(x)=1 \Longleftrightarrow x\in [-d,d],\quad\textup{supp}g\subset [-d',d'].$$
Indeed, let $\alpha=1+\frac{1}{\sigma}$ ($\sigma>0$) and define a non-negative  function $h\in C^\infty(\R)$ as follows: $h(x)=0$ for $x\leq 0$,
$h(x)=$$\exp(-\frac{\lambda}{x^\sigma})$ for $x>0$. Then $h\in\bG^{\alpha,\bL}(\R)$ if the constant $\lambda>(2\bL^\alpha/\sin a)^\sigma/\sigma$ with $a=\frac{\pi}{4}\min\{1,\frac{1}{\sigma}\}$ (cf.  \cite[Lemma A.3]{MS2003}). Next, we define $\psi(x)=\int_{-\infty}^xh\big(t+\frac{d'-d}{2}\big)h\big(-t+\frac{d'-d}{2}\big)~dt.$
It's easy to compute that $\psi\geq 0$ is non-decreasing  and
\begin{equation*}
\psi(x)=\left\{
         \begin{array}{ll}
           0, & x\leq-\frac{d'-d}{2} \\
           K, & x\geq \frac{d'-d}{2}
         \end{array}
       \right.
\end{equation*}
where $$K=\int^{\frac{d'-d}{2}}_{-\frac{d'-d}{2}}h\big(t+\frac{d'-d}{2}\big)h\big(-t+\frac{d'-d}{2}\big)~dt>0.$$
Then we define the function
$$g(x)=\frac{1}{K^2}\psi\big(x+\frac{d'+d}{2}\big)\psi\bigg(-x+\frac{d'+d}{2}\big).$$
Obviously, $0\leq g\leq 1$,  $\textup{supp}g\subset[-d',d']$, and $g(x)=1$ $\Longleftrightarrow$ $x\in[-d,d]$.  It can be viewed as a function defined on $\T$. Hence by property (G\ref{algebra norm}) in Section \ref{introduction}, $g\in\bG^{\alpha,\bL}(\T)$, which proves our claim.

Next, without loss of generality we assume $D=[-d_1,d_1]\times\cdots\times[-d_n,d_n]$ with $0<d_i<\frac{1}{2}$. By assumption, we can find another cube $D'=[-d'_1,d'_1]\times\cdots\times[-d'_n,d'_n]$ such that $D\subset D'\subset U\subset\T^n$. By the claim above, for each $i\in\{1,\cdots,n\}$ there exists a function $f_i\in\bG^{\alpha,\bL}(\T)$ such that $0\leq f_i\leq 1$, $\textup{supp}f_i\subset[-d_i',d_i']$, $f_i(x)=1$ $\Longleftrightarrow$ $x\in[-d_i,d_i]$. Thus we define $$f(x_1,\cdots,x_n):=\prod_{i=1}^nf_i(x_i),$$
which meets our requirements.
\end{proof}

Next, we prove that the inverse  of a Gevrey map is still Gevrey smooth. For each high dimensional map  $\varphi=(\varphi_1,\cdots,\varphi_n): V\to\R^n$ where $\varphi_i\in\bG^{\alpha,\bL}(V)$,  its norm could be defined as follows: $$\|\varphi\|_{\alpha,\bL}:=\sum\limits_{i=1}^n\|\varphi_i\|_{\alpha,\bL}.$$
In what follows,  $(0,1)^n$ denotes the unit domain $(0,1)\times\cdots\times(0,1)$ in $\R^n$. We also refer the reader to \cite{Kom1979} for the inverse function theorem of a general ultra-differentiable mapping.
\begin{The}[Inverse Function Theorem of Gevrey class]\label{inverse thm}
Let $X,Y$ be two open sets in $(0,1)^n$ and  let  $f: X\to Y$ be a Gevrey-$(\alpha,\bL)$ map with $\alpha\geq 1$. If the Jacobian matrix $Jf$ is non-degenerate at $x_0\in X$, then there exist an open set $U$ containing $x_0$, an open set $V$ containing $f(x_0)$, a constant $\bL_1<\bL$, and a unique inverse map $f^{-1}:V\to U$ such that $f^{-1}\in\bG^{\alpha,\bL_1}(V)$.
\end{The}
\begin{proof}
For simplicity we suppose the Jacobian matrix $J_{x_0}f=I_n$ where $I_n=\textup{diag}(1,1,\cdots,1)$, otherwise we can replace $f$ by $f\circ(J_{x_0}f)^{-1}$. We also suppose $f(x_0)=x_0$, otherwise we can replace $f$ by $f+x_0-f(x_0)$. If we write $f=id+h$ in a neighborhood of $x_0$, then $h(x_0)=0$, $J_{x_0}h=0$. For $0<\vep\ll 1$ there exist $d>0$ and an open ball $B_{d}(x_0)=$$\{x\in X:\|x-x_0\|<d\}$ such that
\begin{equation}\label{derofh}
\quad \|h\|_{C^1(B_d(x_0))}\leq\vep.
\end{equation}
By classical Inverse Function Theorem, there exist two small open sets $U, V\subset B_{d/2}(x_0)$ containing $x_0$ and a unique $ C^\infty$ inverse map $f^{-1}:V\to U$ where $f^{-1}(x_0)=x_0$. Let $\bL_1=\vep^{\frac{2}{3\alpha}}$, next we will prove  $f^{-1}\in\bG^{\alpha,\bL_1}(V)$ by the contraction mapping principle.

We can write $f^{-1}=id+g$ , so $g\in C^\infty(V)$ and the equality $$g(y)=-h(y+g(y)),~\forall y\in V$$ holds. Define the set
$E=\{\varphi=(\varphi_1,\cdots,\varphi_n): \varphi(x_0)=0,~\varphi\in\bG^{\alpha,\bL_1}(V),~\|\varphi\|_{\alpha,\bL_1}\leq \vep^{\frac{3}{4}}\}$ with the norm $\|\cdot\|_{\alpha,\bL_1}$, it's a non-empty, closed and convex set in the space $\bG^{\alpha,\bL_1}(V)$.
Define the operator $$(T\varphi)(y):=-h(y+\varphi(y)), \forall y\in V.$$

\noindent$\bullet$  We first claim that the mapping $T\varphi\in E,$ $\forall \varphi\in E.$ In fact, for each $\varphi\in E$, $(T\varphi)(x_0)=0$.  For $y\in V\subset B_{d/2}(x_0)$, we have
$\|y+\varphi(y)-x_0\|\leq\|y-x_0\|+\|\varphi(y)-\varphi(x_0)\|\leq\frac{d}{2}+\|J\varphi\|\|y-x_0\|<d$, and hence $(id+\varphi)(V)$$\subset B_d(x_0)$. Moreover, let $\bL_2:=\bL\vep^{\frac{1}{2\alpha}}$ and $\vep$ be suitably small.  For each $i\in\{1,\cdots,n\}$,
\begin{equation*}
 \begin{split}
 \|x_i+\varphi_i\|_{\alpha,\bL_1}-\|x_i+\varphi_i\|_{C^0}=&\sum\limits_{j=1}^n\bL_1^{\alpha}\|\delta_{ij}+\partial_{x_j}\varphi_i\|_{C^0}+\sum\limits_{k\in\N^n,|k|\geq2}\frac{\bL_1^{|k|\alpha}}{(k!)^\alpha}\|\partial^k\varphi_i\|_{C^0}\\
\leq&n\bL_1^\alpha(1+\frac{\vep^\frac{3}{4}}{\bL_1^\alpha})+\|\varphi_i\|_{\alpha,\bL_1}\leq2n\vep^{\frac{2}{3}}+\vep^{\frac{3}{4}}\leq\frac{\bL_2^\alpha}{n^{\alpha-1}},
 \end{split}
\end{equation*}
where $\delta_{ij}=1$ for $i=j$ and $\delta_{ij}=0$ for $i\neq j$. Hence by  property (G\ref{composition}) in Section \ref{introduction}, $\|T\varphi\|_{\alpha,\bL_1}=\|h\circ(id+\varphi)\|_{\alpha,\bL_1}\leq\|h\|_{\alpha,\bL_2,B_d(x_0)}$ since  $(id+\varphi)(V)$$\subset B_d(x_0)$. Now it only remains to verify that
$$\|h\|_{\alpha,\bL_2,B_d(x_0)}\leq\vep^{\frac{3}{4}}.$$
Recall that for $|k|\geq 2$ and $x\in B_d(x_0)$, $\partial^k f_i(x)=\partial^k h_i(x)$. By using \eqref{derofh}, we have
\begin{equation}\label{hi esti}
 \begin{split}
\|h_i\|_{\alpha,\bL_2,B_d(x_0)}&=\|h_i\|_{C^0(B_d(x_0))}+\sum\limits_{k\in\N^n,|k|=1}\bL_2^\alpha\|\partial^kh_i\|_{C^0(B_d(x_0))}+\sum\limits_{k\in\N^n,|k|\geq2}\frac{\bL_2^{|k|\alpha}}{k!^\alpha}\|\partial^kf_i\|_{C^0(B_d(x_0))}\\
&\leq (1+n\bL_2^\alpha)\vep+\sum\limits_{k\in\N^n,|k|\geq2}\frac{\bL^{|k|\alpha}\vep^{\frac{|k|}{2}}}{k!^\alpha}\|\partial^kf_i\|_{C^0(B_d(x_0))}\\
&\leq (1+n\bL^\alpha\vep^\frac{1}{2})\vep+\vep\|f\|_{\alpha,\bL}\leq\frac{\vep^\frac{3}{4}}{n},
\end{split}
\end{equation}
which proves the claim.

\noindent$\bullet$ On the other hand, for $\varphi,\tilde\varphi\in E$ and $i\in\{1,\cdots,n\}$, by the Newton-Leibniz formula  we have
\begin{equation*}
\begin{split}
h_i(x+\varphi(x))-h_i(x+\tilde{\varphi}(x))=&\bigg(\int_0^1Jh_i\big(x+s\varphi(x)+(1-s)\tilde{\varphi}(x)\big)ds\bigg)\bigg(\varphi(x)-\tilde{\varphi}(x)\bigg)\\
=&F(x)\big(\varphi(x)-\tilde{\varphi}(x)\big)
\end{split}
\end{equation*}
where $Jh_i$ is the Jacobian matrix. It follows from property (G\ref{derivative Gevrey}) in Section \ref{introduction} and \eqref{hi esti} that
 $$\|Jh_i\|_{\alpha,\frac{\bL_2}{2},B_d(x_0)}\leq \frac{\|h_i\|_{\alpha,\bL_2,B_d(x_0)}}{(\bL_2-\bL_2/2)^\alpha}\sim O(\vep^\frac{1}{4})<\frac{1}{2n}$$
provided $\vep$ is suitably small. By property (G\ref{composition}), $\|F\|_{{\alpha,\bL_1,V}}\leq \|Jh_i\|_{\alpha,\frac{\bL_2}{2},B_d(x_0)}\leq\frac{1}{2n}$.

Finally, we deduce from  (G\ref{algebra norm}) that
$$\|h_i\circ(id+\varphi)-h_i\circ(id+\tilde{\varphi})\|_{\alpha,\bL_1}\leq\|F\|_{\alpha,\bL_1}\|\varphi-\tilde{\varphi}\|_{\alpha,\bL_1}\leq\frac{1}{2n}\|\varphi-\tilde{\varphi}\|_{\alpha,\bL_1}.$$
Hence $\|h\circ(id+\varphi)-h\circ(id+\tilde{\varphi})\|_{\alpha,\bL_1}\leq\frac{1}{2}\|\varphi-\tilde{\varphi}\|_{\alpha,\bL_1}$, namely $$\|T\varphi-T\tilde\varphi\|_{\alpha,\bL_1}\leq\frac{1}{2}\|\varphi-\tilde\varphi\|_{\alpha,\bL_1}.$$

In conclusion, $T: E\to E$ is a contraction mapping. By the contraction mapping principle, $T$ has a unique fixed point,  and hence the fixed point must be  $g$. Therefore, $f^{-1}=id+g\in\bG^{\alpha,\bL_1}(V)$.
\end{proof}

Sometimes we need to approximate a continuous function by Gevrey smooth ones. Convolution provides us with a systematic technique. More specifically, for any $\alpha>1, \bL>0$, by Lemma \ref{Gevrey bumpfunction} there exists a non-negative function $\eta\in\bG^{\alpha,\bL}(\R^n)$ such that  $\textup{supp}\eta$ $\subset[\frac{1}{4},\frac{3}{4}]^n$
and $\int_{\R^n}\eta(x) dx=1$. Next we set $\eta_\vep(x)=\frac{1}{\vep^n}\eta(\frac{x}{\vep})$ $(0<\vep<1, x\in\R^n)$ which is called the mollifier.  Then we define the convolution of $\eta_\vep$ and $f\in C^0(\T^n)$ by
\begin{equation}\label{mollifier}
 \eta_\vep*f(x)=\int_{\T^n}\eta_\vep(x-y)f(y)dy,~\forall~x\in\T^n.
\end{equation}

\begin{The}[Gevrey approximation]
\begin{enumerate}[\rm(1)]
\item Let $\alpha>1$, and $U\subset\T^n$, $V\subsetneq(0,1)^n$ be two open sets. If $f:U\to V$ is a continuous map, then there exists a sequence of maps $f^\vep:U\to(0,1)^n$ such that $f^\vep\in\bG^{\alpha,\bL_\vep}(U)$. Furthermore, $\bL_\vep\to 0$ and $\|f^\vep-f\|_{C^0}\to 0$ as $\vep$ tends to 0.
\item Let $\alpha>1$, $U, V$ be connected open sets satisfying $\bar{U},\bar{V}\varsubsetneq\T^n$ and $f: U\to V$ be a continuous map. Then there exists a sequence of  maps $f^\vep:U\to \T^n$ such that $f^\vep\in\bG^{\alpha,\bL_\vep}(U)$,  $\bL_\vep\to 0$ and $\|f^\vep-f\|_{C^0}\to 0$ as $\vep$ tends to 0. Specifically, if $f$ is a diffeomorphism and the determinant $\det(Jf)$ ($Jf$ is the Jacobian matrix) has a uniform positive distance away from zero, then the Gevrey map $f^\vep:U\to V^\vep$ with $V^\vep=f^\vep(U)$ will also be a diffeomorphism provided that $\vep$ is small enough.
\end{enumerate}\label{Gevrey approx}
\end{The}
\begin{proof}
(1): Let $f=(f_1,\cdots,f_n)$ and $f_i$ ($1\leq i\leq n$) be continuous, we only need to prove that each $f_i$ can be approximated by a Gevrey smooth function. Indeed, let $f_i^\vep=\eta_\vep*f_i$ ($0<\vep<1$), where $\eta\in\bG^{\alpha,\bL}$. It's easy to check that $f_i^\vep:U\to(0,1)$ since $\int_{\R^n}\eta_\vep(x) dx=1$ and $\textup{supp}\eta_\vep$$\subset[\frac{\vep}{4},\frac{3\vep}{4}]^n$. By the classical properties of convolutions, one obtains $f^\vep_i\in C^\infty$ and
\begin{equation*}
\|f_i^\vep-f_i\|_{C^0}\to 0, \textup{~as~} \vep\to 0.
\end{equation*}
\begin{equation*}
  \partial^kf^\vep_i=\partial^k\eta_\vep*f_i=\int_{\T^n}\partial^k\eta_\vep(x-y)f_i(y)dy,~\forall~k=(k_1,\cdots,k_n)\in\Z^n, k_i\geq 0.
\end{equation*}
It only remains to prove $f_i^\vep$ is Gevrey smooth. In fact, if one sets $\bL_\vep=\bL\vep^{\frac{1}{\alpha}}$, then
\begin{equation*}
\begin{split}
 \|f^\vep_i\|_{\alpha,\bL_\vep}
 &\leq\sum\limits_{k}\frac{\bL_\vep^{|k|\alpha}}{k!^\alpha}\|\partial^k\eta_\vep\|_{C^0}\|f_i\|_{C^0}\\
&\leq \frac{\|f_i\|_{C^0}}{\vep^n}\sum\limits_{k}\frac{\bL_\vep^{|k|\alpha}\vep^{-|k|}}{k!^\alpha}\|\partial^k\eta\|_{C^0}\\
&=\frac{\|f_i\|_{C^0}}{\vep^n}\sum\limits_{k}\frac{\bL^{|k|\alpha}}{k!^\alpha}\|\partial^k\eta\|_{C^0}=\frac{\|f_i\|_{C^0}}{\vep^n}\|\eta\|_{\alpha,\bL}.
\end{split}
\end{equation*}
Obviously, $\bL_\vep\to 0$ as $\vep\to 0$. This completes the proof of (1).

(2): The first part is not hard to prove by the technique in (1). Furthermore, if $f$ is a diffeomorphism from $U$ to $V$, then by using $\partial^kf^\vep=\eta_\vep*\partial^kf$ with $|k|=1$, one gets
\begin{equation}\label{gevapp}
\|f^\vep-f\|_{C^1}\to 0,\quad \vep\to 0.
\end{equation}
Since $\det(Jf)$ has a uniform positive distance away from zero, it concludes from  \eqref{gevapp} and Theorem \ref{inverse thm} that $f^\vep: U\to f^\vep(U)$ would also be a diffeomorphism for $\vep$ small enough. 
\end{proof}

\section{Proof of  the main results}\label{sec proof main}
This section is the main part of the present paper, which aims to prove  Theorem \ref{main theorem} and Theorem \ref{main thm2}. We will explain how to apply the tools exposed in the previous sections to \emph{a priori} unstable and Gevrey smooth systems. Before that, we need to do some preparations.

\subsection{Genericity of uniquely minimal measure in Gevrey or analytic topology} Let $M=\T^n$.
We fix an $h\in H_1(M,\R)$, it is well known that in the $C^r$ ($ 2\leq r \leq\infty$) topology, a generic Lagrangian has only one  minimal measure $\mu$ with the rotation vector $\rho(\mu)=h$ (see  \cite{Mane1996}).   Next, we will show that such a property  still holds in the Gevrey topology.
For this purpose, we shall consider it in a Gevrey space $\bG^{\alpha,\bL}(M\times\T)$ with $\alpha\geq 1, \bL>0$.  A property is called \emph{generic} in the sense of Ma\~n\'e if, for each  Lagrangian $L: TM\times\T\to\R$, there exists a residual\footnote{A residual subset $X$ of a Baire space is one whose complement is the union of countably many nowhere dense subsets. The residual set is a dense set.} subset $\mathcal{O}\subset\bG^{\alpha,\bL}(M\times\T)$ such that the property holds for each Lagrangian $L+\phi$ with $\phi\in\mathcal{O}$.

\begin{The}\label{generic G1}
Let $h\in H_1(M,\R)$, $\alpha\geq 1, \bL>0$ and $L:TM\times\T\to\R$ be a Tonelli Lagrangian, then there exists a residual subset $\mathcal{O}(h)\subset\bG^{\alpha,\bL}(M\times\T)$ such that, for each $ \phi\in\mathcal{O}(h)$, the Lagrangian $L+\phi$ has only one minimal measure with  the rotation vector $h$.
\end{The}
\begin{Rem}
We shall note that the residual set $\mathcal{O}(h)$ depends on the homology class $h$.	
\end{Rem}
\begin{proof}
Recall Ma\~n\'e's equivalent definition of minimal measure in Section \ref{Mathertheory},  we are going to prove this theorem in the following setting based on Ma\~{n}\'{e}'s approach.
\begin{enumerate}[(a)]
  \item  Set $E:=\bG^{\alpha,\bL}(M\times\T).$ Obviously, it is a Banach space.
  \item Denote by $F\subset \mathrm{C}^*$ the vector space spanned by the set of probability measures $\mu\in\mathcal{H}$ with $\int_{TM\times\T}L\,d\mu<\infty$, the definitions of the sets $\mathcal{H}$ and $\mathrm{C}^*$ are in Section \ref{Mathertheory}. Recall that for $\mu_k, \mu\in F$,
      $$\lim\limits_{k\to+\infty}\,\mu_k=\mu \Longleftrightarrow \lim\limits_{k\to+\infty}\int_{TM\times\T}f\, d\mu_k=\int_{TM\times\T}f d\mu,\quad\forall f\in \mathrm{C}.$$
  \item Let $\mathcal{L}: F\to \R$ be a linear map satisfying $\mathcal{L}(\mu)=\int L\, d\mu$, for every $\mu\in F$.
  \item Let $\varphi: E\to F^*$ be a linear map such that for each $\phi\in E$, $\varphi(\phi)\in F^*$ is defined as follows
    $$\langle\varphi(\phi),\mu\rangle:=\int \phi\, d\mu,~ \mu\in F.$$
  \item $K:=\{\mu\in F~|~\rho(\mu)=h\}$. It's easy to check that $K$ is a separable metrizable convex subset.

\end{enumerate}

For  $\phi\in E$, we denote
$$\arg\min(\phi):=\{~\mu\in K~|~\mathcal{L}(\mu)+\langle \varphi(\phi),\mu \rangle=\min\limits_{\nu\in K}(\mathcal{L}(\nu)+\langle \varphi(\phi),\nu \rangle )~\}.$$
It's easy to verify that our setting above satisfies all conditions of that in \cite[Proposition 3.1]{Mane1996}, so there exists a residual subset $\mathcal{O}(h)\subset E$ such that each $\phi\in\mathcal{O}(h)$ has the following property: $$\#\arg\min(\phi)=1.$$ Since $\arg\min(\phi)$ $=$ $\mathfrak{H}_h(L+\phi)$, see \eqref{definition of mane}, it follows from Proposition \ref{Mather and Mane} that the Lagrangian $L+\phi$  admits only one  minimal measure with the rotation vector $h$.
\end{proof}
\begin{Rem}
For $\alpha=1$, $\bG^{1,\bL}$ is the space of real analytic functions. This therefore means that the uniqueness of minimal measure is also a generic property in the analytic topology.
\end{Rem}

\begin{Cor}\label{corgeneric G1}
 Let $\bL>0$, $\alpha\geq 1$ and $L: T\T^n\times\T$ be a Tonelli Lagrangian. Then there exists a residual set $\mathcal{O}_1\subset\bG^{\alpha,\bL}(\T^n\times\T)$ such that for any $V\in\mathcal{O}_1$,  the Lagrangian $L+V$ has the following property: for each rational $h=(h_1,\cdots,h_n)\in H_1(\T^n,\R)$ with $h_i\in\Q$,  $L+V$ has one and only one  minimal measure with  the rotation vector $h$.
\end{Cor}
\begin{proof}
For each $h\in H_1(\T^n,\R)$, thanks to Theorem \ref{generic G1}, we obtain a residual subset $\mathcal{O}(h)\subset\bG^{\alpha,\bL}(\T^n\times\T)$ such that for each $ \phi\in\mathcal{O}(h)$, the Lagrangian $L+\phi$ has only one minimal measure with  the rotation vector $h$. Then we set 
\[\mathcal{O}_1=\bigcap_{h\in \Q^n} \mathcal{O}(h),\]
which is the intersection of countably many residual sets. Of course, by the definition of residual set, 
 $\mathcal{O}_1$ is non-empty and still residual (also dense) in the Banach space $\bG^{\alpha,\bL}(\T^n\times\T)$. The corollary is now evident from what we have proved.
\end{proof}

\subsection{H\"older regularity of elementary weak KAM solutions}\label{sub holder}
In this part, we will choose a family of elementary weak KAM solutions
which can be parameterized so that they are H\"older continuous in the $C^0$ topology. Such a property is crucial for our proof of Theorem \ref{generic G2}.

To this end, we need to do study the normally hyperbolic cylinders (refer to Appendix \ref{appendix_NHIC}). Let us go back to our two and a half degrees of freedom Hamiltonian model \eqref{hamiltonian}.  Let 
$$\Sigma(0):=\{(q_1,\bfo,p_1,\bfo): q_1\in\T, |p_1|\leq R\}\subset\T^2\times\R^2.$$
It is a cylinder restricted on the time-$0$ section, where $R$ is the constant fixed in Section \ref{introduction}. By condition {\bf(H2)}, $\Sigma(0)$ is a normally hyperbolic invariant cylinder (NHIC) for the time-1 map of the Hamiltonian flow $\Phi_{H_0}^t$. Since the  Hamiltonian $H_0$ is integrable when restricted in the cylinder $\Sigma(0)$, the rate $\mu$ in \eqref{hyp splitting} is 1 and $\log\mu=0$, so it follows from Theorem \ref{persistence} that there exists
\begin{equation}\label{diyigeepsilon}
\vep_1=\vep_1(H_0,R)>0
\end{equation}
such that if  $\|H_1\|_{C^3(\sD_R)}\leq\vep_1$, the time-1 map $\Phi_H^1$ of the Hamiltonian $H$ still admits a $C^{r-1}$ normally hyperbolic invariant cylinder $\Sigma_H(0)$, which is a small deformation of $\Sigma(0)$ and can be considered as the image of the following diffeomorphism (see figure \ref{crumpled cylin})
\begin{equation}\label{graph of cylinder}
  \begin{split}
    \psi:\Sigma(0)&\to \Sigma_H(0)\subset \T^2\times\R^2, \\
    (q_1,\bfo,p_1,\bfo)&\mapsto(q_1,\bq_2(q_1,p_1),p_1, \bp_2(q_1,p_1)).
   \end{split}
\end{equation}
Here, $\bq_2$ and $\bp_2$ are two $C^{r-1}$ functions taking values close to zero.
Then $\psi$ induces a 2-form $\psi^*\Omega$ on the standard cylinder $\Sigma(0)$ where $\Omega=\sum_{i=1}^2 dq_i\wedge dp_i$,
$$\psi^*\Omega=\bigg(1+\frac{\partial(\bq_2,\bp_2)}{\partial(q_1,p_1)}\bigg)dq_1\wedge dp_1.$$
Since the second de Rham cohomology group $H^2(\Sigma(0) ,\R)=\{0\}$, by using Moser's trick on the isotopy of symplectic forms, one can find a diffeomorphism $\psi_1:\Sigma(0)\to\Sigma(0)$ such that
\begin{equation*}
  \psi_1^*\psi^*\Omega=dq_1\wedge dp_1.
\end{equation*}
\begin{figure}[H]
  \centering
  \includegraphics[height=2.4cm]{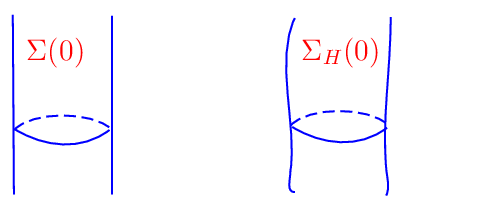}
  \caption{$\Sigma_H(0)$ is a small deformation of $\Sigma(0)$ }\label{crumpled cylin}
\end{figure}
Recall that $\Sigma_H(0)$ is invariant under $\Phi_H^1$ and $(\Phi_H^1)^*\Omega=\Omega$, one obtains
\begin{equation*}
  \big( (\psi\circ\psi_1)^{-1}\circ\Phi_H^1\circ (\psi\circ\psi_1) \big)^*dq_1\wedge dp_1=dq_1\wedge dp_1.
\end{equation*}
Combining with the fact that $(\psi\circ\psi_1)^{-1}\circ\Phi_H^1\circ (\psi\circ\psi_1)$ is a small perturbation of $\Phi_{H_0}^1$, the map $(\psi\circ\psi_1)^{-1}\circ\Phi_H^1\circ (\psi\circ\psi_1)$ is an exact twist map, and hence one can apply the classical Aubry-Mather theory to characterize the minimal orbits on $\Sigma(0)$: given a $\rho\in\R$, there exists an Aubry-Mather set with rotation number $\rho$ satisfying
\begin{enumerate}
  \item if $\rho\in\Q$, the set consists of  periodic orbits.
  \item if  $\rho\in\R\setminus \Q$, the set is either an invariant circle or a Denjoy set.
\end{enumerate}

For simplicity, we denote by
$$\Sigma_H(s)=\Phi^s_H(\Sigma_H(0),0), \quad \Sigma(s)=\Phi^s_{H_0}(\Sigma(0),0)$$
the 2-dimensional manifolds, and denote by
\begin{equation}\label{all_t}
	\widetilde{\Sigma}_H=\bigcup_{s\in\T}\Sigma_H(s),\quad\widetilde{\Sigma}=\bigcup_{s\in\T}\Sigma(s)
\end{equation}
the 3-dimensional manifolds in  $T^*\T^2\times\T$.  By using the Legendre transformation $\sL$ (see \eqref{Legendre_tran}),  the set $\sL\widetilde{\Sigma}_H$ is $\phi_L^t$-invariant  in $T\T^2\times\T$. Given a cohomology class $c=(c_1,\bfo)\in H^1(\T^2,\R)$ with $|c_1|\leq R-1$,  the following lemma shows that the Aubry set $\widetilde{\cA}(c)$ lies inside the cylinder $\sL\widetilde{\Sigma}_H$.

\begin{Lem}[Location of the minimal sets]\label{minimal set on cylinder}
Let $H$ be the  Hamiltonian \eqref{hamiltonian} and $L$ be the associated Lagrangian \eqref{lagrangian}. There exists $\vep_1=\vep_1(H_0,R)>0$ such that if $\|H_1\|_{C^3(\sD_R)}\leq \vep_1$ , then for each $c=(c_1, 0)$ with $|c_1|\leq R-1$, the globally minimal set $\widetilde{\cG}_L(c)\subset\sL\widetilde{\Sigma}_H$.
\end{Lem}
\begin{proof}
We first consider the autonomous Lagrangian $l_2(q_2,v_2)$.  It follows from \eqref{weiyimin} that $(\bfo,\bfo)$ is the unique minimal point of  $l_2$, so the globally minimal set of the Lagrangian $l_2$ is
$$\widetilde{\cG}_{l_2}=(\bfo,\bfo)\times\T\subset T\T\times\T.$$
Then for all $c=(c_1,\bfo)$ with $|c_1|\leq R$, the globally minimal  set of $L_0=l_1(v_1)+l_2(q_2,v_2)$ is
$$\widetilde{\cG}_{L_0}(c)=\{(q_1,\bfo,D h_1(c_1),\bfo,t)~:~ q_1\in\T,t\in\T\}  \textup{~and~} \widetilde{\cG}_{L_0}(c)\subset\sL\widetilde{\Sigma}.$$
Here, the function $h_1$ is given in \eqref{hamiltonian}.

Next,  we take a small neighborhood $U$ of $\sL\widetilde{\Sigma}$ in the space $T\T^2\times\T$ and let $\vep_1=\vep_1(H_0,R)$ be the constant defined in \eqref{diyigeepsilon}. Since $\|H_1\|_{C^3(\sD_R)}\leq\vep_1$, by letting $\vep_1$ suitably small, it follows that $\|L_1\|_{C^2(\sD_R)}$ is also sufficiently small. Thus, by the upper semi-continuity in  Proposition \ref{upper semi},  $\widetilde{\cG}_{L}(c)\subset U$ for all $c\in[-R+1,R-1]\times\{\bfo\}$ where $L=L_0+L_1$. Equivalently, 
$$\sL^{-1}\widetilde{\cG}_{L}(c)\subset \sL^{-1} U.$$

On the other hand, due to normal hyperbolicity and Theorem \ref{persistence}, 
$$\widetilde{\Sigma}_H\subset \sL^{-1} U,$$
provided that $\vep_1$ is small enough. 
Moreover, $\widetilde{\Sigma}_H$ is the unique $\phi^t_L$-invariant set in the neighborhood  $\sL^{-1} U$. This therefore implies  $\sL^{-1}\widetilde{\cG}_{L}(c)\subset\widetilde{\Sigma}_H$ since $\sL^{-1}\widetilde{\cG}_{L}(c)$ is  $\phi^t_L$-invariant.
\end{proof}

In the remainder of this section, we will use the following notation for simplicity.

\noindent\textbf{Notation:}
\begin{enumerate}[\rm(1)]
	\item In what follows,  we use $M$ to denote the manifold $\T^2=\R^2/\Z^2$. Also, we denote by
  $$\check{M}=\T\times2\T=\R/\Z\times\R/2\Z,\quad \check{\pi}:\check{M}\to M$$  the double covering of $M$. We use such a double covering to distinguish between $0$ and $1$ in the $q_2$-coordinate, and identify 0 with 2 in the $q_2$-coordinate.

  The Hamiltonian $H: T^*M\times\T\to\R$ and the Lagrangian $L:TM\times\T\to\R$  could  extend naturally to $T^*\check{M}$ and $T\check{M}$ respectively. By abuse of notation, we continue to write $H: T^*\check{M}\times\T\to\R$ and  $L: T\check{M}\times\T\to\R$ for the new Hamiltonian and Lagrangian respectively. In this setting, the lift of  the NHIC $\Sigma_H(0)$ will have two copies
  $$\check{\pi}^{-1}\Sigma_H(0)=\Sigma_{H,l}(0)\cup\Sigma_{H,u}(0), $$
where the subscripts $l, u$ are introduced to indicate ``lower" and ``upper" respectively. Then 
$\check{\pi}^{-1}\widetilde\Sigma_H=\widetilde\Sigma_{H,l}\cup\widetilde\Sigma_{H,u}$.
\item For simplicity,  we always use $\pi_q$ to denote the natural projection from $T\check{M}$ (resp. $TM$) to $\check{M}$ (resp. $M$) or from $T^*\check{M}$ (resp. $T^*M$) to $\check{M}$ (resp. $M$). 
\item Let $\kappa>0$ be small, we denote by $\mathrm{U}_\kappa=\mathrm{U}_{\kappa,l}\cup\mathrm{U}_{\kappa,u}$ the disconnected subset in $\check{M}$ where
$$ \mathrm{U}_{\kappa,l}=\T\times[\kappa, 1-\kappa],\quad\mathrm{U}_{\kappa,u}=\T\times[1+\kappa, 2-\kappa].$$
Let $\rN_{\kappa}=\check{M}\setminus\mathrm{U}_\kappa=\mathrm{N}_{\kappa,l}\cup\mathrm{N}_{\kappa,u}$ where 
$$\rN_{\kappa,l}=\T\times(-\kappa, \kappa),\quad\rN_{\kappa,u}=\T\times(1-\kappa, 1+\kappa).$$
The subscripts $l, u$ are also introduced to indicate the ``lower" and the ``upper" respectively (See figure \ref{picture2}). The number $\kappa$ should be chosen such that
      $$\pi_q\circ\Sigma_{H,l}(0)\subset\rN_{\kappa/2,l},\qquad\pi_q\circ\Sigma_{H,u}(0)\subset\rN_{\kappa/2,u}.$$
      Namely, the perturbed cylinder  is contained in a $\kappa/2$-neighborhood of the unperturbed one. Also,
 we let
 \begin{equation}\label{c_t1}
 	\widetilde{\Sigma}_{H,l}\subset\rN_{\kappa/2,l}\times\R^2\times\T,\quad\widetilde{\Sigma}_{H,u}\subset\rN_{\kappa/2,u}\times\R^2\times\T. 
 \end{equation}  
      
\item For $c=(c_1, \bfo)$$\in H^1(M,\R)$, if the Aubry set $\widetilde{\cA}_{L}(c,M)|_{t=0}$ is an invariant circle, we denote by $$\Upsilon_c=\sL^{-1}\widetilde{\cA}_{L}(c,M)|_{t=0}\subset T^*M\times\{t=0\}$$ the invariant circle in the cotangent space. This leads us to introduce an index set
      \begin{equation}\label{buianquandeshangtongdiao}
     \mathbb{S}:=\{ (c_1,\bfo)~:~ |c_1|\leq R-1, ~\Upsilon_c \textup{~is an invariant circle lying in~} \Sigma_H(0)\}.
      \end{equation}
 \item Let $\rro>0$ be  small satisfying $\rro>\kappa$.  Since $\Sigma_{H}(0)$ is a NHIM for the time-1 map $\Phi_H^1$, 
  we have the associated local stable and unstable manifolds, denoted by $W_{\Sigma_{H}(0)}^{s,loc}$ and $W_{\Sigma_{H}(0)}^{u,loc}$ respectively, in the $\rro$-tubular neighborhood of $\Sigma_H(0)$. In addition,  $W^{s,loc}_{\Sigma_{H}(0)}=\bigcup_{q\in\Sigma_{H}(0)}W^{s,loc}_q$ and $ W^{u,loc}_{\Sigma_{H}(0)}=\bigcup_{q\in\Sigma_{H}(0)}W^{u,loc}_q$.
\end{enumerate}

Now, let us focus on  $c=(c_1, \bfo)\in\mathbb{S}$. For the $\Phi^1_H$-invariant circle $\Upsilon_c$, it has local stable manifold 
$W^{s,loc}_{\Upsilon_c}=\bigcup_{q\in\Upsilon_c}W^{s,loc}_q$ and local unstable manifold $ W^{u,loc}_{\Upsilon_c}=\bigcup_{q\in\Upsilon_c}W^{u,loc}_q $.  Theorem \ref{property of NHIM} tells us that the leaf $W^{s,loc}_q$ (resp. $W^{u,loc}_q$) has smooth dependence on the base point $q\in\Sigma_H(0)$. Consequently,   $W^{s,loc}_{\Upsilon_c}$  (resp. $W^{u,loc}_{\Upsilon_c}$) is a Lipschitz manifold since  $\Upsilon_c$ is only  Lipschitz  in general. Besides,  the local stable (unstable) manifold can be viewed as a Lipschitz graph over $\check\pi\circ\rN_\rro$, namely
\begin{equation*}
  \begin{split}
   W^{s,loc}_{\Upsilon_c}&=\{ \big(q_1,q_2, \bp_1^s(q_1,q_2), \bp_2^s(q_1,q_2)\big)\in T^*M\times\{t=0\}: (q_1,q_2)\in\check\pi\circ\rN_\rro \}\\
   W^{u,loc}_{\Upsilon_c}&=\{ \big(q_1,q_2, \bp_1^u(q_1,q_2), \bp_2^u(q_1,q_2)\big)\in T^*M\times\{t=0\}: (q_1,q_2)\in\check\pi\circ\rN_\rro  \}
   \end{split}
\end{equation*}
Here, $\bp_1^{s,u}, \bp_2^{s,u}$ are Lipschitz functions on $\check\pi\circ\rN_\rro\subset M$, and the domain $\rN_{\rro}=\mathrm{N}_{\rro,l}\cup\mathrm{N}_{\rro,u}$ with
\[\rN_{\rro,l}=\T\times(-\rro, \rro),\quad\rN_{\rro,u}=\T\times(1-\rro, 1+\rro).\]

Next, in the  covering space $\check{M}$, the Aubry set $\widetilde{\cA}_{L}(c,\check{M})$ is the union of two disjoint copies of $\widetilde{\cA}_{L}(c,M)$ satisfying $\check{\pi}\widetilde{\cA}_{L}(c,\check{M})=\widetilde{\cA}_{L}(c,M)$. More precisely,  $\sL^{-1}\widetilde{\cA}_{L}(c,\check{M})\big|_{t=0}=\Upsilon_{c,l}\cup\Upsilon_{c,u}$, where $\Upsilon_{c,\imath}$ lies in $\Sigma_{H,\imath}(0)$
and its stable and unstable manifolds are
\begin{equation*}
  \begin{split}
   W^{s,loc}_{\Upsilon_{c,\imath}}&=\{ \big(q_1, q_2, \bp_1^s(q_1, q_2), \bp_2^s(q_1, q_2)\big)\in T^*\check{M}\times\{t=0\}: (q_1, q_2)\in\rN_{\rro,\imath} \}\\
   W^{u,loc}_{\Upsilon_{c,\imath}}&=\{ \big(q_1, q_2, \bp_1^u(q_1, q_2), \bp_2^u(q_1, q_2)\big)\in T^*\check{M}\times\{t=0\}: (q_1, q_2)\in\rN_{\rro,\imath} \}
   \end{split}
\end{equation*}
with $\imath=l, u$. Here, by abuse of notation, we have continued to use $\bp_1^{s,u}, \bp_2^{s,u}$ to denote the corresponding Lipschitz functions defined on the lift of $\check\pi\circ\rN_\rro$.
The lemma below gives the relation between the elementary weak KAM solutions and the local stable  and unstable manifolds.
\begin{Lem}\label{local manifolds representation}
There exists $\rro>0$ such that for each $c=(c_1, \bfo)\in\mathbb{S}$, we have
\begin{enumerate}[\rm(1)]
  \item for each backward elementary weak KAM solution $u^-_{c, \imath}(q,t)$ with $\imath=l, u$,   the function $u^-_{c, \imath}(q,0)$ is $C^{1,1}$ in the domain $\rN_{\rro,\imath}$ and generates the local unstable manifold of $\Upsilon_{c,\imath}$, i.e.
$$ W^{u,loc}_{\Upsilon_{c,\imath}}=\{ \big(q, c+\partial_qu^-_{c, \imath}(q,0)\big): q\in\rN_{\rro,\imath} \},\qquad \imath=l, u.$$
  \item for each forward elementary weak KAM solution $u^+_{c, \imath}(q,t)$ with $\imath=l, u$,   the function $u^+_{c, \imath}(q,0)$ is $C^{1,1}$ in the domain $\rN_{\rro,\imath}$ and generates the local stable manifold of $\Upsilon_{c,\imath}$, i.e.
$$ W^{s,loc}_{\Upsilon_{c,\imath}}=\{ \big(q, c+\partial_qu^+_{c, \imath}(q,0)\big): q\in\rN_{\rro,\imath} \},\qquad \imath=l, u.$$
\end{enumerate}
\end{Lem}
\begin{proof}
We only prove for the case $u^-_{c, l}$ since the other cases are similar.\\
\textbf{Step 1:} We first claim that there exists a neighborhood $V$ of $\pi_q\circ\Upsilon_{c,l}$ in $\check{M}$ such that for each $\xi^-: (-\infty, 0]\to \check{M}$ calibrated by $u^-_{c,l}$ with $\xi^-(0)\in V$, the $\alpha$-limit set of the backward minimal configuration $\{\xi^-(-i)\}_{i\in\Z^+}$ must be contained in $\pi_q\circ\Upsilon_{c,l}$.

Assume by contradiction that there exist a sequence of backward calibrated curves $\xi^-_k:(-\infty, 0]$ $\to\check{M}$ with $\xi_k^-(0)=x_k$, and a sequence $\alpha_k$ which belongs to the $\alpha$-limit set of the backward minimal configuration $\{\xi_k^-(-i)\}_{i\in\Z^+}$ satisfying
\begin{equation}\label{class1}
  \lim\limits_{k\to\infty}x_k=x^*\in\pi_q\circ\Upsilon_{c,l}\quad\textup{and}\quad \lim\limits_{k\to\infty}\alpha_k=\alpha^*\notin\pi_q\circ\Upsilon_{c,l}.
\end{equation}
This implies $\alpha^*\in\pi_q\circ\Upsilon_{c,u}$ since the $\alpha$-limit set of each minimal curve shall be contained in the Aubry set. By Theorem \ref{representation of EWS}, each $\xi^-_k: (-\infty, 0]\to\check{M}$ is $c$-semi static and calibrated by $h^\infty_c((x^*,0),\cdot)$:
$$h^\infty_c\big((x^*,0),(\xi^-_k(0),0)\big)-h^\infty_c\big((x^*,0),(\xi^-_k(-t),-t)\big)=h_c^{-t,0}\big(\xi^-_k(-t),\xi^-_k(0)\big),\quad\forall t\in\Z^+.$$
This further gives $h^\infty_c\big((x^*,0),(x_k,0)\big)-h^\infty_c\big((x^*,0),(\alpha_k,0)\big)\geq h^\infty_c\big((\alpha_k,0),(x_k,0)\big).$ The opposite inequality is obvious, 
therefore $h^\infty_c((x^*,0),(x_k,0))-h^\infty_c((x^*,0),(\alpha_k,0))=$ $ h^\infty_c((\alpha_k,0),(x_k,0))$.
Sending $k\to\infty$, it follows that
$$0=h^\infty_c\big((x^*,0),(x^*,0)\big)=h^\infty_c\big((x^*,0),(\alpha^*,0)\big)+ h^\infty_c\big((\alpha^*,0),(x^*,0)\big).$$
Hence, $(x^*,0)$ and $(\alpha^*,0)$ belong to the same Aubry class, which contradicts  \eqref{class1}.\\
\textbf{Step 2:}  By letting the above domain $V$ be suitably small if necessary,   $W^{u}_{\Upsilon_{c,l}}$ is a Lipschitz graph over  $V$,  we will show that there exists a small number $\rro>0$ such that $\rN_{\rro,l}\subset V$, and each  $u^-_{c,l}$-calibrated curve $\gamma^-: (-\infty, 0]\to \check{M}$ with $\gamma^-(0)\in\rN_{\rro,l}$  satisfies $\gamma^-(-m)\in V$, $\forall~m\in\N$.

To prove this, assume by contradiction that there exist a sequence of $u^-_{c,l}$-calibrated curves $\gamma^-_j: (-\infty, 0]$ $\to$ $\check{M}$ and a sequence of  positive integers $T_j$ such that  $\gamma_j^-(-T_j)\notin V,$
$\gamma_j^-(-m)\in V,$ $m\in\{0, 1, \cdots$, $T_j-1\}$ and $\lim_{j\to\infty}\textup{dist}(\gamma^-_j(0), \pi_q\circ\Upsilon_{c,l})=0$.

We set  $\eta^-_j(t):=\gamma_j^-(t-T_j)$, then $\eta^-_j: (-\infty, T_j]\to \check{M}$ is still a calibrated curve and
\begin{equation}\label{contain1}
\eta^-_j(0)\notin V,~ \eta_j^-(m)\in V, ~ m\in\{1,2,\cdots,T_j\}
\end{equation}
and
\begin{equation}\label{a sequence of curves}
 \lim\limits_{j\to\infty}\textup{dist}(\eta^-_j(T_j), \pi_q\circ\Upsilon_{c,l})=0.
\end{equation}
Extracting a subsequence if necessary, we  suppose that $(\eta^-_j(t), \dot{\eta}^-_j(t))$ converges uniformly to a limit curve $(\eta^-(t), \dot{\eta}^-(t)): I\to\check{M}$  on any compact interval of $\R$. Here,  the interval $I$ is either $(-\infty, T]$ or $\R$ ($T$ is a positive integer). Obviously, $\eta^-(t)$ is still calibrated by $u^-_{c,l}$ and
\begin{equation}\label{initial value of limit curve}
 \eta^-(0)\notin V.
\end{equation}

In the case  $I=(-\infty, T]$, one obtains $\eta^-(T)\in\pi_q\circ\Upsilon_{c,l}$ as a consequence of \eqref{a sequence of curves}, and hence  $\{\eta^-(m)\}_{m\in\Z, m\leq T}\subset\pi_q\circ\Upsilon_{c,l}$, which contradicts  \eqref{initial value of limit curve}. In the case  $I=\R$,  it follows from \eqref{contain1} that the $\omega$-limit set of $\{\eta^-(m)\}_{m\in\Z}$ lies in $\pi_q\circ\Upsilon_{c,l}$, then $\{\eta^-(m)\}_{m\in\Z}\subset\pi_q\circ\Upsilon_{c,l}$ since the $\omega$-limit set and $\alpha$-limit set belong to the same Aubry class, which contradicts  \eqref{initial value of limit curve}. \\
\textbf{Step 3:}  By what we have proved above, for each $u^-_{c,l}$-calibrated curve $\gamma^-$ with $\gamma^-(0)\in\rN_{\rro,l}$,  $\{(\gamma^-(-m)$, $\dot{\gamma}^-(-m))\}_{i\in\Z^+}$  would always stay in a small neighborhood of the cylinder $\sL\Sigma_{H,l}(0)$. This also means that the $\alpha$-limit set of $\{\sL^{-1}\big(\gamma^-(-m)$, $\dot{\gamma}^-(-m)\big)\}_{m\in\Z^+}$ lies in $\Upsilon_{c,l}$. By normal hyperbolicity, $\{\sL^{-1}\big(\gamma^-(-m)$, $\dot{\gamma}^-(-m)\big)\}_{m\in\Z^+}$ $\subset$ $W^{u,loc}_{\Upsilon_{c,l}}$. Thus, for each $q\in \rN_{\rro,l}$, there is a unique $u^-_{c,l}$-calibrated  curve $\gamma^-: (-\infty, 0]\to\check{M}$  with $\gamma^-(0)=q$ since $W^{u,loc}_{\Upsilon_{c,l}}$ is a Lipschitz graph over $\rN_{\rro,l}$ $\subset V$. By weak KAM theory, $u^-_{c,l}$ is therefore $C^{1,1}$ in  $\rN_{\rro,l}$. Moreover,  Proposition \ref{properties weak KAM} implies that
$$(q, c+\partial_qu^-_{c,l}(q,0))=\sL^{-1}\big(\gamma^-(0),\dot{\gamma}^-(0)\big)\in W^{u,loc}_{\Upsilon_{c,l}}.$$
This completes the proof.
\end{proof}

In \cite{CY2004}, the authors introduced an``area" parameter $\sigma$ to parameterize an invariant circle lying on the NHIC so that the invariant circle  $\Gamma_\sigma$ is $\frac{1}{2}$-H\"older continuous with respect to $\sigma$, namely
$$\|\Gamma_{\sigma_1}-\Gamma_{\sigma_2}\|_{C^0}\leq C|\sigma_1-\sigma_2|^{\frac{1}{2}}$$
However, this result can be improved by taking advantage of the tools in  weak KAM theory. Roughly speaking,  the ``area" parameter $\sigma$ is, to some extent, the cohomology class $c$ (see  Lemma \ref{local11} and Theorem \ref{global regularity of elementary solutions} below). Similar results could be found in \cite{BKZ2016}. 

Recall that the invariant circle $\Upsilon_{c,\imath}$, where $c\in\mathbb{S}$ and $\imath=l,u$, can be viewed as a Lipschitz graph over $q_1$. More precisely, by abuse of notation, we continue to write $\Upsilon_{c,\imath}$  for this Lipschitz function
$$ \Upsilon_{c,\imath}: \T\longrightarrow\Sigma_{H,\imath}(0)\subset \T^2\times\R^2$$
 $$q_1\longmapsto(q_1, ~\pi_{q_2}\circ\Upsilon_{c,\imath}(q_1), ~\pi_{p_1}\circ\Upsilon_{c,\imath}(q_1),~ \pi_{p_2}\circ\Upsilon_{c,\imath}(q_1))$$
with $\pi_{q_1}\circ\Upsilon_{c,\imath}(q_1)=q_1$ and $\imath=l, u$.  Then, we have:
\begin{Lem}[$\frac{1}{2}$-H\"older regularity]\label{local11}
There exists a positive constant $C$ such that for any $c, c'\in\mathbb{S}$,
\begin{enumerate}[\rm(1)]
  \item $\max_{q_1}\|\Upsilon_{c,l}(q_1)-\Upsilon_{c',l}(q_1)\|\leq C\|c-c'\|^{\frac{1}{2}},$
  \item $\max_{q_1}\|\Upsilon_{c,u}(q_1)-\Upsilon_{c',u}(q_1)\|\leq C\|c-c'\|^{\frac{1}{2}}.$
\end{enumerate}
\end{Lem}
\begin{proof}
We only prove item (1) and the other one is similar. Recall that Lemma \ref{local manifolds representation}  tells us that the elementary weak KAM solution $u^-_{c,l}$ is $C^{1,1}$ in $\rN_{\rro,l}$. Now, in the 4-dimensional space $T^*\rN_{\rro,l}=\rN_{\rro,l}\times\R^2$,
we define two 1-forms $\omega_1$ $=\big(c_1+\partial_{q_1}u^-_{c,l}(q,0)\big)dq_1$ $+\partial_{q_2}u^-_{c,l}(q,0)dq_2$ and $\omega_2=p_1dq_1+p_2 dq_2$. Note that  $\omega_1|_{\Upsilon_{c,l}}=\omega_2|_{\Upsilon_{c,l}}$ as a consequence of Proposition \ref{properties weak KAM} and Lemma \ref{local manifolds representation}. Then,
\begin{equation*}
\int_{\Upsilon_{c,l}}\omega_2=\int_{\Upsilon_{c,l}}\omega_1=\int_{\Upsilon_{c,l}}[c_1dq_1+du^-_{c,l}(q,0)]=\int_{\Upsilon_{c,l}}c_1dq_1=c_1.
\end{equation*}
For $c, c'\in\mathbb{S}$, we may assume $c'_1>c_1$. Let $D$ be the region on the cylinder $\Sigma_{H,l}(0)$ between $\Upsilon_{c,l}$ and $\Upsilon_{c',l}$ (see figure \ref{regionD}). By Stoke's theorem,
\begin{equation}\label{stoke formula}
  \int_{D}\sum\limits_{i=1}^2 dp_i\wedge dq_i=\int_{\Upsilon_{c,l}}\omega_2-\int_{\Upsilon_{c',l}}\omega_2=c_1-c'_1.
\end{equation}
\begin{figure}[H]
  \centering
  \includegraphics[width=4.4cm]{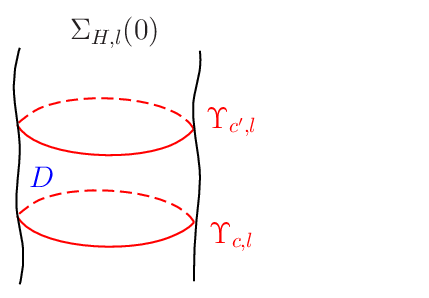}
  \caption{The region $D$ is bounded by two invariant circles}\label{regionD}
\end{figure}
Then \eqref{graph of cylinder} and \eqref{stoke formula} together imply 
\begin{equation}\label{area form estimate}
\begin{split}
|c_1-c_1'|&=\left|\int_{D}\sum\limits_{i=1}^2 dp_i\wedge dq_i\right|=\left|\int_{D}\left(1+\frac{\partial(\bp_{2},\bq_{2})}{\partial(p_1,q_1)}\right)dp_1\wedge dq_1\right|\\
&\geq\frac{1}{4}\left|\int_{D}dp_1\wedge dq_1\right|=\frac{1}{4}\left|\int_{\Upsilon_{c,l}}p_1dq_1-\int_{\Upsilon_{c',l}}p_1dq_1\right|\\
&=\frac{1}{4}\left|\int_{\T}\pi_{p_1}\circ\Upsilon_{c,l}(q_1)-\pi_{p_1}\circ\Upsilon_{c',l}(q_1)~dq_1\right|
\end{split}
\end{equation}
As the Lipschitz functions $\pi_{p_1}\circ\Upsilon_{c,l}$, ~ $\pi_{p_1}\circ\Upsilon_{c',l}:\T\rightarrow\R$ satisfy $\pi_{p_1}\circ\Upsilon_{c',l}> \pi_{p_1}\circ\Upsilon_{c,l}$, we have
\begin{equation}\label{holder regularity of Upsilon}
 \int_{\T}\pi_{p_1}\circ\Upsilon_{c',l}(q_1)-\pi_{p_1}\circ\Upsilon_{c,l}(q_1)~dq_1\geq\frac{1}{4C_L}\left(\max\limits_{q_1}|\pi_{p_1}\circ\Upsilon_{c',l}(q_1)-\pi_{p_1}\circ\Upsilon_{c,l}(q_1)|\right)^2,
\end{equation}
where $C_L$ is the Lipschitz bound of the functions $\pi_{p_1}\circ\Upsilon_{c,l}$ and $\pi_{p_1}\circ\Upsilon_{c',l}$.

Recall that the function $\bp_2(q_1,p_1)$ is at least $C^1$, then there exists a constant $K>0$ such that
\begin{equation}\label{norm estimate}
\begin{split}
 &\|\pi_{p}\circ\Upsilon_{c,l}(q_1)-\pi_{p}\circ\Upsilon_{c',l}(q_1)\|\\
 =& |\pi_{p_1}\circ\Upsilon_{c,l}(q_1)-\pi_{p_1}\circ\Upsilon_{c',l}(q_1)|+|\pi_{p_2}\circ\Upsilon_{c,l}(q_1)-\pi_{p_2}\circ\Upsilon_{c',l}(q_1)|\\
=&|\pi_{p_1}\circ\Upsilon_{c,l}(q_1)-\pi_{p_1}\circ\Upsilon_{c',l}(q_1)|+|\bp_2(q_1,\pi_{p_1}\circ\Upsilon_{c,l}(q_1))-\bp_2(q_1,\pi_{p_1}\circ\Upsilon_{c',l}(q_1))|\\
\leq&(1+K)|\pi_{p_1}\circ\Upsilon_{c,l}(q_1)-\pi_{p_1}\circ\Upsilon_{c',l}(q_1)|,
\end{split}
\end{equation}
Thus, combining \eqref{area form estimate}, \eqref{holder regularity of Upsilon} with \eqref{norm estimate}, one obtains
\begin{equation*}
\begin{split}
\|c-c'\|&\geq|c_1-c_1'|\geq\frac{1}{16C_L}\big(\max\limits_{q_1}|\pi_{p_1}\circ\Upsilon_{c,l}(q_1)-\pi_{p_1}\circ\Upsilon_{c',l}(q_1)|\big)^2\\
&\geq\frac{1}{16C_L(1+K)^{2}}\big(\max\limits_{q_1}\|\pi_{p}\circ\Upsilon_{c,l}(q_1)-\pi_{p}\circ\Upsilon_{c',l}(q_1)\|\big)^2,
\end{split}
\end{equation*}
which implies 
$$\max\limits_{q_1}\|\pi_{p}\circ\Upsilon_{c,l}(q_1)-\pi_{p}\circ\Upsilon_{c',l}(q_1)\|\leq4\sqrt{C_L}(1+K)\|c-c'\|^\frac{1}{2}. $$

Next, since the function $\bq_2(q_1, p_1)$ is at least $C^1$,  there exists a constant $\tilde{C}>0$ such that
\begin{align*}
	\max\limits_{q_1}|\pi_{q_2}\circ\Upsilon_{c,l}(q_1)-\pi_{q_2}\circ\Upsilon_{c',l}(q_1)|& =\max\limits_{q_1}|\bq_2(q_1,\pi_{p_1}\circ\Upsilon_{c,l}(q_1))-\bq_2(q_1,\pi_{p_1}\circ\Upsilon_{c',l}(q_1))\\
	& \leq \widetilde{C}\max\limits_{q_1}|\pi_{p_1}\circ\Upsilon_{c,l}(q_1)-\pi_{p_1}\circ\Upsilon_{c',l}(q_1)|\\
	&\leq 4\sqrt{C_L}\widetilde{C}(1+K)\|c-c'\|^\frac{1}{2}.
\end{align*}
Consequently, item (1) follows immediately by  
 setting $C=4\sqrt{C_L}(1+\widetilde{C})(1+K)$.
\end{proof}

We also mention that, for the Peierls barriers restricted on the NHIC, one can even obtain the H\"older continuity  with respect to perturbations \cite{CC2017}. 

The result below is analogous to \cite[Lemma 6.4]{CY2009} and will be crucial for the proof of genericity. 

\begin{The}\label{global regularity of elementary solutions}
Let $\rro$ be the constant given in Lemma \ref{local manifolds representation}, and we fix two points $z_l\in\rN_{\rro,l}, z_u\in\rN_{\rro,u}$. Let $u^\pm_{c,l}(q,t)$, $u^\pm_{c,u}(q,t)$ be the elementary weak KAM solutions satisfying $u^{\pm}_{c,l}(z_l,0)=u^{\pm}_{c,l}(z_u,0)$ $\equiv \textup{constant}$, for all $c\in\mathbb{S}$. Then there exists $C_h>0$ such that for any $c, c'\in\mathbb{S}$
$$|u^{\pm}_{c,l}(q,0)-u^\pm_{c',l}(q,0)|\leq C_h(\|c'-c\|^{\frac{1}{2}}+\|c'-c\|),\quad\forall q\in\check{M}\setminus \rN_{\rro,u}$$
and
$$|u^{\pm}_{c,u}(q,0)-u^\pm_{c',u}(q,0)|\leq C_h(\|c'-c\|^{\frac{1}{2}}+\|c'-c\|), \quad\forall q\in\check{M}\setminus \rN_{\rro,l}.$$
\end{The}

\begin{Rem}
By adding suitably constants, we can take $u^{\pm}_{c,l}(z_l,0)=u^{\pm}_{c,l}(z_u,0)=0$ for all $c\in\mathbb{S}$, since any elementary weak KAM solution  plus a constant is still an elementary weak KAM solution.
\end{Rem}

\begin{proof}
We only prove the case for $u^-_{c,l}$ and the others are similar. The normal hyperbolicity guarantees the smooth dependence of the unstable leaf $W_q^{u,loc}$ with respect to the base point $q\in\Sigma_{H,l}(0)$. By Lemma \ref{local11}, the local unstable manifold $W^{u, loc}_{\Upsilon_{c,l}}$ of $\Upsilon_{c,l}$ is also $\frac{1}{2}-$H\"{o}lder continuous in $c\in\mathbb{S}$. Then,  Lemma  \ref{local manifolds representation} implies that some constant $C_1>0$ exists  such that
\begin{equation*}
 \| \big(c+\partial_qu_{c,l}^-(q,0)\big)-\big(c'+\partial_qu_{c',l}^-(q,0)\big) \|\leq C_1\|c-c'\|^{\frac{1}{2}}, ~\forall q\in\rN_{\rro,l},  ~\forall c, c'\in\mathbb{S}.
\end{equation*}
Further, using integration we obtain that for all $c, c'\in\mathbb{S}$ and all $ q\in\rN_{\rro, l}$,
\begin{equation*}
\begin{split}
\big| \big( u^{-}_{c,l}(q,0)- u^{-}_{c,l}(z_l,0)+\langle c, q-z_l\rangle \big)-\big( u^-_{c',l}(q,0)-u^-_{c',l}(z_l,0)+\langle c', q-z_l\rangle \big)  \big|\leq C_1\|c-c'\|^{\frac{1}{2}}.
\end{split}
\end{equation*}
Since we have chosen $u^{-}_{c,l}(z_l,0)\equiv constant$ for all $c\in\mathbb{S}$, we get that $\forall$$c, c'\in\mathbb{S}$ and $\forall q\in\rN_{\rro, l}$
\begin{equation}\label{local regularity of sol}
\begin{split}
\big| u^{-}_{c,l}(q,0)- u^-_{c',l}(q,0) \big|\leq C_1\|c-c'\|^{\frac{1}{2}}+\|c-c'\|.
\end{split}
\end{equation}

Next, for each $z\in\check{M}\setminus \rN_{\rro,u}$, there exists a backward calibrated curve $\gamma^-_{c,l}$ with $\gamma_{c,l}^-(0)=z$, which is negatively asymptotic to $\pi_q\circ\Upsilon_{c,l}$. Since the duration of $\gamma^-_{c,l}$ staying outside  $\rN_{\rro,l}$ is uniformly bounded, denoted by $T_l\in\Z^+$, we have $\gamma^-_{c,l}(-k)\in \rN_{\rro,l}$ for every integer $k\geq T_{l}$. Then,
$$u^-_{c,l}(\gamma^-_{c,l}(0),0)-u^-_{c,l}(\gamma^-_{c,l}(-T_l),-T_l)=\int^0_{-T_l}L(\gamma^-_{c,l}(s),\dot{\gamma}^-_{c,l}(s),s)-\langle c, \dot{\gamma}^-_{c,l}(s)\rangle+\alpha(c)\, ds,$$
$$u^-_{c',l}(\gamma^-_{c,l}(0),0)-u^-_{c',l}(\gamma^-_{c,l}(-T_l),-T_l)\leq\int^0_{-T_l}L(\gamma^-_{c,l}(s),\dot{\gamma}^-_{c,l}(s),s)-\langle c', \dot{\gamma}^-_{c,l}(s)\rangle+\alpha(c')\,ds.$$
Subtracting the first formula from the second one, one deduces from inequality  \eqref{local regularity of sol} that 
\begin{equation*}
\begin{split}
u^-_{c',l}(z,0)-u^-_{c,l}(z,0)\leq & u^-_{c',l}(\gamma^-_{c,l}(-T_l),-T_l)-u^-_{c,l}(\gamma^-_{c,l}(-T_l),-T_l)\\
&+\int^0_{-T_l}\langle c-c', \dot{\gamma}^-_{c,l}(s)\rangle+\alpha(c')-\alpha(c)\, ds\\
\leq & u^-_{c',l}(\gamma^-_{c,l}(-T_l),0)-u^-_{c,l}(\gamma^-_{c,l}(-T_l),0)+C_2\|c'-c\|\\
\leq &C_1\|c'-c\|^{\frac 12}+\|c'-c\|+C_2\|c'-c\|.
\end{split}
\end{equation*}
Here, the second inequality follows from the  fact that $\|\dot{\gamma}^-_{c,l}\|$ is uniformly bounded  and Mather's $\alpha$-function  is Lipschitz continuous. So we conclude that there exists $C_h>0$ such that
$$u^-_{c',l}(z,0)-u^-_{c,l}(z,0)\leq C_h(\|c'-c\|^{\frac{1}{2}}+\|c'-c\|), ~\forall z\in\check{M}\setminus \rN_{\rro,u}.$$
In a similar way, we can prove that $u^-_{c,l}(z,0)-u^-_{c',l}(z,0)$$\leq C_h(\|c'-c\|^{\frac{1}{2}}+\|c'-c\|)$ for all $z\in\check{M}\setminus \rN_{\rro,u}$,
which completes the proof.
\end{proof}

\subsection{Choice of the Gevrey space }\label{determin of coeff}
In what follows, we assume  $\alpha>1$ and $M=\T^2$.  As we will see later, our proof of genericity  is not always valid for all Gevrey space $\bG^{\alpha,\bL}$ ($\bL>0$), but only for $\bG^{\alpha,\bL}$ with $\bL$ bounded by a positive constant $\bL_0$. This is caused by Gevrey approximation, and we will explain it and show how to choose $\bL_0$ below.

Let us first look at the unperturbed Lagrangian $L_0=l_1(v_1)+l_2(q_2,v_2)$ in \eqref{lagrangian}. For each $c=(c_1,\bfo)$, $|c_1|\leq R-1$, the  Aubry set and Ma\~n\'e set are 
$$\widetilde{\cA}_{L_0}(c,M)=\widetilde{\cN}_{L_0}(c,M)=\{(q_1, \bfo, D h_1(c_1), \bfo, t)\in TM\times\T: q_1\in\T, t\in\T\},$$
which is a $\phi_{L_0}^1$-invariant circle. Next, we work in the covering space $\check{M}$ and consider $L_0: T\check{M}\to\R$. Restricted on the time section $\{t=0\}$, the lift of the Aubry set has two copies:
\begin{equation*}
\begin{split}
\widetilde{\cA}_{L_0,l}(c,\check{M})|_{t=0}&=\{ (q_1, 0,D h_1(c_1),0)\in T\check{M}: q_1\in\T \},\\
\widetilde{\cA}_{L_0,u}(c,\check{M})|_{t=0}&=\{ (q_1, 1,D h_1(c_1),0)\in T\check{M}: q_1\in\T \},
\end{split}
\end{equation*}
and they lie on the following two invariant cylinders respectively
\begin{equation*}
\begin{split}
  \sL\Sigma_l(0)=&\{(q_1, 0, D h_1(p_1), \bfo)\in T\check{M}: q_1\in\T, |p_1|\leq R-1 \}, \\
  \sL\Sigma_u(0)=&\{(q_1, 1, D h_1(p_1), \bfo)\in T\check{M}: q_1\in\T, |p_1|\leq R-1 \}.
\end{split}
\end{equation*}
Notice that $\pi_q\circ\sL\Sigma_l(0)=\T\times\{0\}$ and $\pi_q\circ\sL\Sigma_u(0)=\T\times\{1\}$.
\begin{figure}[H]
    \centering
    \includegraphics[width=7.5cm]{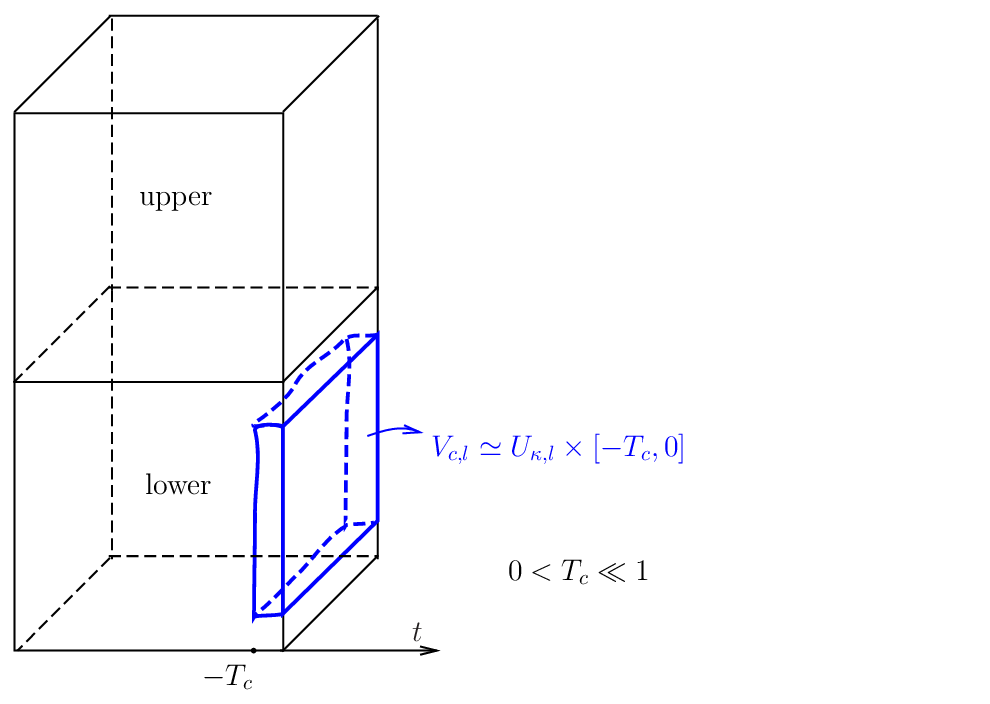}
    \caption{$V_{c,l}$ (blue) in a fundamental domain of $\check{M}\times\T$}\label{tubular}
\end{figure}

Denote by $u^\pm_{c,l,L_0}, u^\pm_{c,u,L_0}$ the elementary weak KAM solutions of $L_0$ with respect to the cohomology class $c$. Recall that $\kappa<\rro$.
For each $x\in \mathrm{U}_{\kappa,l}$, there exists a unique  $u^-_{c,l,L_0}$-calibrated curve $\xi^-_{x,c}(t): (-\infty, 0]\to\check{M}$  such that $\xi^-_{x,c}(0)=x$, and it is negatively asymptotic to $\cA_{L_0,l}(c)$. We pick and fix a constant $T_c=T_c(\kappa,L_0)>0$ small enough, then we obtain a local neighborhood $$V_{c,l}=\{(\xi_{x,c}^-(t),t)\in\check{M}\times\T: x\in \mathrm{U}_{\kappa,l}, -T_c\leq t\leq 0\}$$ which is diffeomorphic to $\mathrm{U}_{\kappa,l}\times [-T_c, 0]$ (see figure \ref{tubular}), namely there is a diffeomorphism $$f: \mathrm{U}_{\kappa,l}\times [-T_c, 0]\to V_{c,l}$$
such that $f(x,t)=(\xi^-_{x,c}(t),t)$ and $V_{c,l}\cap(\rN_{3\kappa/4,l}\times\T)=\emptyset$, this is guaranteed by $T_c\ll 1$. Notice that $V_{c,l}$ would vary in $c$.

Recall that $\check{M}=\T\times[0,2]/\sim$, where the equivalence relation $\sim$ is defined by identifying 0 with 2 in the $q_2$-coordinate. In the sequel, we will fix, once and for all, a sufficiently small constant $\delta>0$, which is smaller than $\kappa/4$. Thanks to Theorem \ref{Gevrey approx} there exists a Gevrey-$(\alpha, \lambda_c)$  diffeomorphism
$$\Psi_{c,l}: \mathrm{U}_{\kappa,l}\times [-T_c,0]\to \text{\uj V}_{c,l}$$  such that
 $\|\Psi_{c,l}-f\|_{C^0(\mathrm{U}_{\kappa,l}\times [-T_c,0])}\leq \delta/2$, where
$\text{\uj V}_{c,l}\subsetneq\T\times(0,1)\times\T$ and $\text{\uj V}_{c,l}\cap(\rN_{\kappa/2,l}\times\T)=\emptyset$,  $\lambda_c=\lambda_c(\kappa,L_0)\ll 1$.
It means  that $\Psi_{c,l}(x,\cdot)$ remains $\delta/2-$close to $\xi^-_{x,c}(\cdot)$ in the following sense:  $$\textup{dist}(\Psi_{c,l}(x,t), \xi^-_{x,c}(t))\leq\delta/2, \quad \forall~(x,t)\in\mathrm{U}_{\kappa,l}\times[-T_c, 0].$$
Recall that the number $\vep_1$ given in Lemma \ref{minimal set on cylinder} is small enough, then one can find an small interval $I_c=\{(c'_1,0): c'_1\in (c_1-\tau, c_1+\tau) \}$  depending on $\kappa, L_0$, such that if the perturbation term $\|L_1\|_{C^2}<2\vep_1$, then the Lagrangian $L=L_0+L_1$ satisfies: for each $c'\in I_c$, $x\in\rU_{\kappa,l}$,

$\bullet$ the $u^-_{c',l,L}$-calibrated  curve $\gamma^-_{x,c',L}(t):(-\infty,0]\to\check{M}$ with $\gamma^-_{x,c',L}(0)=x$  is negatively asymptotic to $\cA_{L,l}(c',\check{M})$.

$\bullet$ $\gamma^-_{x,c',L}(\cdot)$ is still $\delta-$close to $\Psi_{c,l}(x,\cdot)$ in the sense that
\begin{equation}\label{tubular approximation}
 \textup{dist}(\Psi_{c,l}(x,t),\gamma^-_{x,c',L}(t) )\leq \delta,\quad \forall -T_c\leq t\leq 0.
\end{equation}
These properties are guaranteed by the upper semi-continuity. By the finite covering theorem, there exist finitely many intervals $\{I_{c^i}\}_{i=0}^m$ such that
\begin{equation}\label{interval decomp}
\bigcup_{0\leq i\leq m} I_{c^i}\supset [-R+1,R-1]\times\{\bfo\},
\end{equation}
and the corresponding diffeomorphism $\Psi_{c^i,l}: \mathrm{U}_{\kappa,l}\times[-T_{c^i}, 0]\to \text{\uj V}_{c^i,l}$ is Gevrey-$(\alpha, \lambda_{c^i})$, and the positive number $T_{c^i}\ll 1$. According to Theorem \ref{inverse thm}, some constant  $\lambda'_{c^i}<\lambda_{c^i}$ exists such that $\Psi_{c^i,l}^{-1}$ is Gevrey-$(\alpha, \lambda'_{c^i})$ smooth.

In what follows, we set 
\begin{equation}\label{bold_L_0}
	\bL_0:=\min\{\lambda'_{c^i}: i=0,\cdots,m\},
\end{equation}
 and hence
$$\Psi_{c^i,l}^{-1}: \text{\uj V}_{c^i,l}\to\mathrm{U}_{\kappa,l}\times[-T_{c^i}, 0]$$
 is Gevrey-$(\alpha, \bL)$  smooth, for all $\bL\leq\bL_0$. We also point out that $\bL_0$ is independent of the perturbation Hamiltonian $H_1$, and it depends only on $H_0$, $R$ and $\alpha$ since the choice of $\kappa$ depends only on $H_0$ and $R$, and $L_0$ depends only on $H_0$. 

 Similarly, these procedures can be carried out for the region $\mathrm{U}_{\kappa,u}$, and one can get the corresponding Gevrey diffeomorphism $\Psi_{c^i,u}:\mathrm{U}_{\kappa,u}\times[-T_{c^i},0]\to\text{\uj V}_{c^i,u}$. For simplicity, we still assume the same interval covering $\bigcup_{i=0}^m I_{c^i}$ as \eqref{interval decomp} and  $\Psi_{c^i,u}^{-1}:$$\text{\uj V}_{c^i,u}\to\mathrm{U}_{\kappa,u}\times[ -T_{c^i},0]$ is Gevrey-$(\alpha, \bL)$ smooth for all $\bL\leq\bL_0$, where each $\Psi_{c^i,u}$ ($i=0,\cdots,m$) possesses the property analogous to \eqref{tubular approximation}.

\subsection{Total disconnectedness}\label{total discon}
Let $\alpha>1$, we will study the topological structure of the set of minimal points for 
\[B_{c,l,u}(x,\tau)=u^{-}_{c,l}(x,\tau)-u^{+}_{c,u}(x,\tau),\quad\text{and} \quad B_{c,u,l}(x,\tau)=u^{-}_{c,u}(x,\tau)-u^{+}_{c,l}(x,\tau)\] defined in \eqref{anotherkind barr}, where $u^{\pm}_{c,\imath}$  ($\imath=l, u$) are the elementary weak KAM solutions. Actually,
we will show that the minimal set is totally disconnected for generic Lagrangian systems.
Inspired by the technique in \cite[Section 4.2]{Ch2017}, we will perturb  directly  a Lagrangian by small potential functions. Compared with the perturbative techniques used in \cite{CY2004, CY2009} which  perturb the generating functions to create genericity,  our technique in the current paper provides more information, we can prove the genericity not only in the usual sense but also in the sense of Ma\~n\'e.

Let $L=L_0+L_1$ be our Lagrangian given in \eqref{lagrangian}, where $\|L_1\|_{C^2}<\vep_1$.
Recall the interval covering  $\bigcup_{0\leq i\leq m} I_{c^i}$ given in section \ref{determin of coeff}, one can always suppose that the length of each interval $I_{c^i}$ is less than 1. Then Theorem \ref{global regularity of elementary solutions} implies that for any $c, c'\in I_{c^i}\cap\mathbb{S}$ and  $q\in\mathrm{U}_{\kappa}=\mathrm{U}_{\kappa,l}\cup\mathrm{U}_{\kappa,u},$
\begin{equation}\label{simplicity of global regularity}
\begin{split}
 |u^{\pm}_{c,l}(q,0)-u^\pm_{c',l}(q,0)|\leq 2C_h\|c'-c\|^{\frac{1}{2}}, \qquad
 |u^{\pm}_{c,u}(q,0)-u^\pm_{c',u}(q,0)|\leq 2C_h\|c'-c\|^{\frac{1}{2}}.
\end{split}
\end{equation}

Fixing  $\bL\in(0, \bL_0]$ and $\vep_0\in (0,\vep_1)$, we consider the following set in  $\bG^{\alpha,\bL}(M\times\T)$ with $M=\T^2$:
\begin{equation}\label{space1}
\mathfrak{P}:=\{P\in\bG^{\alpha,\bL}(M\times\T): \|P\|_{\alpha,\bL}<\vep_0,~\text{supp}P\cap(\check{\pi}\rN_{\kappa/2}\times\T) =\emptyset\}.
\end{equation}
Then it is easily seen that each potential perturbation $P\in\mathfrak{P}$  to the Hamiltonian $H$ would not affect the NHIC since $\widetilde\Sigma_H\subset \check{\pi}\rN_{\kappa/2}\times\R^2\times\T$, see \eqref{c_t1}. We also point out that, by a natural extension, any function in $\bG^{\alpha,\bL}(M\times\T)$ can be viewed as a function defined on  $\check{M}\times\T$.  
\begin{The}\label{generic G2}
Let $\alpha>1$, $\bL\leq\bL_0$. There exists a residual set $\mathcal{W}$$\subset\mathfrak{P}$ such that for each Gevrey potential function $P\in\mathcal{W}$, the Lagrangian $L+P: T\check{M}\times\T\to\R$ satisfies: for each $c\in\mathbb{S}$, the sets
\[\arg\min B_{c,l,u}\big|_{\rU_{\kappa,l}\cup\rU_{\kappa,u}},\quad\arg\min B_{c,u,l}\big|_{\rU_{\kappa,l}\cup\rU_{\kappa,u}}\]
are both totally disconnected. Here, $\arg\min B_{c,l,u}|_{\rU_{\kappa,l}\cup\rU_{\kappa,u}}$ stands for $\arg\min B_{c,l,u}\bigcap(\rU_{\kappa,l}\cup\rU_{\kappa,u})$, and $\arg\min B_{c,u,l}|_{\rU_{\kappa,l}\cup\rU_{\kappa,u}}$ stands for $\arg\min B_{c,l,u}\bigcap(\rU_{\kappa,l}\cup\rU_{\kappa,u})$.
\end{The}
\begin{proof}	
For $c\in\mathbb{S}$, we first study the set $\arg\min B_{c,l,u}$ restricted on the region $\rU_{\kappa,l}\subset\check{M}\times\{t=0\}$. Let $\rro>0$ be the constant give in Lemma \ref{local manifolds representation}. Recall that we have  $\rro>\kappa$. Then,
to prove the total disconnectedness of  $\arg\min B_{c,l,u}\big|_{\rU_{\kappa,l}}$,  it is enough to verify that 
\begin{equation}\label{Tot_Disc1}
	\textsf{the  set~} \arg\min B_{c,l,u}\big|_{\bar{\rN}_{\rro,l}\setminus \rN_{\kappa,l}} \textsf{~is  totally  disconnected},
\end{equation}
where ${\bar{\rN}_{\rro,l}\setminus \rN_{\kappa,l}}$ $\subset$ $\rU_{\kappa,l}$ and $\bar{\rN}_{\rro,l}$ is the closure of  $\rN_{\rro,l}$. To explain this, we recall that, according to Proposition  \ref{double description},  a point $(x, 0)$ $\in$ $\arg\min B_{c,l,u}|_{\rU_{\kappa,l}}$ if and only if there exists a $c$-semi static curve $\gamma_{x, c}: \R\to\check{M}$ with $\gamma_{x, c}(0)=x$,
such that the orbit $\{\gamma_{x, c}(n): n\in\Z\}$ 
is negatively asymptotic to $\pi_q\circ \Upsilon_{c,l}$ and positively asymptotic to $\pi_q\circ \Upsilon_{c,u}$. Then, by letting $\kappa$ be suitably small if necessary, the orbit $\{\gamma_{x, c}(n): n\in\Z\}$ has to pass through the region ${\bar{\rN}_{\rro,l}\setminus \rN_{\kappa,l}}$ when it approaches $\pi_q\circ \Upsilon_{c,l}$, as a result of the normal hyperbolicity.  Consequently, in what follows, we only need to check  \eqref{Tot_Disc1}.

 We first focus on the subinterval $I_{c^0}$ in the interval covering  $\bigcup_{0\leq i\leq m} I_{c^i}$.
Let us pick a $2$-dim disk
$$D=\{(x_1, x_2, t)\in\check{M}\times\T:~t=0, ~ |x_1-x_{1,0}|\leq d, ~ |x_2-x_{2,0}|\leq d \}\subset\bar{\rN}_{\rro,l}\setminus \rN_{\kappa,l} $$ which  is centered at the point $(x_{1,0}, x_{2,0})$, and $d$ is small. We also set
$$D+d_1:=\{(x_1, x_2,t)\in\check{M}\times\T:~t=0, ~|x_1-x_{1,0}|\leq d+d_1, ~ |x_2-x_{2,0}|\leq d+d_1\}\subset \bar{\rN}_{\rro,l}\setminus \rN_{\kappa,l} $$
with $0<d_1\ll 1$ (see figure \ref{picture2}).
\begin{figure}[H]
  \centering
  \includegraphics[width=7.5cm]{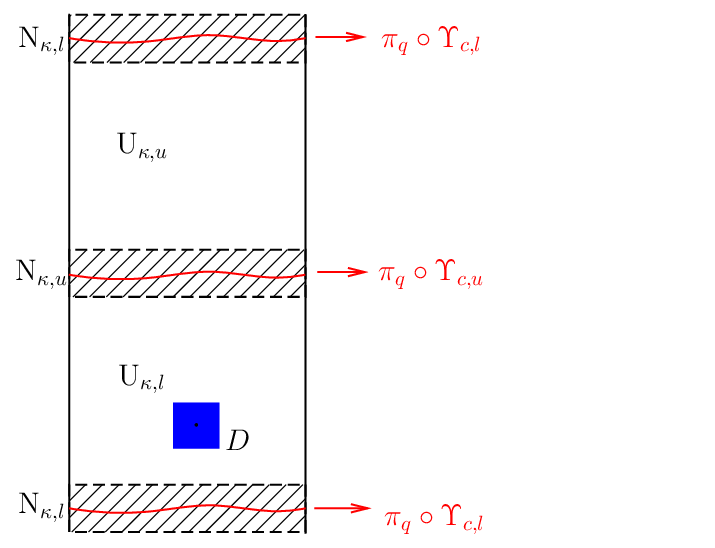}\\
  \caption{A fundamental domain of $\check{M}\times\{t=0\}$}\label{picture2}
\end{figure}
Let $\mu$ be suitably  small, for the index $i=1$ or $2$, we consider the following space
\begin{align}\label{Per_spaces}
\mathfrak{V}_{i}:=\bigg\{\mu\Big(\sum\limits_{\ell=1,2}a_{i,\ell}\cos2\ell\pi(x_i-x_{i,0}) +b_{i,\ell}\sin2\ell\pi(x_i-x_{i,0})\Big)~:~a_{i,\ell}, b_{i,\ell}\in[1,2]\bigg\}.
\end{align}
Obviously, $\mathfrak{V}_1, \mathfrak{V}_2\subset C^\omega(M)$. Next, we will construct perturbations  based on potential functions of the form  in $\mathfrak{V}_{i}$. To this end, we use the notation given in section \ref{determin of coeff}. Fixing a sufficiently large constant $\mathfrak{L}\gg\bL$, by Lemma \ref{Gevrey bumpfunction} one can construct a function $\rho(x,t)=g(x)\chi(t):\check{M}\times\T\to\R$ such that  $\chi:\T\to\R$ and $g(x):\check{M}\to\R$  are both  non-negative  Gevrey-$(\alpha, \mathfrak{L})$ functions. We choose
\begin{equation*}
\chi(t)=\left\{
  \begin{array}{ll}
   >0, &~t\in (-T_{c^0}, 0)\\
   =0, &~t\in\T\setminus(-T_{c^0}, 0)
  \end{array}
\right.
\end{equation*}
where  $ T_{c^0}\ll 1$ is given in section \ref{determin of coeff}, and require that  $g|_D\equiv1$ and supp$g\subset D+d_1$$\subset\bar{\rN}_{\rro,l}\setminus \rN_{\kappa,l} $. We set
$$\text{\uj C}:=\{ \Psi_{c^0,l}(x,t)~|~(x,t)\in (D+d_1)\times[-T_{c^0},0] \},$$
then $\text{\uj C}\subset\text{\uj V}_{c^0,l}\subsetneq \T\times(0, 1)\times\T$, and  therefore $\text{\uj C}\cap(\rN_{\kappa/2,l}\times\T)=\emptyset$.

$\bullet$ With each $V\in\mathfrak{V}_{1}$ or $\mathfrak{V}_{2}$, which can also be viewed as a function  on  $\check{M}$, one can define  $\widetilde{V}\in C^\infty(\check{M}\times\T)$ as follows: on the ``lower" domain $\T\times[0, 1]\times\T\subset\check{M}\times\T$,
\begin{equation*}
  \widetilde{V}(z)=\begin{cases}
  (\rho V)\circ\Psi_{c^0,l}^{-1}(z)=\rho(x,t)V(x), & ~\Psi_{c^0,l}(x,t)=z\in\text{\uj C},\\
 0, & z\in(\T\times[0, 1]\times\T)\setminus\text{\uj C}.
  \end{cases}
\end{equation*}
Then we extend symmetrically the function  to the ``upper" domain $\T\times[1, 2]\times\T$ such that $$\widetilde{V}(y,t)=\widetilde{V}(y-\mathbf{e}_2,t)$$
with $\mathbf{e}_2=(0, 1).$ The support of $\widetilde{V}$ satisfies
\begin{equation}\label{supp_2copies}
	\textup{supp}\widetilde{V}\subset \text{\uj C}\cup(\text{\uj C}+\mathbf{e}_2).
\end{equation}
Since  $\mathfrak{L}\gg\bL$, according to the properties (G\ref{algebra norm}), (G\ref{composition}) in Section \ref{introduction}, we have
\begin{equation}\label{construction of V}
\widetilde{V}\in\bG^{\alpha, \bL}(\check{M}\times\T).
\end{equation}

$\bullet$ We also remark that, by the symmetry of $\widetilde{V}\in\bG^{\alpha,\bL}(\check{M}\times\T)$, $\widetilde{V}$ can also be viewed as a function on $M\times\T$. By abuse of notation, we continue to write $\widetilde{V}\in\bG^{\alpha,\bL}(M\times\T)$. Thus, $\widetilde{V}\in \mathfrak{P}.$ As a result of the construction above, some constant $C_1>0$ exists such that
\begin{equation}\label{integration of V}
 \int_{-T_{c^0}}^{0}\widetilde{V}(\Psi_{c^0,l}(x,t))dt=V(x)\int^0_{-T_{c^0}}g(x)\chi(t) dt =V(x)\int^0_{-T_{c^0}}\chi(t) dt=C_1 V(x),\quad \text{for~} x\in D.
\end{equation}
Here, we have used the fact $g|_D\equiv1$.

Let $\Pi_i$,  $i=1,2$,  be the standard projection to the $i$-th coordinate of $\check{M}$. For the Lagrangian $L:T\check{M}\times\T\to\R$, we consider   two elementary weak KAM solutions $u^-_{c,l}(q,t)$, $u^+_{c,u}(q,t)$, and denote by $u^-_{c,l,\widetilde{V}}(q,t)$, $u^+_{c,u,\widetilde V}(q,t)$ the elementary weak KAM solutions of the perturbed Lagrangian  $ L+\widetilde V$. Then the following result holds:

\begin{Lem}\label{diameter compare}
There exists an open and dense set $\mathcal{U}_{D}\subset\mathfrak{P}$ (see \eqref{space1}) such that for each $\widetilde{V}\in\mathcal{U}_{D}$,
\begin{equation}\label{banjingxiao}
\Pi_i\left(\arg\min\big(u^-_{c,l,\widetilde V}(x,0)-u^+_{c,u,\widetilde V}(x,0)\big)\big|_{D}\right)\subsetneq [x_{i,0}-d,x_{i,0}+d],\quad \text{for all~} c\in I_{c^0}\cap\mathbb{S},
\end{equation}
where $ i=1, 2$.
\end{Lem}
\begin{proof}
We start with the perturbation $\widetilde{V}$ of the form \eqref{construction of V}, where $V\in\mathfrak{V}_{1}\cup \mathfrak{V}_{2}$.  Note that under such a potential perturbation, the cylinders $\Sigma_{H,l}(0)$ and $\Sigma_{H,u}(0)$ remain unchanged, and hence the Aubry set $\widetilde{\cA}_{L+\widetilde{V}}(c,\check{M})$ $=$ $\widetilde{\cA}_{L}(c,\check{M})$.

\textbf{Step 1:} For $c\in I_{c^0}\cap\mathbb{S}$, the projected Aubry set $\cA_{L}(c, \check{M})\subset \check{M}\times\T$ has two copies, denoted by $\cA_{L,l}(c, \check{M})$ and $\cA_{L, u}(c, \check{M})$.  Each set $\cA_{L,\imath}(c, \check{M})$, $\imath=l, u$ is diffeomorphic to $\T^2$ since $c\in \mathbb{S}$.
 For each $x\in D$ and each $u^+_{c,u}$-calibrated curve $\gamma^+_{x,c}(t):[0,+\infty)\to\check{M}$ with $\gamma^+_{x,c}(0)=x$,   the minimizing curve $(\gamma^+_{x,c}(t), t) :\R^+\longrightarrow\check{M}\times\T$ is positively asymptotic  to $\cA_{L, u}(c, \check{M})$. Now,  we claim that 
\begin{equation}\label{supp_nonintersection}
	\textup{supp}\widetilde{V}~\bigcap~ \big(\bigcup_{t>0}(\gamma^+_{x,c}(t),t)\big)=\emptyset,
\end{equation}
as long as $D$ is small enough.
 In fact, according to \eqref{supp_2copies}, 
 the support of $\tilde{V}$ has two  copies in the lower and upper region respectively and 
$\textup{supp}\widetilde{V}\subset \text{\uj C}\cup(\text{\uj C}+\mathbf{e}_2).
$
It is clear that the minimizing orbit $(\gamma^+_{x,c}(t),t)$ never intersects itself, then 
 $ \text{\uj C}\bigcap \big(\bigcup_{t>0}(\gamma^+_{x,c}(t),t)\big)=\emptyset$
 since $D$ is a small neighborhood of $x$.  Moreover,  observe that the sets $\cA_{L,l}(c, \check{M})$ and $\cA_{L, u}(c, \check{M})$, which are  diffeomorphic to $\T^2$, divide the $3$-dimensional configuration space $\check{M}\times\T$ into two connected components, then the minimizing curve 
$(\gamma^+_{x,c}(t), t) :\R^+\longrightarrow\check{M}\times\T$ always stays in the lower region, which means $ (\text{\uj C}+\mathbf{e}_2)$ $\bigcap$ $ \big(\bigcup_{t>0}(\gamma^+_{x,c}(t),t)\big)$ $=\emptyset$. This proves our claim \eqref{supp_nonintersection}. Consequently,  \begin{equation}\label{u_plus_equal}
	u^+_{c,u,\widetilde V}(x)=u^+_{c,u}(x),\quad \text{for all}~x\in D.
\end{equation}

But the function $u^-_{c,l,\widetilde V}$ would undergo a small perturbation. Indeed, for $x\in D$, we can take a $u^-_{c,l,\widetilde{V}}$-calibrated curve $\gamma^-_{x,c,\widetilde{V}}:(-\infty,0]\to\check{M}$ with $\gamma^-_{x,c,\widetilde{V}}(0)=x$, then for $ m\in\Z^+$,
\begin{equation}\label{dafd1}
	u^-_{c,l,\widetilde{V}}(\gamma^-_{x,c,\widetilde{V}}(0),0)-u^-_{c,l,\widetilde{V}}(\gamma^-_{x,c,\widetilde{V}}(-m),-m)= \int_{-m}^0 (L-\eta_c+\widetilde{V})(d\gamma^-_{x,c,\widetilde{V}}(t),t)+\alpha(c)\,dt,
	\end{equation}
For another perturbation $\widetilde{V}'$, we have
\begin{equation}\label{dafd2}
u^-_{c,l,\widetilde V'}(\gamma^-_{x,c,\widetilde{V}}(0),0)-u^-_{c,l,\widetilde V'}(\gamma^-_{x,c,\widetilde{V}}(-m),-m)\leq \int_{-m}^0(L-\eta_c+\widetilde V')(d\gamma^-_{x,c,\widetilde{V}}(t),t)+\alpha(c)\,dt.
\end{equation}
By normal hyperbolicity, there exists a uniform upper bound $T\in\Z^+$, $T>T_{c^0}$, such that the orbit $\{(\gamma^-_{c,l,\widetilde{V}}(-t), t)\}_{t\geq T}$ shall retreat into the small neighborhood $\rN_{\kappa/2,l}\times\T$.  As the supports of $\widetilde{V}$ and  $\widetilde{V}'$ have empty intersection with $\rN_{\kappa/2,l}\times\T$, we have $u^-_{c,l,\widetilde{V}}= u^-_{c,l,\widetilde V'}$ on $\rN_{\kappa/2,l}\times\T$. Then \eqref{dafd1}-\eqref{dafd2} imply that
$$u^-_{c,l,\widetilde V'}(x,0)-u^-_{c,l,\widetilde{V}}(x,0)\leq \int_{-T}^0 (\widetilde V'-\widetilde{V})(\gamma^-_{x,c,\widetilde{V}}(t),t)\,dt.$$
Conversely, we can prove similarly that
$$u^-_{c,l,\widetilde V'}(x,0)-u^-_{c,l,\widetilde{V}}(x,0)\geq \int_{-T}^0 (\widetilde V'-\widetilde{V})(\gamma^-_{x,c,\widetilde{V}'}(t),t)\,dt,$$
where $\gamma^-_{x,c,\widetilde V'}$ denotes the backward $u^-_{c,l,\widetilde V'}$-calibrated curve with $\gamma^-_{x,c,\widetilde V'}(0)=x$. Since $x$ lies in the region $D$ where $u^-_{c,l,\widetilde V}$ is differentiable (see Lemma \ref{local manifolds representation}), one has $\|\gamma^-_{x,c,\widetilde V'}(t)-\gamma^-_{x,c,\widetilde V}(t)\|\to 0$ as $\|\widetilde{V}' -\widetilde V\|\to 0$, which is guaranteed by the upper semi-continuity. Therefore, for $c\in I_{c^0}\cap\mathbb{S}$,
\begin{equation}\label{potential perturbation}
u^-_{c,l,\widetilde V'}(x,0)-u^-_{c,l,\widetilde V}(x,0)=\sK_{\widetilde{V},c}(\widetilde V'-\widetilde{V})(x)+\sR_c(\widetilde V'-\widetilde{V})(x),\quad x\in D,
\end{equation}
where the operator 
\begin{equation}\label{linear operator}
\sK_{\widetilde{V},c}(\widetilde V'-\widetilde{V})(x)=\int_{-T}^0(\widetilde V'-\widetilde{V})(\gamma_{x,c,\widetilde{V}}^-(t), t)\,dt,
\end{equation}
and the remainder  $$\sR_c(\widetilde V'-\widetilde V)=o(\| V'-V\|_{C^0})$$ since $V, V'\in\mathfrak{V}_{1}\cup \mathfrak{V}_{2}$ are linear combinations of trigonometric functions.

\textbf{Step 2:} Now, we claim that there exists an arbitrarily small perturbation $\widetilde{V}\in\mathfrak{P}$ of the form \eqref{construction of V}, such that
\begin{equation}\label{touying1}
\Pi_1\left(\arg\min\big(u^-_{c,l,\widetilde V}(x,0)-u^+_{c,u,\widetilde V}(x,0)\big)\big|_{D}\right)\subsetneq [x_{1,0}-d,x_{1,0}+d],\quad \text{for all~} c\in I_{c^0}\cap\mathbb{S}.
\end{equation}

To prove this claim, we construct a grid for the parameters $(a_{1,1},b_{1,1}, a_{1,2}, b_{1, 2})$ in $\mathfrak{V}_1$ by splitting the domain $[1,2]^4$ equally into a family of  4-dimensional cubes whose side length is  $\mu^2$, namely
$$\Delta a_{1,\ell}=\Delta b_{1,\ell}=\mu^{2},\quad \ell=1,2.$$
There are  as many as $[\mu^{-8}]$ cubes.

 In the sequel, we use the symbol $\textrm{Osc}_{x\in D}f$ to denote the oscillation of $f$, which describes the difference  between the supremum and infimum of $f$ on  $D$. 
 According to  \eqref{tubular approximation}, for each $c\in I_{c^0}\cap\mathbb{S}$ and $x\in D$, the backward $c$-semi static curve $\gamma^-_{x,c,\widetilde{V}}(t)$ will stay in the $\delta$-neighborhood of the curve $\Psi_{c^0,l}(t)$ for $t\in[-T_{c^0},0]$ provided that $\mu$ is small enough. Besides, since minimizing orbits have no  self intersections, by letting $D$ be suitably small if necessary,  the minimizing curve $\big(\gamma^-_{x,c,\widetilde{V}}(t), t\big):$ $[-T, -T_{c^0})$ $\longrightarrow$ $\check{M}\times\T$ does not intersect the support of $\widetilde{V}'-\widetilde{V}$. This is guaranteed by using arguments similar to the proof of \eqref{supp_nonintersection}. 
 Therefore,  \eqref{integration of V} and the above estimate imply that 
\begin{align}\label{approxi_del}
	\sK_{\widetilde{V},c}(\widetilde{V}'-\widetilde{V})(x)
=\int_{-T_{c^0}}^0   (\widetilde{V}'-\widetilde{V})(\Psi_{c^0,l}(x,t))\,dt+O(\mu\delta)= C_1 (V'(x)-V(x))+O(\mu\delta).
\end{align}
Then some constant $C_2>0$ exists such that
\begin{align}\label{zhendang1}
\text{Osc}_{x\in D}(\sK_{\widetilde{V}, c}(\widetilde V'-\widetilde V))>\frac{1}{4}C_1\text{Osc}_{x\in D}(V-V')>C_2\mu\Delta
\end{align}
where $\Delta=\max\{|a_{1,\ell}-a'_{1,\ell}|,|b_{1,\ell}-b'_{1,\ell}|:~\ell=1,2\}$. This is guaranteed by \eqref{approxi_del} and the fact that $V$ is a finitely linear combination of $\{\sin2\pi x_1,\cos2\pi x_1, \sin 4\pi x_1,\cos 4\pi x_1\}$.

Next, we split the interval $I_{c^0}$ equally into $[K_s\mu^{-6}]$ subintervals, where $K_s=L_s(\frac{24C_h}{C_2})^2$ and $L_s$ is the length of $I_{c^0}$, $C_h$ is the constant given in \eqref{simplicity of global regularity} and $C_2$ is  given in \eqref{zhendang1}. We pick up the subinterval that has non-empty intersection with $\mathbb{S}$, and then denote all such kinds of subintervals by $\{\text{\uj J}_i\}_{i\in\mathbb{J}}$. Clearly, the cardinality of $\mathbb{J}$ is less than $[K_s\mu^{-6}]$.

Let us fix a $c^*\in\text{\uj J}_i\cap\mathbb{S}$. If for some parameter $(a^*_{1,\ell}, b^*_{1,\ell}), \ell=1,2$  and its corresponding perturbation $V^*\in\mathfrak{V}_1$, formula \eqref{touying1} does not hold, then $\min_{x_2}\big( u^-_{c^*,l,\widetilde V^*}(x,0)-u^+_{c^*,u,\widetilde V^*}(x,0)\big)$ is identically equal to a constant for all $x\in D$, and hence
\begin{equation}\label{zhendang2}
\textup{Osc}_{x\in D}\min\limits_{x_2}\big( u^-_{c^*,l,\widetilde V^*}(x,0)-u^+_{c^*,u,\widetilde V^*}(x,0)\big)=0.
\end{equation}
Next, for another $V'=\mu\Big(\sum_{\ell=1,2}a'_{1,\ell}\cos2\ell\pi(x_1-x_{1,0}) +b'_{1,\ell}\sin2\ell\pi(x_1-x_{1,0})\Big)\in\mathfrak{V}_1$
and the corresponding perturbation $\widetilde {V}'$, it follows from \eqref{u_plus_equal} and \eqref{potential perturbation} that for all $c\in \text{\uj J}_i\cap\mathbb{S}$ and $x\in D$,
\begin{equation}
\begin{split}
& \big(u^-_{c,l,\widetilde V'}(x,0)-u^+_{c,u,\widetilde V'}(x,0)\big)-\big( u^-_{c^*,l, \widetilde V^*}(x,0)-u^+_{c^*,u,\widetilde V^*}(x,0)\big)\\
=&\big(u^-_{c,l,\widetilde V'}(x,0)-u^-_{c^*,l,\widetilde V'}(x,0)\big)-\big(u^+_{c,u,\widetilde V'}(x,0)-u^+_{c^*,u,\widetilde V'}(x,0)\big) +\big(\sK_{\widetilde{V}^*, c^*}+\sR_{c^*}\big)(\widetilde V'-\widetilde V^*)(x).
\end{split}
\end{equation}
As the length of  $\text{\uj J}_i$ is $\frac{L_s}{[K_s\mu^-6]}$ and $c,c^*\in\text{\uj J}_i\cap\mathbb{S}$, formula  \eqref{simplicity of global regularity} then implies  that
\begin{equation}
\begin{split}
&\Big|\big(u^-_{c,l,\widetilde V'}(x,0)-u^-_{c^*,l,\widetilde V'}(x,0)\big)-\big(u^+_{c,u,\widetilde V'}(x,0)-u^+_{c^*,u,\widetilde V'}(x,0)\Big)\big|\\
\leq & 4C_h\|c-c^*\|^{\frac{1}{2}}
\leq 4C_h\Big(\frac{L_s}{[K_s\mu^{-6}]}\Big)^{\frac{1}{2}}\leq\frac{C_2\mu^3}{6}.
\end{split}
\end{equation}
Since $\mu\ll 1$,  one has $\|\widetilde V'-\widetilde V^*\|\ll 1$ and
\begin{equation}\label{frac13}
\|\sR_{c^*}(\widetilde V'-\widetilde V^*)\|\leq\frac{1}{6}\|\sK_{\widetilde{V}^*,c^*}(\widetilde V'-\widetilde V^*)\|.
\end{equation}
Regarding the potential function $V'$, if its parameter $(a'_{1,\ell}, b'_{1,\ell}), \ell=1, 2$  satisfies
\begin{equation}\label{relation of coefficients}
\max\{|a^*_{1,\ell}-a'_{1,\ell}|, |b^*_{1,\ell}-a'_{1,\ell}| : \ell=1,2\}\geq \mu^2,
\end{equation}
then inequalities \eqref{zhendang1}--\eqref{frac13} together give rise to
$$\textup{Osc}_{x\in D}\min\limits_{x_2}\big(u^-_{c,l,\widetilde V'}(x,0)-u^+_{c,u,\widetilde V'}(x,0)\big)\geq\frac{C_2}{3}\mu^3>0.$$
Thus we can conclude that for each $c\in\text{\uj J}_i\cap\mathbb{S}$ and $V'\in\mathfrak{V}_1$ satisfying \eqref{relation of coefficients}, we have
\begin{equation}\label{reduce}
\begin{split}
\textup{Osc}_{x\in D}\min\limits_{x_2}\big(u^-_{c,l,\widetilde V'}(x,0)-u^+_{c,u,\widetilde V'}(x,0)\big)>0.
\end{split}
\end{equation}
Consequently, for each  $\text{\uj J}_i$, we only need to cancel out at most $2^4$ cubes from the grid  $\{\Delta a_{1,\ell}, \Delta b_{1,\ell}: \ell=1, 2\}$ so that formula \eqref{reduce} holds for all other cubes. Let the index $i$ range over $\mathbb{J}$, we therefore obtain a set  $\text{\uj P}_1\subseteq\{(a_{1,1}, a_{1,2}, b_{1,1}, b_{1,2}): a_{1,\ell}, b_{1,\ell}\in[1, 2], \ell=1, 2\}$ with Lebesgue measure
$$\textup{meas}\text{\uj P}_1\geq1-2^4(\mu^2)^4|\mathbb{J}|\geq1-2^4K_s\mu^2>0,$$
such that formula \eqref{reduce} holds for any $V'$ with parameter in $\text{\uj P}_1$ and any $c\in I_{c^0}\cap\mathbb{S}$.  As $\mu$ is small enough,  the claim \eqref{touying1} is now evident from what we have proved.

\textbf{Step 3:} 
Actually, the arguments above in Step 2 also ensure  that  formula \eqref{touying1} has density in $\mathfrak{P}$. The openness is obvious, so there is an open-dense set $\mathcal{U}_{D,1}$ in $\mathfrak{P}$ such that formula \eqref{touying1} holds for each perturbed Lagrangian $L+\tilde{V}$ with  $\widetilde{V}\in\mathcal{U}_{D,1}$.

Analogously,  we can consider a potential function $V\in\mathfrak{V}_2$ and its associated perturbation $\widetilde{V}$. By repeating similar arguments as in Step 2, we  also obtain an open-dense set $\mathcal{U}_{D,2}\subset\mathfrak{P}$, such that for each  perturbed Lagrangian $L+\tilde{V}$ where $\widetilde{V}\in\mathcal{U}_{D,2}$,
\begin{equation*}
\Pi_2\left(\arg\min\big(u^-_{c,l,\widetilde V}(x,0)-u^+_{c,u,\widetilde V}(x,0)\big)\big|_{D}\right)\subsetneq [x_{2,0}-d,x_{2,0}+d],\quad \text{for all~} c\in I_{c^0}\cap\mathbb{S}.
\end{equation*}
Thus, the proof of Lemma \ref{diameter compare} is now completed by taking a set $\mathcal{U}_D=\mathcal{U}_{D,1}\cap \mathcal{U}_{D,2}$.
\end{proof}

Now we continue to prove Theorem \ref{generic G2}.

$\bullet$ From Lemma \ref{diameter compare} we see that for each small disk $D\subseteq \bar{\rN}_{\rro,l}\setminus \rN_{\kappa,l}$, there exists an open-dense set $\mathcal{U}_D\subset\mathfrak{P}$  such that formula \eqref{banjingxiao} holds for each Lagrangian $L+\widetilde V$  with $\widetilde V\in\mathcal{U}_D$. Next, we take a countable topology basis $\bigcup_j D_j$ for $\bar{\rN}_{\rro,l}\setminus \rN_{\kappa,l}$ where the diameter of $D_j$ approaches to $0$ as $j\to\infty$, and therefore obtain an open-dense set $\mathcal{U}_{D_j}$ for each $j$. Clearly, $\mathcal{U}_{I_{c^0}}=\bigcap_j\mathcal{U}_{D_j}$ is a residual set in $\mathfrak{P}$, and  the set $\arg\min\big(u^-_{c,l,P}(x,0)-u^+_{c,u,P}(x,0)\big)\big|_{\bar{\rN}_{\rro,l}\setminus \rN_{\kappa,l}}$
is totally disconnected for each $P\in\mathcal{U}_{I_{c^0}}$ and $c\in\mathbb{S}\cap I_{c^0}$.

The technique above also works for other subintervals $I_{c^i}, i=1,\cdots,m$, we can then obtain the corresponding residual sets $\mathcal{U}_{I_{c^i}}, i=1,\cdots,m$.  Hence the intersection  $\mathcal{U}_l=\bigcap_{i=0}^m\mathcal{U}_{I_{c^i}}$ is residual, and 
$\arg\min\big(u^-_{c,l,P}(x,0)-u^+_{c,u,P}(x,0)\big)\big|_{\bar{\rN}_{\rro,l}\setminus \rN_{\kappa,l}}$
is totally disconnected for each $P\in\mathcal{U}_l$  and $c\in\mathbb{S}$. By what have shown at the beginning of the proof, this is equivalent to saying that $\arg\min\big(u^-_{c,l,P}(x,0)-u^+_{c,u,P}(x,0)\big)\big|_{\rU_{\kappa, l}}$
is totally disconnected for each $P\in\mathcal{U}_l$  and $c\in\mathbb{S}$.

$\bullet$ Similarly, one can prove that there exists a residual set $\mathcal{U}_u\subset\mathfrak{P}$, such that  the set $$\arg\min\big(u^-_{c,l,P}(x,0)-u^+_{c,u,P}(x,0)\big)\big|_{\rU_{\kappa,u}}$$
is totally disconnected for each $P\in\mathcal{U}_u$ and $c\in\mathbb{S}$.

$\bullet$ Conversely, by applying the  technique  above to $u^-_{c,u,P}(x,0)-u^+_{c,l,P}(x,0)$, we can also obtain two residual sets $\mathcal{V}_l$ and $\mathcal{V}_u$ in $\mathfrak{P}$, such that the set
$\arg\min\big(u^-_{c,u,P}(x,0)-u^+_{c,l,P}(x,0)\big)\big|_{\rU_{\kappa,l}}$ is totally disconnected for each $c\in\mathbb{S}$ and $P\in\mathcal{V}_l$, and the set $\arg\min\big(u^-_{c,u,P}(x,0)-u^+_{c,l,P}(x,0)\big)\big|_{\rU_{\kappa,u}}$ is totally disconnected for each $ c\in\mathbb{S}$ and $P\in\mathcal{V}_u$.

Therefore,  the proof of Theorem \ref{generic G2} is now completed by taking $\mathcal{W}=\mathcal{U}_l\cap\mathcal{U}_u\cap\mathcal{V}_l\cap\mathcal{V}_u$. 
\end{proof}

\subsection{Proof of Theorem \ref{main theorem}  and \ref{main thm2}}
Now, we are able to prove our main results. Let $R>1$, $\alpha>1$ and $0<\bL\leq\bL_0=\bL_0(H_0,\alpha,R)$, where the constant $\bL_0$ is given in \eqref{bold_L_0} and independent of $H_1$.
\begin{proof}
In our problem $M=\T^2$, $s>0$, $y_\ell\in[-R+1,R-1]\times\{\bfo\}$, $\ell\in\{1,\cdots,k\}$. Let $\vep_0$ be a small positive number satisfying 
$$\vep_0<\min\{1,\bL^{\alpha},\frac{\bL^{2\alpha}}{2!^\alpha},\frac{\bL^{3\alpha}}{3!^\alpha}\}\vep_1,$$
 where $\vep_1$ is chosen as in Lemma \ref{minimal set on cylinder}. Now,  $\|H_1\|_{\alpha,\bL}<\vep_0$  implies  $\|H_1\|_{C^3}<\vep_1$, and hence the  Hamiltonian $H=H_0+H_1$ has a persistent normally hyperbolic invariant cylinder (NHIC), and the globally minimal set $\widetilde{\cG}_L(c)$ lies in this NHIC  for each $c=(c_1,\bfo)$ with $|c_1|\leq R-1$. 
Here,  $L=L_0+L_1$ is the Lagrangian associated to $H=H_0+H_1$.
 Also, observe that for the Ma\~n\'e set $\widetilde{\cN}_{L_0}(c)$, 
  \begin{equation}\label{proj_N_L0}
 	\pi_p\circ\sL^{-1}\widetilde{\cN}_{L_0}(c)=c, \quad\text{~for each~} c=(c_1,\bfo), |c_1|\leq R-1.
 \end{equation}
 By letting $\vep_0$ be small enough,  the perturbation term $L_1$ is also sufficiently small, then the upper semi-continuity (see Proposition \ref{upper semi}) implies that for each $c=(c_1,\bfo)$, $|c_1|\leq R-1$,
\begin{equation}\label{dist_N_L_and_N_L0}
 d(\pi_p\circ\sL^{-1}\widetilde{\cN}_{L}(c),\pi_p\circ\sL^{-1}\widetilde{\cN}_{L_0}(c))< s/2, \quad\text{whenever~} \|H_1\|_{\alpha,\bL}<\vep_0
\end{equation}
 With this fact, one can find that 
$\vep_0=\vep_0(H_0,\alpha,R,s,\bL)>0$
 depends on the Hamiltonian $H_0$ and the constants $\alpha$, $R$, $\bL$ and $s$.

\textbf{Density:} For each Hamiltonian $H_0+H_1$ with $\|H_1\|_{\alpha,\bL}<\vep_0,$ we will prove that there exists an arbitrarily small perturbation $V\in\bG^{\alpha,\bL}(M\times\T)$ such that $\|H_1+V\|_{\alpha,\bL}<\vep_0$, and the perturbed Hamiltonian $H_0+H_1+V$  has an orbit $(q(t),p(t))$ and times $t_1<\cdots<t_k$ such that the action variables $p(t)$ pass through the ball $B_s(y_\ell)$ at the time $t=t_\ell$. To this end, we will  establish  a generalized transition chain along which one is able to apply Theorem \ref{generalized transition thm}.

 Let $d\in(0, \vep_0-\|H_1\|_{\alpha,\bL})$ be arbitrarily small.

$\bullet$ First, by applying the genericity property in Corollary \ref{corgeneric G1} to the Tonelli Lagrangian $L_0+L_1$, one can always choose a small perturbation 
\[\phi\in\bG^{\alpha,\bL}(M\times\T), \quad \|\phi\|_{\alpha,\bL}<\frac{d}{2}\]
such that for each rational homology class $h=(\frac{p}{q},\bfo)$, the perturbed  Lagrangian  $L_0+L_1+\phi$  has only one minimal measure with the rotation vector $h$.  Next, for the irrational case, it is well known in Aubry-Mather theory that for homology $h=(h_1,\bfo)$ with $h_1\in\R\setminus\Q$, only one minimal measure with the rotation vector $h$ exists.   Thus it is easily seen that
 the Aubry class is unique for each $c=(c_1,0)$ with $|c_1|\leq R-1$, as a result of property \eqref{ssff}. Then the uniqueness of Aubry class implies 
\begin{equation}\label{faf}
	\widetilde{\cA}_{L_0+L_1+\phi}(c)=\widetilde{\cN}_{L_0+L_1+\phi}(c).
\end{equation}
By the Legendre transformation, the associated Hamiltonian is exactly $H_0+H_1-\phi$ with $\|H_1-\phi\|_{\alpha,\bL}<\vep_0$, so  \eqref{dist_N_L_and_N_L0} and \eqref{faf} imply that
 \begin{equation}\label{proj_A_L0}
 	\pi_p\circ\sL^{-1}\widetilde{\cA}_{L_0}(c)=c, \quad  d\left(\pi_p\circ\sL^{-1}\widetilde{\cA}_{L_0+L_1+\phi}(c),~\pi_p\circ\sL^{-1}\widetilde{\cA}_{L_0}(c)\right)< s/2, 
 \end{equation}
for each $c=(c_1,\bfo)$, $|c_1|\leq R-1$.

Recall that the set $\widetilde{\cN}_{L_0+L_1+\phi}(c)\big|_{t=0}$ lies in the NHIC, so it is either homologically trivial or not. In the homologically trivial case,  it is well known that the $c$-equivalence holds inside $(c_1-\delta_c,c_1+\delta_c)\times\{\bfo\}$ for some small  $\delta_c>0$, which satisfies condition (1) in Definition \ref{transition chain}. 

In the latter case, $\widetilde{\cN}_{L_0+L_1+\phi}(c)\big|_{t=0}$ must be an invariant curve as a result of \eqref{faf}. Then we define as \eqref{buianquandeshangtongdiao} the set$$\mathbb{S}:=\{ (c_1,\bfo) ~:~|c_1|\leq R-1, \Upsilon_c\textup{~is an invariant circle on the NHIC}\}.$$
Applying Theorem \ref{generic G2} to $L_0+L_1+\phi$, we can find a small potential perturbation with compact support
\[P\in\bG^{\alpha,\bL}(M\times\T),\quad  \|P\|_{\alpha,\bL}<\frac{d}{2},\]
 such that the Lagrangian $L_0+L_1+\phi+P: T\check{M}\times\T\to\R$, defined in the double covering space, satisfies: for all  $c\in\mathbb{S}$, the sets
$$\arg\min B_{c,l,u}\big|_{\rU_{\kappa,l}\cup\rU_{\kappa,u}},\quad\arg\min B_{c,u,l}\big|_{\rU_{\kappa,l}\cup\rU_{\kappa,u}}$$
are both totally disconnected. Then Proposition \ref{manejifenlei} and  \ref{double description} together yield that for each $c\in\mathbb{S}$, there exists a small $\delta_c>0$ such that the set
$$\check{\pi}\cN(c,\check{M})\big|_{t=0}\setminus(\cA(c,M)\big|_{t=0}+\delta_c)$$
is totally disconnected, which satisfies condition (2) in Definition \ref{transition chain}.  Consequently, the corresponding Hamiltonian is exactly $H_0+H_1-\phi-P$ where $\|\phi+P\|_{\alpha,\bL}<d$ and $\|H_1-\phi-P\|_{\alpha,\bL}<\vep_0$.

$\bullet$ Next, we take \[V=-\phi-P.\]
Then the arguments above imply that there exists a generalized transition chain inside $[-R+1,R-1]\times\{\bfo\}$$\subset H^1(M,\R)$ for the Lagrangian $L_0+L_1-V$.  Thus, we conclude from  Theorem \ref{generalized transition thm} and \eqref{proj_A_L0} that the perturbed Hamiltonian $H_0+H_1+V$ has an orbit $(q(t),p(t))$  whose action variables $p(t)$ pass through the ball $B_s(y_\ell)$ at the time $t=t_\ell$, where $t_1<t_2<\cdots<t_k$.
Finally,  thanks to $\|V\|_{\alpha,\bL}<d$ and the arbitrariness of $d$, we complete the proof of density in $\fB^{\bL}_{\vep_0,R}$.

\textbf{Openness:} It only remains to verify the openness. Since the time for the aforementioned trajectory $(q(t),p(t))$ passing through the balls $B_s(y_1)$, $\cdots$, $B_s(y_k)$ is finite, the smooth dependence of solutions of ODEs on parameters can guarantee the openness in $\fB^{\bL}_{\vep_0,R}$.   Theorem \ref{main theorem} is now evident from what we have proved.

Notice that in the proof of density part above, the perturbation  we constructed is Gevrey potential function. Combining with the obvious openness property, Theorem \ref{main thm2} is also true.
\end{proof}

\appendix
\section{Normally hyperbolic theory}\label{appendix_NHIC}
In this appendix, we  review some classical results in the theory  of normally hyperbolic
 manifolds. We only give a less general introduction which is better applied to our problem, and refer the reader to \cite{Fen1971,Fen1977,HPS1977,Pesin04} for the proofs and more detailed introductions.
\begin{Def}\label{Def NHIM}
Let $M$ be a smooth Riemannian manifold and $f: M\to M$  be a $C^r (r\geq 1)$ diffeomorphism. Let $N\subset M$ be a submanifold (probably with boundary) which is invariant under $f$. Then $N$ is called a normally hyperbolic invariant manifold (NHIM) if there is an $f$-invariant tangent bundle
splitting such that, for every $x\in N$
$$T_xM=T_xN\oplus E_x^s\oplus E_x^u,$$
and there exist a constant $C>0$, rates $0<\lambda<1<\mu$ with $\lambda\mu<1$ such that
\begin{equation}\label{hyp splitting}
\begin{split}
    v\in T_xN & \Longleftrightarrow \|Df^k(x)v\|\leq C\mu^{|k|}\|v\| ,  \quad k\in\Z,\\
    v\in E^s_x  & \Longleftrightarrow \|Df^k(x)v\|\leq C\lambda^k\|v\|,  \quad k\geq0,\\
    v\in E^u_x & \Longleftrightarrow \|Df^k(x)v\|\leq C\lambda^{|k|}\|v\|,  \quad k\leq0.\\
\end{split}
\end{equation}
\end{Def}

\medskip

In what follows,  $N$ is assumed to be compact and connected. 
Let $U$ be a tubular neighborhood of the NHIM $N$. In both \cite{Fen1971}  and \cite{HPS1977} the existence of local stable and unstable manifolds
in $U$, denoted by  $W_N^{s,loc}$ and $W_N^{u,loc}$ respectively, are obtained by using the method of Hadamards graph transform. Moreover, the local stable and unstable manifolds can be characterized  as follows:
\begin{equation}\label{definition of local stable}
  \begin{split}
     W_N^{s,loc} &=\{y\in U ~|~ \textup{dist}(f^k(y),N)\leq \tilde{C}_{y}(\lambda+\tilde{\vep})^k, \textup{~for all~} k\geq0 \}, \\
     W_N^{u,loc} &=\{y\in U ~|~  \textup{dist}(f^k(y),N)\leq  \tilde{C}_{y}(\lambda+\tilde{\vep})^{|k|}, \textup{~for all~} k\leq0 \},
   \end{split}
\end{equation}
where the constant $\tilde{C}_{y}>0$,  and $\tilde{\vep}>0$ is a small constant satisfying $\lambda+\tilde{\vep}<1/\mu$.
For each $x\in N$, the corresponding local stable and unstable leaves are defined as follows:
\begin{equation}\label{definition of local stable leaf}
  \begin{split}
     W_x^{s,loc} &=\{y\in U ~|~  \textup{dist}(f^k(x),f^k(y))\leq  \tilde{C}_{x,y}(\lambda+\tilde{\vep})^k, \textup{~for all~} k\geq0 \}, \\
     W_x^{u,loc} &=\{y\in U ~|~  \textup{dist}(f^k(x),f^k(y))\leq  \tilde{C}_{x,y}(\lambda+\tilde{\vep})^{|k|}, \textup{~for all~} k\leq0 \}.
   \end{split}
\end{equation}
where $\tilde{C}_{x,y}>0$ is a constant.
Then we have the following properties (see \cite{HPS1977}):
\begin{The}\label{property of NHIM}
Let $N$ be a  NHIM given in \eqref{Def NHIM}, and we define an integer $l:=\max\{k\in\Z~|~ 1\leq k\leq r$ $\text{~and~} k<\frac{|\log\lambda|}{\log\mu}\}$, then
\begin{enumerate}[\rm(1)]
\item  $N$, $W_N^{s,loc}$ and $W_N^{u,loc}$ are $C^{l}$ manifolds. For each $x\in N$, the manifolds $W_x^{s,loc}$ and $W_x^{u,loc}$ are $C^{r}$ and $T_xW_x^{s,loc}=E_x^s$, $T_xW_x^{u,loc}=E_x^u$.
\item  $W_N^{s,loc}$ and $W_N^{u,loc}$ are foliated by the stable and unstable leaves respectively, i.e. 
\[W_N^{s,loc}=\bigcup_{x\in N} W_x^{s,loc}, \quad W_N^{u,loc}=\bigcup_{x\in N} W_x^{u,loc}.\]
 Moreover, if $x\neq x'$, then $W_x^{s,loc}\bigcap W_{x'}^{s,loc}=\emptyset$ and $W_x^{u,loc}\bigcap W_{x'}^{u,loc}=\emptyset.$
  \item The unstable foliation $\{W^{u, loc}_x : x\in N\}$ is $C^l$ in the sense that $\bigcup_{x\in N}T^k_x W^{u, loc}_x$ is a continuous bundle for each $1\leq k\leq l$, where $T^k$ denotes the $k$-th order tangent. Analogous result holds for the stable foliation.
\end{enumerate}
\end{The}

\begin{Rem}
\rm{(1)}:  We point out that for the models studied in the current paper, we have smoothness as high as the time-$1$ map since the dynamics on $N$ is close to integrable.\\
\rm{(2)}: We can also define the global stable (unstable) sets $W_N^{s,u}$ and  $W_x^{s,u}$, just by replacing $U$ with $M$ in \eqref{definition of local stable}, \eqref{definition of local stable leaf}. But $W_N^{s,u}, W_x^{s,u}$ may fail to be embedded manifolds.
\end{Rem}

The normal hyperbolicity has stability under perturbations. Roughly speaking, the normally hyperbolic invariant manifold  persists under small perturbations.
\begin{The}[Persistence of normally hyperbolic invariant manifolds]\label{persistence}
Suppose that $N\subset M$ is a NHIM for the $C^r$ $(r\geq 1)$ diffeomorphism $f$ and $\vep>0$ is sufficiently small. Then for any $C^r$  diffeomorphism $f_\vep: M\to M$ satisfying $\|f_\vep-f\|_{C^1}<\vep$, there exists a NHIM $N_\vep$ that is $C^l$ diffeomorphic and close to $N$
 where $l=\max\{k : 1\leq k\leq r$ $\text{~and~} k<\frac{|\log\lambda|}{\log\mu}\}$. Moreover, the local stable manifold $W^{s,loc}_{N_\vep}$ and local unstable manifold $W^{u,loc}_{N_\vep}$ are $C^l$ close to those of $N$.
\end{The}

\section{Variational construction of global connecting orbits}\label{sec_proof_of_connectingthm}
The goal of this section is to prove Theorem \ref{generalized transition thm}, which could be achieved by modifying the arguments and techniques in \cite{CY2009}. We also refer the reader to  \cite{Ch2018} or \cite{Ch2012} for more details. Throughout this section, we assume $M=\T^n$. Our diffusing orbits are constructed by shadowing a sequence of local connecting orbits, along each of them the Lagrangian action attains a ``local minimum". 
\subsection{Local connecting orbits}\label{sec localconnect}
Let $d\gamma(t)=(\gamma(t), \dot{\gamma}(t))$.
An orbit $(d\gamma(t),t):\R\to TM\times\T$ is said to connect one Aubry set $\widetilde{\cA}(c)$ to another one $\widetilde{\cA}(c')$ if the $\alpha$-limit set of the orbit  is contained in $\widetilde{\cA}(c)$ and the $\omega$-limit set is contained in $\widetilde{\cA}(c')$. 
We will introduce two types of local connecting orbits: type-$c$ and type-$h$, the former corresponds to Mather's  cohomology equivalence, while the latter corresponds to the variational interpretation of Arnold's mechanism. Before that, we need some preparations. 
\subsubsection{Time-step Lagrangian and upper semi-continuity}
Both types of local connecting orbits depend on the upper semi-continuity of minimal curves of a modified $C^r$ Lagrangian $L^*:T\T^n\times\R\rightarrow\R$ which is defined as follows: let $L^+,L^-$ be two time-1 periodic Tonelli Lagrangians,
\begin{equation*}
  L^*(\cdot,t):=\begin{cases}
  L^-(\cdot,t), & ~t\in(-\infty,0]\\
  L^+(\cdot,t), & ~t\in[1,+\infty),
  \end{cases}
\end{equation*}
and $L^*$ is superlinear and positive definite in the fibers.
Notice that $L^*$ is not periodic in time $t$, instead, it is periodic when restricted on either $(-\infty,0]$ or  $[1,+\infty)$. We call such a modified Lagrangian $L^*$ a \emph{time-step} Lagrangian.

For a \emph{time-step} Lagrangian $L^*$, a curve $\gamma:\R\rightarrow\T^n$ is called minimal if for any $t<t^\prime\in\R$,
\begin{equation*}
  \int_t^{t^\prime}L^*(\gamma(s),\dot{\gamma}(s),s)\,ds=\min\limits_{\substack{\zeta(t)=\gamma(t),\zeta(t^\prime)=\gamma(t^\prime)\\ \zeta\in C^{ac}([t,t^\prime],\T^n)}}\int_t^{t^\prime}L^*(\zeta(s),\dot{\zeta}(s),s)\,ds
\end{equation*}
Hence, we denote by $\sG(L^*)$ the set of all minimal curves and $\widetilde{\sG}(L^*)=\bigcup_{\gamma\in\sG(L^*)}(\gamma(t),\dot{\gamma}(t),t)$. 

Let $\alpha^\pm$ denote Mather's minimal average action of $L^\pm$. For $m_0, m_1\in \T^n$ and $T_0,T_1\in\Z_+$, we define
\begin{equation*}
 h^{T_0,T_1}_{L^*}(m_0,m_1):=\inf\limits_{\substack{\gamma(-T_0)=m_0,\gamma(T_1)=m_1\\ \gamma\in C^{ac}([-T_0,T_1],\T^n)}}\int_{-T_0}^{T_1}L^*(\gamma(t),\dot{\gamma}(t),t)\,dt+T_0\alpha^-+T_1\alpha^+,
\end{equation*}
and
\begin{equation*}
h_{L^*}^{\infty}(m_0,m_1):=\liminf_{T_0,T_1\to+\infty}h_{L^*}^{T_0,T_1}(m_0,m_1),
\end{equation*}
which are bounded. We take any two sequences of positive integers $\{T_0^i\}_{i\in\Z_+}$ and $\{T_1^i\}_{i\in\Z_+}$ with $T_\ell^i\rightarrow+\infty$ ($\ell=0,1$) as $i\to+\infty$ and the associated minimal curve $\gamma_i(t)$: $[-T^i_0,T^i_1]\to \T^n$  connecting $m_0$ to $m_1$ such that
$$
h_{L^*}^{\infty}(m_0,m_1)=\lim_{i\to\infty}h_{L^*}^{T_0^i,T_1^i}(m_0,m_1)=\lim_{i\to\infty}\int_{-T_0^i}^{T_1^i}L^*(\gamma_i(t),\dot{\gamma}_i(t),t)\,dt +T^i_0\alpha^-+T^i_1\alpha^+.
$$

The following lemma shows that any accumulation point $\gamma$ of $\{\gamma_i\}_i$ is a pseudo curve playing an analogous role as a semi-static curve. For the proof, see \cite{CY2004} or \cite{CY2009}.
\begin{Lem}\label{pseudo curve}
Let $\gamma$: $\R\to \T^n$ be an accumulation point of $\{\gamma_i\}_i$ as shown above. Then for any $s\geq 0$, $t\geq 1$,
\begin{equation}\label{pseudo curve formula}
\begin{split}
\int_{-s}^{t} L^*(\gamma(\tau),\dot{\gamma}(\tau),\tau)\,d\tau+s\alpha^-+t\alpha^+=\inf\limits_{\substack{\xi(-s_1)=\gamma(-s)\\\xi(t_1) =\gamma(t) \\  s_1-s\in\Z, ~t_1-t\in\Z\\ s_1\geq 0,~t_1\geq 1}}\int_{-s_1}^{t_1} L^*(\xi(\tau),\dot{\xi}(\tau),\tau)\,d\tau+s_1\alpha^-+t_1\alpha^+,
\end{split}
\end{equation}
where the minimum is taken over all absolutely continuous curves.
\end{Lem}
This leads us to define the set of \emph{pseudo connecting curves}
\begin{equation*}
  \sC(L^*):=\{\gamma |~\gamma\in\sG(L^*) \text{~and~} \eqref{pseudo curve formula}\text{~holds}\}.
\end{equation*}
Clearly, for each $\gamma\in\sC(L^*)$ the orbit $(\gamma(t),\dot\gamma(t),t)$ would negatively approach the Aubry set $\widetilde{\cA}(L^-)$ of the Lagrangian $L^-$ and positively approach the Aubry set $\widetilde{\cA}(L^+)$ of  $L^+$. This is why we call it a pseudo connecting curve. Define the following sets
$$
\widetilde{\cC}(L^*):=\bigcup_{\gamma\in\sC(L^*)}(\gamma(t),\dot\gamma(t),t),\qquad \cC(L^*):=\bigcup_{\gamma\in\sC(L^*)}(\gamma(t),t).
$$
Notice that if $L^*$ is time-1 periodic, then $\widetilde{\cC}(L^*)$ is exactly the Ma\~n\'e set and $\cC(L^*)$ is exactly the projected Ma\~n\'e set. Then we can prove the following property:

\begin{Pro}\label{uppersemi of pseudo curves}
The set-valued map $L^*\mapsto\sC(L^*)$ is upper semi-continuous, namely if $L^*_i\to L^*$ in the $C^2$ topology, then we have the set inclusion
$$ \limsup_i \sC(L^*_i)\subset \sC(L^*).$$ Consequently, the map $L^*\mapsto\widetilde{\cC}(L^*)$ is also upper semi-continuous.
\end{Pro}
\begin{proof}
Let  $L^*_i\to L^*$  in the $C^2$ topology. If $\gamma_i$ converges $C^0$-uniformly  to a curve $\gamma$ on each compact interval of $\R$ with $\gamma_i\in\sC(L^*_i)$. We claim that $\gamma\in\sC(L^*)$. 

Indeed, there exists $K>0$ such that $\|\dot{\gamma}_i(t)\|\leq K$ for all $t\in\R$,  so the set  $\{\gamma_i\}_{i}$  is compact in the $C^1$ topology. Since each $\gamma_i$ satisfies the Euler-Lagrange equation of $L_i$, by using the positive definiteness of $L^*_i$, one can write the Euler-Lagrange equation in the form of $\ddot{x}=f_i(x,\dot{x},t)$ for some $f_i$, which implies  that $\{\gamma_i\}_{i}$  is compact in the $C^2$ topology. By the Arzel\`{a}-Ascoli theorem, extracting a subsequence if necessary, we can assume that $\gamma_i$ converges $C^1$-uniformly to a $C^1$ curve $\eta$ on each compact interval of $\R$. Obviously, $\eta=\gamma.$

Next, if $\gamma\notin\sC(L^*)$, there would be some $s\geq 0, t\geq 1$, a curve $\widetilde{\gamma}:[-s-n_1, t+n_2]\to M$ and $\delta>0$ such that the action
$$\int_{-s-n_1}^{t+n_2}L^*(\widetilde{\gamma}(\tau),\dot{\widetilde{\gamma}}(\tau),\tau)\,d\tau+n_1\alpha^-+n_2\alpha^+\leq \int_{-s}^{t}L^*(\gamma(\tau),\dot{\gamma}(\tau),\tau)\,d\tau-\delta$$
where $s, s+n_1\geq 0$, $t, t+n_2\geq 1$ and $\widetilde{\gamma}(-s-n_1)=\gamma(-s)$, $\widetilde{\gamma}(t+n_2)=\gamma(t)$. Since we have shown that $\gamma$ is an accumulation point of $\gamma_i$ in the $C^1$ topology, for any small $\vep>0$, there would be a sufficiently large $i$ such that $\|\gamma-\gamma_i\|_{C^1[s,t]}\leq\vep$ and a curve $\widetilde{\gamma}_i: [-s-n_1, t+n_2]\to M$ with $\widetilde{\gamma}_i(-s-n_1)=\gamma_i(-s)$, $\widetilde{\gamma}_i(t+n_2)=\gamma_i(t)$ such that
$$\int_{-s-n_1}^{t+n_2}L_i^*(\widetilde{\gamma}_i(\tau),\dot{\widetilde{\gamma}}_i(\tau),\tau)\,d\tau+n_1\alpha^-+n_2\alpha^+\leq \int_{-s}^{t}L_i^*(\gamma_i(\tau),\dot{\gamma}_i(\tau),\tau)\,d\tau-\frac{\delta}{2},$$
By \eqref{pseudo curve formula}, $\gamma_i\notin\sC(L^*_i)$, which is a contradiction. This proves  $\gamma\in\sC(L^*)$. 

Finally,  the upper semi-continuity of  $L^*\mapsto\widetilde{\cC}(L^*)$ is a  consequence of what we have shown   above.
\end{proof}

\subsubsection{Local connecting orbits of type-{\it c}}\label{sub localc}
 In condition of the cohomology equivalence (see definition \ref{def_c_equivalenve}), we will show how to construct local connecting orbits  based on  Mather's variational mechanism. This idea of construction was first proposed by J. N. Mather in \cite{Ma1993}.

\begin{The}\label{clemma connect}
Let $L:T\T^n\times\T\to\R$ be a Tonelli Lagrangian and $c, c'\in H^1(\T^n,\R)$ are cohomology equivalent through a path $\Gamma:[0,1]\to H^1(\T^n,\R)$. Then there would exist $c=c_0, c_1, \dots, c_k=c'$ on the path $\Gamma$, closed 1-forms $\eta_i$ and $\bar{\mu}_i$ on $M$ with $[\eta_i]=c_i$, $[\bar{\mu}_i]=c_{i+1}-c_i$, and a smooth function $\rho_i(t):\R\to[0,1]$ for each $i=1,\cdots,k$, such that the time-step Lagrangian 
\begin{equation*}
	L_{\eta_i,\mu_i}=L-\eta_i-\mu_i,\quad\text{with}\quad \mu_i=\rho_i(t)\bar{\mu}_i
\end{equation*}
possesses the following properties:
 
 For each curve $\gamma\in \sC(L_{\eta_i,\mu_i})$, it determines a trajectory $(d\gamma(t),t)$, connecting $\widetilde{\cA}(c_i)$ to $\widetilde{\cA}(c_{i+1})$,  of the Euler-Lagrange flow $\phi^t_L$.
\end{The}
\begin{proof}
By definition \ref{def_c_equivalenve}, it is obvious that there exist $c=c_0, c_1, \dots, c_k=c'$ on the path $\Gamma$, closed 1-forms $\eta_i$ and $\bar{\mu}_i$ on $M$ with $[\eta_i]=c_i$, $[\bar{\mu}_i]=c_{i+1}-c_i$$\in \mathbb{V}^{\bot}_{c_i}$ for each $i=1,\cdots,k$. By the arguments in Section \ref{sec local and global}, there is also a neighborhood $U_i$ of the projected Ma\~n\'e set $\cN_0(c_i)$ such that $\mathbb{V}_{c_i}=i_{*U_i}H_1(U_i,\R)$.
 
In particular, we can suppose $\bar{\mu}_i=0$ on $U_i$. Indeed, as $[\bar{\mu}_i]\in \mathbb{V}^{\bot}_{c_i}$,  $\bar{\mu}_i$ is exact when restricted on $U_i$ and there is  a smooth function $f:M\to\R$ satisfying $df=\bar{\mu}_i$ on $U_i$, and hence we can replace $\bar{\mu}_i$ by $\bar{\mu}_i-df$.

As $\cN_0(c_i)\subset U_i$, there exists $\delta_i\ll 1$ such that $\cN_t(c_i)\subset U_i$ for all $t\in[0, \delta_i]$. Let $\rho_i:\R\to [0,1]$ be a smooth function such that $\rho_i(t)=0$ for $t\in(-\infty, 0]$, $\rho_i(t)=1$ for $t\in[\delta_i,+\infty)$.  We set $\mu_i=\rho_i(t)\bar{\mu}_i$ and introduce a \emph{time-step} Lagrangian
$$L_{\eta_i,\mu_i}=L-\eta_i-\mu_i: T\T^n\times\R\to\R.$$
For each orbit $\gamma\in\sC(L_{\eta_i,\mu_i})$, by the upper semi-continuity in Proposition \ref{uppersemi of pseudo curves},
\begin{equation}\label{belong typec}
\gamma(t)\in U_i,~\forall~t\in[0,\delta_i]
\end{equation}
holds provided that $|\bar{\mu}_i|$ is small enough.

Clearly, $(\gamma,\dot{\gamma})$ solves the Euler-Lagrange equation of $L_{\eta_i,\mu_i}$. To verify it solves the  Euler-Lagrange equation of $L$, we see that $\gamma\big|_{[0,\delta_i]}\subset U_i$
and $L_{\eta_i,\mu_i}=L-\eta_i$ on $U_i$ where $\eta_i$ is a closed 1-form, so $\gamma(t)$ solves the Euler-Lagrange equation of $L$ for $t\in[0,\delta_i]$. On the other hand, for $t\in(-\infty, \delta_i]$ we have $L_{\eta_i,\mu_i}=L-\eta_i$, then $\gamma(t)$ is a $c_{i}$-semi static curve of $L$ on the interval $(-\infty, \delta_i]$. Similarly, $\gamma(t)$ is a $c_{i+1}$-semi static curve of $L$ for $t\in[\delta_i,+\infty)$. Thus, $(\gamma(t),\dot{\gamma}(t)): \R\to T\T^n$ solves the Euler-Lagrange equation of $L$, and by Section \ref{sec_Preliminaries}, this orbit connects $\widetilde{\cA}(c_i)$ to $\widetilde{\cA}(c_{i+1})$ .
\end{proof}

\subsubsection{Local connecting orbits of type-{\it h}}\label{sub localh}
Next, we will discuss the so-called local connecting orbits of type-$h$, it can be thought of as a variational version of Arnold's mechanism, the condition of geometric transversality is extended to the total disconnectedness of minimal points of barrier function. It is used to handle the situation where the cohomology equivalence does not always exist. Usually, it is applied to the case where the Aubry set lies in a neighborhood of a lower dimensional torus,  in that case, we let $\check{\pi}:\check{M}\rightarrow \T^n$ be a finite covering of $\T^n$. Denote by $\widetilde{\cN}(c,\check{M}), \widetilde{\cA}(c,\check{M})$ the  Ma\~{n}\'{e} set and Aubry set with respect to  $\check{M}$, then $\widetilde{\cA}(c,\check{M})$ would have more than one Aubry classes. In fact, for the construction of type-$h$ local connecting orbits in our proofs of Theorem \ref{main theorem} and \ref{main thm2}, it only involves two Aubry classes (see section \ref{sec proof main}).

Thus,  we only need to deal with the situation where the Tonelli Lagrangian $L: T\T^n\times\T\rightarrow \R$ contains more than one Aubry classes. Let  $\cA(c)\big|_{t=0} $ denote the time-0 section of the projected Aubry set $\cA(c)$, i.e. $\cA(c)\bigcap(\T^n\times\{t=0\})$, then we obtain the local connecting orbits of type-$h$ as follows:
\begin{The}\label{connecting type h}
Let the projected Aubry set $\cA(c)=\cA_{c,1}\cup\cdots \cup\cA_{c,k}$ consists of $k$  $(k\geq 2)$ Aubry classes.  Let  $U:=\T^n\setminus\big(\cA(c)|_{_{t=0}}+\kappa\big)$ be an open set where $\kappa>0$ is small. If $U\bigcap (\cN(c)|_{t=0})$ is non-empty and totally disconnected.
Then for any $c'$ sufficiently close to $c$, there exists an orbit of the Euler-Lagrange flow $\phi^t_L$ whose $\alpha$-limit set lies in $\widetilde{\cA}(c)$ and $\omega$-limit set lies in $\widetilde{\cA}(c')$.
\end{The}
\begin{proof}
As the number of Aubry class of $\cA(c)$ is finite, it is well known that  if $c'$ is sufficiently close to $c$, the projected Aubry set $\cA(c')$ will be contained in a small neighborhood of $\cA(c)$, see e.g. \cite{Be2010On}. 

Since each Aubry class is compact and disjoint with each other,  we have $\text{\rm dist}(\cA_{c,i},\cA_{c,i'})>0$ for any $i\neq i'$, and there exist open neighborhoods $N_1,\cdots,N_k \subset M$ such that $\cA_{c,i}\big|_{t=0}\subset N_i$ for each $1\leq i\leq k$ and $\textup{dist}(N_i,N_{i'})>0$ for $i\neq i'$. Thus,  $\cA(c')|_{t=0}\subset \bigcup_i N_i$. From Proposition \ref{manejifenlei} and definition \eqref{ij maneorbit} we know $\cN(c)=\bigcup_{i,i'}\cN_{i,i'}(c)$, and hence there is a pair $(j, j')$ such that $\cA(c')|_{t=0}\cap N_{j'}\neq\emptyset$ and $U\cap \cN_{j,j'}(c)\neq\emptyset$.  

By the total disconnectedness assumption, we can find  simply connected open sets $F$ and $O$ such that $F\subset O\subset U$, $\textup{dist}(O,\bigcup_{i=1}^k\bar{N}_i)>0$  and $\emptyset\neq O\bigcap\big(\cN_{j,j'}(c)\big|_{t=0}\big)\subset F$. Then some $\delta>0$ exists such that
\begin{equation}\label{belongrelation}
O\bigcap\big(\cN_{j,j'}(c)\big|_{0\leq t\leq\delta}\big)\subset F.
\end{equation}
Let $\eta$ and $\bar{\mu}$ be closed 1-forms such that $[\eta]=c$, $[\bar{\mu}]=c'-c$, and let $\rho:\R\to [0, 1]$ be a smooth function such that $\rho(t)=0$ for $t\leq 0$, $\rho=1$ for $t\geq\delta$. Note that by the simple connectedness of $O$, we are able to choose $\bar{\mu}$ such that $\textup{supp}\bar{\mu}\cap \bar{O}=\emptyset$.
Next, we can construct a smooth function $\psi(x,t)=\vep\psi_1(x)\psi_2(t):M\times\T\to [-1, 1]$ where $\vep>0$, such that 
\begin{equation*}
  \psi_1(x)\left\{
     \begin{array}{ll}
       =1, & x\in \bar{F}, \\
       <1, & x\in O\setminus F,\\
       <0, & x\in\bigcup\limits_{i\neq j,j'}N_i, \\
       =0, & \textup{elsewhere.}
     \end{array}
   \right.
\end{equation*}
and
\begin{equation*}
  \psi_2(t)\left\{
             \begin{array}{ll}
               >0, & t\in (0,\delta), \\
               =0, & t\in(-\infty,0]\cup[\delta,+\infty).
             \end{array}
           \right.
\end{equation*}
Then we set $\mu=\rho(t)\bar{\mu}$ and  introduce a time-step Lagrangian
$$L_{\eta,\mu,\psi}=L-\eta-\mu-\psi: T\T^n\times\R\to\R.$$

Let us first suppose $\mu=0$. Since $\psi(x,t)=0$ for $(x,t)\in\cA_{c,j}\cup\cA_{c,j'}$, and $\psi(x,t)<0$ for $(x,t)\in\bigcup\limits_{i\neq j,j'}N_i$ with $t\in(0,\delta)$, the Lagrangian $L_{\eta,0,\psi}$ contains only two Aubry classes which are exactly $\cA_{c,j}$ and $\cA_{c,j'}$ provided the  number $\vep>0$ is small enough.  The set $\sC(L_{\eta,0,\psi})$ then satisfies:
\begin{enumerate}[(a)]
  \item  $\cA_{c,j}\cup\cA_{c,j'}\subset\sC(L_{\eta,0,\psi})$.
  \item $\sC(L_{\eta,0,\psi})\setminus\big(\cA_{c,j}\cup\cA_{c,j'}\big)$ is non-empty. For each pseudo connecting curve  $\xi\in$ $\sC(L_{\eta,0,\psi})\setminus$ $\big(\cA_{c,j}\cup\cA_{c,j'}\big)$, we have $\xi(t)\in F$ for $0\leq t\leq\delta$, but its integer translation $K^*\xi(t):=\xi(t-K)$ with $K\in\Z\setminus{0}$ does not belong to $\sC(L_{\eta,0,\psi})$ since $L_{\eta,0,\psi}$ is not periodic in $t$.
  \item $\sC(L_{\eta,0,\psi})$ does not contain any other curves.
\end{enumerate}
These properties follow  directly from \eqref{belongrelation} and the fact that $\psi(x,0)$ attains its maximum if and only if $x\in\bar{F}$, and the upper semi-continuity of $(\eta,0,\mu)\mapsto \sC(L_{\eta,0,\psi})$.

If $\mu\neq 0$. For $m_0\in\cA_{c,j}\big|_{t=0}, m_1\in \cA_{c,j'}\big|_{t=0}$, let $T_0^k, T_1^k\to+\infty$ be two sequences of positive integers such that
$$\lim\limits_{k\to\infty}h_{L_{\eta,\mu,\psi}}^{T_0^k,T_1^k}(m_0,m_1)=h_{L_{\eta,\mu,\psi}}^{\infty}(m_0,m_1).$$
Let $\gamma_k(t): [-T^k_0, T^k_1]\to M$ be a minimizer associated with $h_{L_{\eta,\mu,\psi}}^{T_0^k,T_1^k}(m_0,m_1)$ and $\gamma$ be any accumulation point of $\{\gamma_k\}_k$, then $\gamma\in\sC(L_{\eta,\mu,\psi})$. If $\mu$ and $\vep$ are small enough, we deduce from the properties $(a)-(c)$ and the upper semi-continuity of $(\eta,\mu,\psi)\mapsto \sC(L_{\eta,\mu,\psi})$ that
\begin{equation}\label{constant region}
\gamma(t)\in F,\quad\forall t\in[0, \delta].
\end{equation}

Obviously, $(\gamma,\dot{\gamma})$ satisfies the Euler-Lagrangian equation of $L_{\eta,\mu,\psi}$, but we still need to verify that it solves the Euler-Lagrangian equation of $L$. In fact, $L_{\eta,\mu,\psi}=L-\eta$ for $t\leq 0$ and $L_{\eta,\mu,\psi}=L-\eta+\bar{\mu}$ for $t\geq\delta$ where $\eta$, $\bar{\mu}$ are  closed 1-forms, so $\gamma(t)$ solves the Euler-Lagrangian equation of $L$ for $t\in(-\infty, 0]\cup[\delta, +\infty)$.  This also implies that $\gamma:(-\infty,0]\to \T^n$ is a $c$-semi static curve of $L$ and $\gamma:[\delta,+\infty)\to M$ is a $c'$-semi static curve of $L$, then $$\alpha(d\gamma(t),t)\subset\widetilde{\cA}(c),\quad \omega(d\gamma(t),t)\subset\widetilde{\cA}(c').$$
Besides, for $t\in[0, \delta]$, we deduce from \eqref{constant region} that  the Euler-Lagrange equation $(\frac{d}{dt}\partial_v-\partial_x)L_{\eta,\mu,\psi}=0$ is equivalent to $(\frac{d}{dt}\partial_v-\partial_x)L=0$ along the curve $\gamma(t)$ within $0\leq t\leq\delta$, which therefore shows that $(\gamma,\dot{\gamma})$  solves the Euler-Lagrange equation of $L$ for $t\in[0,\delta]$. This completes our proof.
\end{proof}

From the proof of Theorem \ref{connecting type h} we see that the connecting orbit $(\gamma,\dot{\gamma})$ obtained in this theorem is locally minimal in the following sense:

\noindent{\bf Local minimum}: {\it 
There are two open balls $V^-,V^+\subset \T^n$ and  $k^-,k^+\in\Z+$ such that $\bar{V}^-\subset N_j\setminus \cA(c)\big|_{t=0}$ and $\bar{V}^+\subset N_{j'}\setminus \cA(c')\big|_{t=0}$,
$\gamma(-k^-)\in V^-$, $\gamma(k^+)\in V^+$ and
\begin{equation}\label{local minimal property}
\begin{aligned}
&h_c^{\infty}(x^-,m_0)+h_{L_{\eta,\mu,\psi}}^{k^-,k^+}(m_0,m_1)+h_{c'}^{\infty}(m_1,x^+)\\
>&\liminf_{k^-_i, k_i^+\to\infty}\int_{-k^-_i}^{k^+_i}
L_{\eta,\mu,\psi}(\gamma(t),\dot{\gamma}(t),t)\,dt+k^-_i\alpha(c)+k^+_i\alpha(c')
\end{aligned}
\end{equation}
holds for all $(m_0,m_1)\in \partial(V^-\times V^+)$, $x^-\in N_j\cap \alpha(\gamma)|_{t=0}$, $x^+\in N_{j'}\cap\omega(\gamma)|_{t=0}$, where
$k_i^-, k_i^+$ are the sequences such that $\gamma(-k_i^-)\to x^-$ and $\gamma(k_i^+)\to x^+$.}

The set of curves starting from $V_i^-$ and reaching $V_{i'}^+$ within time $k^-+k^+$ would make up a neighborhood of the curve $\gamma$ in the space of curves. If it touches the boundary of this neighborhood, the action of $L_{\eta,\mu,\psi}$ along a curve $\xi$ will be larger than the action along $\gamma$. Besides, the connecting orbit of type-$c$ also has local minimal property. In this case, the modified Lagrangian has the form $L_{\eta,\mu}$. The local minimality is crucial in the variational construction of global connecting orbits.

\subsection{Global connecting orbits}\label{sec variationconst}
Now, we are ready to prove Theorem \ref{generalized transition thm} from a variational viewpoint.  Intrinsically, we construct a global connecting orbit  by shadowing a sequence of local connecting  orbits.

\begin{proof}[Sketch of the proof of Theorem \ref{generalized transition thm}]
The proof parallels to that of \cite{CY2009} by a small modification. Here we only give a sketch of the basic idea, and the reader can refer to \cite[Section 5]{CY2009}, \cite{Ch2018}, \cite{Ch2012} for more details. For the generalized transition chain $\Gamma:[0,1]\to H^1(\T^n,\R)$ with $\Gamma(0)=c$ and $\Gamma(1)=c'$, by definition there exists a sequence $0=s_0<s_1<\cdots<s_m=1$ such that  $s_i$ is sufficiently close to $s_{i+1}$ for each $0\leq i\leq m-1$, and $\cA(\Gamma(s_i))$ could be connected to $\cA(\Gamma(s_{i+1}))$ by a local minimal orbit of either type-$c$ (as Theorem \ref{clemma connect}) or type-$h$ (as Theorem \ref{connecting type h}). Then the global connecting orbits are just constructed by shadowing these local ones. For simplicity, we set $c_i=\Gamma(s_i)$.

For each $i\in \{0,1,\cdots,m-1\}$,  we take $\eta_i,\mu_i$, $\psi_i$ and $\delta_i>0$ as that in the proof of Theorem \ref{clemma connect} and \ref{connecting type h}, where $\psi_i=0$ in the case of type-$c$. Then we choose $k_i\in\Z_+$ with $k_0=0$ and $k_{i+1}-k_i$ is suitably large for each $i\in \{0,1,\cdots,m-1\}$, and  introduce a modified Lagrangian
\begin{equation*}
 L^*:=L-\eta_0-\sum\limits_{i=0}^{m-1}k_i^*(\mu_i+\psi_i).
\end{equation*}
Here,  $k_i^*$ denotes a time translation operator such that $k^*_if(x,t)=f(x, t-k_i)$, and $\psi_i= 0$ in the case of type-$c$.  By this definition,  we see that $L^*=L-\eta_0$ for $t\leq k_0=0$, $L^*=L-\eta_m$ for $t\geq k_{m-1}+\delta_{m-1}$, and for each $i\in\{0, 1,\cdots,m-2\}$, $L^*=L-\eta_i-k_i^*(\mu_i+\psi_i)$ on $t\in [k_i, k_i+\delta_i]$ and $L^*=L-\eta_{i+1}$ for $t\in [k_i+\delta_i, k_{i+1}]$.

For integers $T_0, T_m\in\Z+$ and $x_0,x_m\in \T^n$, we define
\begin{equation*}
  h^{T_0,T_m}(x_0,x_m)=\inf_{\xi}\int_{-T_0}^{T_m+k_{m-1}}L^*(\xi(s),\dot{\xi}(s),s)\,ds+\sum\limits_{i=1}^{m-1}(k_i-k_{i-1})\alpha(c_i)+T_0\alpha(c_0)+T_m\alpha(c_m),
\end{equation*}
where the infimum is taken over all absolutely continuous curves $\xi$ defined on the interval $[-T_0,T_m+k_{m-1}]$ under some boundary conditions.
By carefully setting boundary conditions and using standard arguments in variational methods, one could find that the minimizer $\gamma(t; T_0, T_m,m_0,m_1)$ of the action $h^{T_0,T_m}(x_0,x_m)$ is smooth everywhere, along which the term $k_i^*(\mu_i+\psi_i)$ would not contribute to the Euler-Lagrange equation. Hence the minimizer produces an orbit of the flow $\phi^t_L$, which  passes through the $\vep$-neighborhood of $\widetilde{\cA}(c_i)$ at some time $t=t_i$. Let $T_0, T_m\to+\infty$, we can then get an accumulation curve $\gamma(t):\R\to \T^n$ of the sequence $\{\gamma(t; T_0, T_m,m_0,m_1)\}$ such that the $\alpha$-limit set of $(d\gamma(t),t)$ lies in $\widetilde{\cA}(c)$ and the $\omega$-limit set of $(d\gamma(t),t)$ lies in $\widetilde{\cA}(c')$. This completes the proof.

\end{proof}

\medskip
\noindent{\bf Acknowledgments}
We sincerely thank the anonymous referees for their
insightful comments and valuable suggestions on improving our results. The first author was partially supported by China Postdoctoral Science Foundation (Grant No.2018M641500). The second author was partially supported by National Natural Science Foundation of China (Grant No. 11631006, No. 11790272) and a program PAPD of Jiangsu Province, China.    
\def\cprime{$'$}

\end{document}